\RequirePackage{fix-cm}
\documentclass[smallextended]{svjour3}       % onecolumn (second format)
\smartqed  % flush right qed marks, e.g. at end of proof
\usepackage{graphicx}
\usepackage{multirow,bigstrut}
\usepackage{amsmath,amsfonts,amssymb}
\usepackage{mathrsfs}
\usepackage{color}
\usepackage{upgreek}
\usepackage{bm}
\usepackage{verbatim}
\usepackage{mathtools}
\begin{document}

\title{Mathematical and numerical analysis of the time-dependent Maxwell--Schr\"{o}dinger Equations in the Coulomb gauge\thanks{This work is supported by National Natural Science Foundation
of China (grant 11571353, 91330202),  and
Project supported by the Funds for Creative Research Group of China
(grant 11321061).}
}
% \subtitle{Do you have a subtitle?\\ If so, write it here}

\titlerunning{Mathematical and numerical analysis of M-S equations}        % if too long for running head

\author{ Chupeng Ma\textsuperscript{1}  \and Liqun Cao\textsuperscript{1}  \and
        Jizu Huang\textsuperscript{1} \and  Yanping Lin\textsuperscript{2}
}

\authorrunning{Chupeng Ma et al.} % if too long for running head

\institute{Chupeng Ma  \\
   \email{machupeng@lsec.cc.ac.cn}\\
   \\      
   Liqun Cao \\
  \email{clq@lsec.cc.ac.cn}      \\
  \\
  Jizu Huang     \\
 \email{huangjz@lsec.cc.ac.cn} \\
 \\
 Yanping Lin \\
 \email{malin@polyu.edu.hk}\\
 \\
  \textsuperscript{1} Institute of
      Computational Mathematics and Scientific/Engineering Computing,
      Academy of Mathematics and Systems Science, Chinese Academy of
      Sciences, Beijing 100190, China \\
 \\
 \textsuperscript{2} Department of Applied Mathematics,
The Hong Kong Polytechnic University, Kowloon, Hong Kong, China      
 }

\date{Received: date / Accepted: date}
% The correct dates will be entered by the editor

\maketitle

\begin{abstract}
In this paper, we consider the initial-boundary value problem for the time-dependent Maxwell--Schr\"{o}dinger equations in the Coulomb gauge. We first prove the global existence of weak solutions to the equations. Next we propose an energy-conserving fully discrete finite element scheme for the system and prove the existence and uniqueness of solutions to the discrete system. The optimal error estimates for the numerical scheme without any time-step restrictions are then derived. Numerical results are provided to support our theoretical analysis.
% \PACS{PACS code1 \and PACS code2 \and more}
% \subclass{MSC code1 \and MSC code2 \and more}
\end{abstract}

\section{Introduction}\label{sec-1}
In this paper, we consider one of the fundamental equations of nonrelativitic quantum mechanics, the Maxwell--Schr\"{o}dinger (M-S) system, which describes the time-evolution of an electron within its self-consistent generated and external electromagnetic fields. In this system, the Schr\"{o}dinger's equation can be written as follows:
\begin{equation}\label{eq:1.1}
\begin{array}{lll}
{\displaystyle  \mathrm{i}\hbar\frac{\partial \Psi}{\partial t}=
\Big\lbrace\frac{1}{2m}\left[\mathrm{i}\hbar\nabla +q\mathbf{A}\right]^{2}
 + q \phi+V\Big\rbrace\Psi \quad {\rm in} \,\,\, \Omega_{T}, }\\[2mm]
\end{array}
\end{equation}
where $\Omega_T = \Omega \times (0, T)$. Here $\Psi$,  $m$, and $q$ denote the wave function, the mass, and the charge of the electron, respectively. $V$ is the time-independent potential energy and is assumed to be bounded in this paper. The magnetic potential $\mathbf{A}$ and the electric potential $\phi$ are obtained by solving the following equations:
\begin{equation}\label{eq:1.2}
\mathbf{E} = -\nabla \phi - \frac{\partial \mathbf{A}}{\partial t}, \qquad \mathbf{B} = \nabla\times\mathbf{A},
\end{equation}
where the electric fields $\mathbf{E}$ and the magnetic fields $\mathbf{B}$ satisfy the Maxwell's equations:
\begin{equation}\label{eq:1.3}
\begin{array}{lll}
{\displaystyle \nabla \times \mathbf{E}+\frac{\partial \mathbf{B}}{\partial t}=0, \;\quad \quad \nabla\cdot\mathbf{B} = 0 , }\\[2mm]
{\displaystyle \frac{1}{{\mu}}\nabla \times \mathbf{B} - \epsilon\frac{\partial \mathbf{E}}{\partial t}=\mathbf{J}, \quad \nabla\cdot(\epsilon \mathbf{E}) = \rho.}
\end{array}
\end{equation}
Here $\epsilon$ and $\mu$ denote the electric permittivity and the magnetic permeability of the material, respectively. The charge density $\rho$ and the current density $\mathbf{J}$ are defined as follows
\begin{equation}\label{eq:1.4}
\rho = q|\Psi|^{2}, \qquad \mathbf{J} = -\frac{\mathrm{i}q\hbar}{2m}\big(\Psi^{\ast}\nabla{\Psi}-\Psi\nabla{\Psi}^{\ast}\big)-\frac{\vert q\vert^{2}}{m}\vert\Psi\vert^{2}\mathbf{A},
\end{equation}
 where $\Psi^{\ast} $ denotes the conjugate of $\Psi$.
 
 Substituting (\ref{eq:1.2}) into (\ref{eq:1.3}) and combining (\ref{eq:1.1}) and (\ref{eq:1.4}), we have the following M-S system 
 \begin{equation}\label{eq:1.5}
\left\{
\begin{array}{@{}l@{}}
{\displaystyle  \mathrm{i}\hbar\frac{\partial \Psi}{\partial t}=
\left\lbrace\frac{1}{2m}\left[\mathrm{i}\hbar\nabla +q\mathbf{A}\right]^{2}
 + q \phi+V \right\rbrace\Psi \,\quad {\rm in} \,\,\, \Omega_{T},}\\[2mm]
 {\displaystyle -\frac{\partial}{\partial t}\nabla\cdot\big(\epsilon\mathbf{A}\big)
 -\nabla\cdot\big(\epsilon\nabla\phi\big) =q |\Psi|^{2} \quad \quad \quad \quad {\rm in} \,\,\, \Omega_{T},} \\[2mm]
{\displaystyle \epsilon\frac{\partial ^{2}\mathbf{A}}{\partial t^{2}}+\nabla\times
\big({\mu}^{-1}\nabla\times \mathbf{A}\big)
+\epsilon \frac{\partial (\nabla \phi)}{\partial t} =\mathbf{J} \quad {\rm in} \,\,\, \Omega_{T}, }\\[2mm]
{\displaystyle \mathbf{J}=-\frac{\mathrm{i}q\hbar}{2m}\big(\Psi^{\ast}\nabla{\Psi}-\Psi\nabla{\Psi}^{\ast}\big)-\frac{\vert q\vert^{2}}{m}\vert\Psi\vert^{2}\mathbf{A} \,\,\quad {\rm in} \,\,\, \Omega_{T},} \\[2mm]
{\displaystyle \Psi, \phi, \mathbf{A} \,\, \mathrm{subject \ to \ the \ appropriate \ initial \ and\ boundary \ conditions}. }
\end{array}
\right.
\end{equation}
We assume that $\Omega$ is a bounded domain in $\mathbb{R}^d$ ($ d = 2,\,3$). The total energy of the system, at time $t$, are defined as follows
\begin{equation}\label{eq:1.6}
\begin{array}{@{}l@{}}
{\displaystyle \mathcal{E}(t) = \int_{\Omega}\big(\frac{1}{2}\big\vert \left(i\nabla + q\mathbf{A}\right)\Psi(t,\mathbf{x})\big\vert^{2} + V|\Psi(t,\mathbf{x})|^{2} }\\[2mm]
{\displaystyle \qquad \qquad + \frac{\epsilon}{2}|\mathbf{E}(t,\mathbf{x})|^{2} + \frac{1}{2\mu}|\mathbf{B}(t,\mathbf{x})|^{2} \big)\mathrm{d}\mathbf{x}.  }
\end{array}
\end{equation}
For a smooth solution $(\Psi, \mathbf{A}, \phi)$ satisfying certain appropriate boundary conditions,  the energy is a conserved quantity. 

It is well known that the solutions of  the above M-S system are not uniquely determined. In fact, the M-S system is invariant under the gauge transformation:
\begin{equation}\label{eq:1.7}
\Psi\longrightarrow \Psi^{\prime} = e^{\mathrm{i}\chi}\Psi, \quad \mathbf{A}\longrightarrow \mathbf{A}^{\prime} = \mathbf{A} + \nabla \chi ,\quad \phi \longrightarrow \phi^{\prime} = \phi - \frac{\partial \chi}{\partial t}, 
\end{equation} 
for any sufficiently smooth function $\chi : \Omega\times (0,T)\rightarrow \mathbb{R}$. That is, if $(\Psi, \mathbf{A}, \phi)$ satisfies M-S system, then so does $(\Psi^{\prime}, \mathbf{A}^{\prime}, \phi^{\prime})$.

In view of the gauge freedom, to obtain mathematically well-posed equations, we can impose some extra constraint, commonly known as gauge choice, on the solutions of the M-S system. In this paper, we study the M-S system in the Coulomb gauge, i.e. $\nabla \cdot \mathbf{A} = 0$.

By employing the atomic units,  i.e. $\hbar=m=q=1$,  and assuming that $\epsilon=\mu=1$, the M-S system in the Coulomb gauge (M-S-C) can be reformulated as follow:
\begin{equation}\label{eq:1.8}
\left\{
\begin{array}{@{}l@{}}
{\displaystyle  -\mathrm{i}\frac{\partial \Psi}{\partial t}+
\frac{1}{2}\left(\mathrm{i}\nabla +\mathbf{A}\right)^{2}\Psi
 + V\Psi + \phi\Psi= 0 \quad {\rm in} \,\,\, \Omega_{T},}\\[2mm]
{\displaystyle \frac{\partial ^{2}\mathbf{A}}{\partial t^{2}}+\nabla\times
(\nabla\times \mathbf{A}) +\frac{\partial (\nabla \phi)}{\partial t}+\frac{\mathrm{i}}{2}\big(\Psi^{*}\nabla{\Psi}-\Psi\nabla{\Psi}^{*}\big) }\\[2mm]
{\displaystyle \qquad\qquad+\vert\Psi\vert^{2}\mathbf{A}=0 \quad {\rm in} \,\,\, \Omega_{T},}\\[2mm]
{\displaystyle  \nabla \cdot \mathbf{A} =0, \quad -\Delta \phi = \vert\Psi\vert^{2}\quad {\rm in} \,\,\, \Omega_{T} }.
\end{array}
\right.
\end{equation}

In this paper, the M-S-C system (\ref{eq:1.8}) is considered in conjunction with the following initial boundary conditions:
\begin{equation}\label{eq:1.9}
\left\{
\begin{array}{lll}
{\displaystyle \Psi(\mathbf{x},t)=0,\quad \mathbf{A}(\mathbf{x},t)\times\mathbf{n}=0, \quad \phi(\mathbf{x},t) = 0 \quad \mathrm{on}\;
\partial \Omega\times(0,T),}\\[2mm]
{\displaystyle \Psi(\mathbf{x},0) = \Psi_0(\mathbf{x}),\quad\mathbf{A}(\mathbf{x},0)=\mathbf{A}_{0}(\mathbf{x}),\quad
\mathbf{A}_{t}(\mathbf{x},0)=\mathbf{A}_{1}(\mathbf{x}) \quad \rm{in}\;\Omega,}
\end{array}
\right.
\end{equation}
with $\nabla \cdot \mathbf{A}_{0} = \nabla\cdot\mathbf{A}_{1} = 0$.

In the M-S-C system, the energy $\mathcal{E}(t)$ takes the following form
\begin{equation}\label{eq:1.10}
\begin{array}{@{}l@{}}
{\displaystyle \mathcal{E}(t) = \int_{\Omega}\left(\frac{1}{2}\big\vert \left[i\nabla + q\mathbf{A}\right]\Psi\big\vert^{2} + V|\Psi|^{2} + \frac{1}{2}\big|\frac{\partial \mathbf{A}}{\partial t}\big|^{2} + \frac{1}{2}|\nabla \times \mathbf{A}|^{2} + \frac{1}{2}|\nabla \phi|^{2} \right)\;\mathrm{d}\mathbf{x}.  }
\end{array}
\end{equation}

\begin{remark}\label{rem1-1}
The boundary condition $
\Psi(\mathbf{x}, t)=0$ implies that the particle is confined in the domain $\Omega$. 
The boundary conditions $\mathbf{A}(\mathbf{x},t)\times\mathbf{n}=0 $ and $\phi(\mathbf{x},t) = 0$ denote the perfect conductive boundary condition (PEC).
We refer readers to \cite{Weng} for the determination of the boundary conditions for the vector potential $\mathbf{A}$ and the scalar potential $\phi$ in different electromagnetic environment.
\end{remark}

The existence and uniqueness of (smooth) solutions to the time-dependent M-S equations (\ref{eq:1.5}) in all of $\mathbb{R}^{2}$ or $\mathbb{R}^{3}$ have been studied in \cite{Guo,Nak,Nak-1,Nak-2}. However, these results don't hold for bounded domains because some important tools used in these work can't apply to bounded domains. For example, Strichartz estimates and many tools from Fourier analysis. In this paper, we will prove the existence of weak solutions to the M-S-C system in a bounded smooth domain by Galerkin's method and compactness arguments. To the best of our knowledge, this is the first result on the existence of weak solutions to the inital-boundary problem of the M-S-C system in a bounded smooth domain.

In recent years, with the development of nanotechnology, there has been considerable interest in developing physical models and numerical methods to simulate  light-matter interaction at the nanoscale.  Due to the natural coupling of the electromagnetic fields and quantum effects, the Maxwell--Schr\"{o}dinger model is widely used in simulating self-induced transparency in quantum dot systems \cite{karni},  laser-molecule interaction \cite{Lor-1},  carrier dynamics in nanodevices \cite{Pi} and molecular nanopolaritonics \cite{Lop}. However, in these existing methods, the Maxwell's equations of field type (\ref{eq:1.3}), instead of the potential type in (\ref{eq:1.5}),  are usually coupled to the Schr\"{o}dinger's equation through the dipole approximation or by extracting the vector potential $\mathbf{A}$ and the scalar potential $\phi$ from the electric field $\mathbf{E}$ and the magnetic field $\mathbf{H}$ \cite{Ah-1,Oh,Sui,Tur}. In part because there exists 
robust numerical algorithms for the Maxwell's equations (\ref{eq:1.3}), for example, the time domain finite difference (FDTD) method, the transmission line matrix (TLM) method, etc. Recently, RYU \cite{Ryu} used the FDTD scheme to discretize  the Maxwell--Schr\"{o}dinger equations (\ref{eq:1.5}) directly in the Lorentz gauge to simulate a single electron in an artificial atom excited by an incoming electromagnetic field. But so far,  there are rather limited studies on the numerical algorithms of the M-S system (\ref{eq:1.5}) as well as their convergence analysis.

In this paper we will present a fully discrete finite element method for solving the problem
(\ref{eq:1.8})-(\ref{eq:1.9}) and show that it is equivalent to a fully discrete Crank--Nicolson scheme based on mixed finite element method. We will show that our scheme maintains the conservation properties of the original system. Compared with the commonly used method which couples the Maxwell's equations of field type with the Schr\"{o}dinger's equation and solves the system by the FDTD method, our method keeps the total charge and energy of the discrete system conserved and may suffer from less restriction in the time step-size since we use the Crank--Nicolson scheme in the time direction. In this paper we establish the optimal error estimates for the proposed method without any restrictions on the spatial mesh step $h$ and the time step $\tau$. In general it is very difficult to derive error estimates without any restrictions on the spatial mesh step and the time step for the highly complicated, nonlinear equations since the inverse inequalities are usually used to bound the nonlinear terms. In this paper we avoid using the inverse inequalities due to two aspects. On the one hand, we deduce some stability estimates of the approximate solutions by using the conservation properties of our scheme. More importantly, we take advantage of the special structures of the system and make some difficult nonlinear terms in the Schr\"{o}dinger's equation and the Maxwell's equations respectively cancel out. To the best of our knowledge, this is the first theoretical analysis on the numerical algorithms for the M-S-C system (\ref{eq:1.8}).

The rest of this paper is organized as follows. In section~\ref{sec-2} we introduce some notation and prove the existence of weak solutions to the M-S-C system (\ref{eq:1.8})-(\ref{eq:1.9}). In section~\ref{sec-3},  we present two fully discrete finite element schemes for the M-S-C system and show that they are equivalent.  Section~\ref{sec-4} is devoted to the proof of energy-conserving property of the discrete system and some stability estimates of the approximate solutions. In section~\ref{sec-5}, we prove the existence and uniqueness of solutions to the discrete system. The optimal error estimates without any restrictions on the time step are derived in section~\ref{sec-6}. We provide some numerical experiments in section~\ref{sec-7} to confirm our theoretical analysis.
\section{Global existence of weak solutions to the M-S-C system}\label{sec-2}
In this section, we study the existence of weak solutions to the M-S-C system (\ref{eq:1.8}) together with the initial-boundary conditions (\ref{eq:1.9}) in a bounded smooth domain. For simplicity, We introduce some notation below. 

For any nonnegative integer $s$,  we denote $W^{s,p}(\Omega)$ as the conventional Sobolev spaces of the real-valued functions defined in $\Omega$ and  $W^{s,p}_{0}(\Omega)$ as the subspace of $W^{s,p}(\Omega)$ consisting of functions whose traces are zero on $\partial \Omega$.
As usual, we denote $H^{s}(\Omega)\,=\,W^{s,2}(\Omega)$,  $H^{s}_{0}(\Omega)\,=\,W^{s,2}_{0}(\Omega) $, and $L^{p}(\Omega)\,=\,W^{0,p}(\Omega)$, respectively.
We use  $\mathcal{H}^{s}(\Omega)=\{u+\mathrm{i}v\,|\, u,v \in H^{s}(\Omega)\}$ and $\mathcal{L}^{p}(\Omega)=\{u+\mathrm{i}v\,|\, u,v \in L^{p}(\Omega)\} $
with calligraphic letters for Sobolev spaces and Lebesgue spaces of the complex-valued functions, respectively.
Furthermore, let $\mathbf{H}^{s}(\Omega)=[H^{s}(\Omega)]^{d}$ and $\mathbf{L}^{p}(\Omega) = [L^{p}(\Omega)]^{d} $ with bold faced letters be Sobolev spaces and Lebesgue spaces of 
the vector-valued functions with $d$ components ($d$=2,\,3). The dual spaces of $\mathcal{H}_{0}^{s}(\Omega)$, ${H}_{0}^{s}(\Omega)$, and $\mathbf{H}_{0}^{s}(\Omega)$ are denoted by $\mathcal{H}^{-s}(\Omega)$, ${H}^{-s}(\Omega)$, and $\mathbf{H}^{-s}(\Omega)$, respectively.
$ L^{2}$ inner-products in $H^{s}(\Omega) $, $\mathcal{H}^{s}(\Omega)$,
and $\mathbf{H}^{s}(\Omega)$  are denoted by $(\cdot,\cdot )$ without ambiguity.

In particular, we consider the following subspaces of $\mathbf{H}^{1}(\Omega)$ and $\mathbf{L}^{2}(\Omega)$:
\begin{equation*}
\begin{array}{@{}l@{}}
{\displaystyle \mathbf{H}_{t}^{1}(\Omega)=\{\mathbf{v}\in\mathbf{H}^{1}(\Omega) \, |  \,\, \mathbf{v}\times\mathbf{n}=0 \,\, \,{\rm on}\, \,\partial\Omega\}, }\\[2mm]
{\displaystyle \mathbf{H}^{1}_{t,0}(\Omega)=\{\mathbf{v}\in\mathbf{H}_{t}^{1}(\Omega) \, | \,\,\nabla \cdot \mathbf{v} = 0 \}, }\\[2mm]
{\displaystyle \mathbf{L}^{2}_{0}(\Omega)=\{\mathbf{v}\in\mathbf{L}^{2}(\Omega) \, | \,\,\nabla \cdot \mathbf{v} = 0 \; {\rm weakly}\,\} }.
\end{array}
\end{equation*}
The semi-norms on $\mathbf{H}_{t}^{1}(\Omega)$ and $\mathbf{H}^{1}_{t,0}(\Omega)$ are defined by 
\begin{equation*}
| \mathbf{u} |_{ \mathbf{H}_{t}^{1}} : = \Vert\nabla \cdot \mathbf{u}\Vert_{{L}^{2}(\Omega)} + \Vert\nabla \times \mathbf{u}\Vert_{\mathbf{L}^{2}(\Omega)}, \quad | \mathbf{u} |_{ \mathbf{H}^{1}_{t,0}} : = \Vert\nabla \times \mathbf{u}\Vert_{\mathbf{L}^{2}(\Omega)},
\end{equation*}
both of which are equivalent to the standard $\mathbf{H}^{1}(\Omega)$-norm $\Vert \mathbf{u} \Vert_{ \mathbf{H}^{1}(\Omega)}\,$\cite{Gir}.

To take into account the time dependence, for any Banach space $W$ and integer $s \geq 1$, we define function spaces $C\big([0,T], \; W\big) $, $C\big((0,T), \; W\big) $, and $C^{s}_{0}\big((0,T), \; W\big)$ consisting of $W$-valued functions in $C[0,T]$, $C(0,T)$, and $C^{s}_{0}(0,T)$, respectively.

We now give two definitions of weak solution to the M-S-C system (\ref{eq:1.8}) together with the initial-boundary conditions (\ref{eq:1.9}).
\begin{definition}[Weak solution \uppercase\expandafter{\romannumeral1}]\label{def:2.1}
$(\Psi,\,\mathbf{A},\,\phi)$ is a weak solution of type \uppercase\expandafter{\romannumeral1} to (\ref{eq:1.8})-(\ref{eq:1.9}), if
\begin{subequations}
\begin{gather}
\Psi \in C\big([0,T]; \mathcal{L}^{2}(\Omega)\big)\cap L^{\infty}\big(0, T; \mathcal{H}_{0}^{1}(\Omega)\big), \;\frac{\partial \Psi}{\partial t} \in L^{\infty}\big(0, T; \mathcal{H}^{-1}(\Omega)\big),   \label{subeq:6-000}\\
\phi \in C\big([0,T]; {L}^{2}(\Omega)\big)\cap L^{\infty}\big(0, T; H_{0}^{1}(\Omega)\big),\;\frac{\partial \phi}{\partial t} \in L^{\infty}\big(0, T; L^{2}(\Omega)\big), \label{subeq:6-001}\\
\mathbf{A} \in C\big([0,T]; \,\mathbf{L}^{2}(\Omega)\big)\cap L^{\infty}\big(0, T; \,\mathbf{H}^{1}_{t,0}(\Omega)\big), \label{subeq:6-002}\\
\frac{\partial \mathbf{A}}{\partial t} \in C\big([0,T]; \,\big(\mathbf{H}^{1}_{t,0}(\Omega) \big)^{\prime}\,\big)\cap L^{\infty}\big(0, T;\, \mathbf{L}^{2}(\Omega)\big),  \label{subeq:6-003} 
\end{gather}
\end{subequations}
with the initial condition $\Psi(\cdot , 0) \in \mathcal{H}^{1}_{0}(\Omega)$, $\mathbf{A}(\cdot , 0) \in \mathbf{H}^{1}_{t,0}(\Omega)$, $\mathbf{A}_{t}(\cdot , 0) \in \mathbf{L}_{0}^{2}(\Omega)$, and the variational equations
\begin{equation}\label{eq:6-01}
\int_{0}^{T}\bigg[\mathrm{i}\big(\Psi,\,\frac{\partial \widetilde{\Psi}}{\partial t}\big)+
\frac{1}{2}\big(\left(\mathrm{i}\nabla +\mathbf{A}\right)\Psi,\,\left(\mathrm{i}\nabla +\mathbf{A}\right) \widetilde\Psi\big)
 + \big(V\Psi,  \widetilde\Psi \big)+\big(\phi\Psi, \widetilde\Psi\big)\bigg]\mathrm{d}t = 0,  
 \end{equation}
 \begin{equation}\label{eq:6-02}
 \int_{0}^{T}\bigg[\big( \mathbf{A},\frac{\partial ^{2}  \widetilde{\mathbf{A}}}{\partial t^{2}})+\big(\nabla\times
 \mathbf{A},\nabla\times  \widetilde{\mathbf{A}}\big) +\big(\frac{\mathrm{i}}{2}(\Psi^{*}\nabla{\Psi}-\Psi\nabla{\Psi}^{*}), \widetilde{\mathbf{A}}\big) + \big( |\Psi|^{2}\mathbf{A},  \widetilde{\mathbf{A}}\big)\bigg]\mathrm{d}t= 0, 
 \end{equation}
 \begin{equation}\label{eq:6-03}
\displaystyle \int_{0}^{T}\left[\big(\nabla \phi,\, \nabla  \widetilde{\phi}\big) - \big(|\Psi|^{2},\, \widetilde{\phi}\big)\right]\mathrm{d}t= 0, 
\end{equation}
hold for all $  \widetilde{\Psi} \in C_{0}^{1}\big((0,T); \,\mathcal{H}_{0}^{1}(\Omega)\big)$, $  \widetilde{\mathbf{A} }\in C_{0}^{2}\big((0,T);\,\mathbf{H}^{1}_{t,0}(\Omega)\big) $ and $\widetilde{\phi}\in C\big((0,T);\\ \, H_{0}^{1}(\Omega)\big)$.
\end{definition}

\begin{definition}[Weak solution \uppercase\expandafter{\romannumeral2}]\label{def:2.2}
$(\Psi,\,\mathbf{A},\,\phi)$ is a weak solution of type \uppercase\expandafter{\romannumeral2} to (\ref{eq:1.8})-(\ref{eq:1.9}), if (\ref{subeq:6-000}), (\ref{subeq:6-001}), (\ref{eq:6-01}) and (\ref{eq:6-03}) in Definition~\ref{def:2.1} are satisfied and 
\begin{subequations}
\begin{gather}
\mathbf{A} \in C\big([0,T]; \,\mathbf{L}^{2}(\Omega)\big)\cap L^{\infty}\big(0, T; \,\mathbf{H}_{t}^{1}(\Omega)\big), \label{subeq:6-004}\\
\frac{\partial \mathbf{A}}{\partial t} \in C\big([0,T]; \,(\mathbf{H}_{t}^{1}(\Omega))^{\prime}\big)\cap L^{\infty}\big(0, T;\, \mathbf{L}^{2}(\Omega)\big),  \label{subeq:6-005}
\end{gather}
\end{subequations}
 \begin{equation}\label{eq:6-04}
 \begin{array}{lll}
{\displaystyle  \int_{0}^{T}\bigg[\big( \mathbf{A},\frac{\partial ^{2}  \widetilde{\mathbf{A}}}{\partial t^{2}})+\big(\nabla\times
 \mathbf{A},\nabla\times  \widetilde{\mathbf{A}}\big)+\big(\frac{\mathrm{i}}{2}(\Psi^{*}\nabla{\Psi}-\Psi\nabla{\Psi}^{*}), \,\widetilde{\mathbf{A}}\big)   }\\[2mm]
 {\displaystyle \qquad \quad - \big(\frac{\partial \phi}{\partial t},\, \nabla \cdot \widetilde{\mathbf{A}}\big) + \big( |\Psi|^{2}\mathbf{A},  \,\widetilde{\mathbf{A}}\big)\bigg]\mathrm{d}t= 0,   \quad \forall \widetilde{\mathbf{A} }\in C_{0}^{2}\big((0,T);\,\mathbf{H}_{t}^{1}(\Omega)\big) }.
 \end{array}
 \end{equation}
\end{definition}

The following theorem shows that the above two definitons of weak solutions are equivalent.
\begin{theorem}\label{thm6-00}
The weak solutions to the M-S-C system defined in Definiton~\ref{def:2.1} and Definition~\ref{def:2.2} are equivalent.    
\end{theorem}
\begin{proof}
It suffices to show that the vector potential $\mathbf{A}$ in Definiton~\ref{def:2.1} and Definiton~\ref{def:2.2} are consistent.

For any $\varphi(\mathbf{x}) \in H_{0}^{1}(\Omega)\cap H^{2}(\Omega),\; \eta(t) \in C_{0}^{1}(0,\, T)$, by choosing  $\widetilde{\Psi} = \eta\Psi\varphi$ in (\ref{eq:6-01}),  $\widetilde{\phi} = \frac{{\rm d}\eta}{{\rm d} t}\varphi$ in (\ref{eq:6-03}), and taking the imaginary part of (\ref{eq:6-01}), we have 
\begin{equation}\label{eq:6-05}
\begin{array}{lll}
{ \displaystyle \int_{0}^{T}\Big[\big(\frac{\partial \rho}{\partial t},\, \varphi\big) - \big(\mathbf{J},\, \nabla \varphi\big) \Big]\eta(t)\mathrm{d}t= 0,  }\\[2mm]
{\displaystyle  \displaystyle \int_{0}^{T}\Big[\big(\frac{\partial\phi}{\partial t},\, \Delta \varphi\big) + \big(\frac{\partial \rho}{\partial t},\, \varphi\big)\Big]\eta(t) \mathrm{d}t= 0, } 
\end{array}
\end{equation}
where 
\begin{equation}\label{eq:6-06}
\rho = |\Psi|^{2},\;\; \mathbf{J} = -\frac{\mathrm{i}}{2}(\Psi^{*}\nabla{\Psi}-\Psi\nabla{\Psi}^{*}) -  |\Psi|^{2}\mathbf{A}.
\end{equation}
Since $\eta(t)$ is arbitrary, from (\ref{eq:6-05}) we see that 
\begin{equation}\label{eq:6-07}
\big(\frac{\partial \phi}{\partial t},\, \Delta \varphi\big) + \big(\mathbf{J},\, \nabla \varphi\big) = 0, \quad \forall \varphi \in H_{0}^{1}(\Omega)\cap H^{2}(\Omega).
\end{equation}

For any $\mathbf{v} \in \mathbf{H}_{t}^{1}(\Omega)$, we have the Helmholtz decomposition \cite{Gir}:
\begin{gather}\label{eq:6-08}
\mathbf{v} = \mathbf{v}_{0} + \nabla \varphi, \quad \mathbf{v}_{0} \in \mathbf{H}^{1}_{t,0}(\Omega),\quad \varphi \in H_{0}^{1}(\Omega)\cap H^{2}(\Omega), \\
\quad \Vert \mathbf{v}_{0} \Vert_{\mathbf{H}^{1}} \leq C\Vert \mathbf{v} \Vert_{\mathbf{H}^{1}}, \;\; \Vert\varphi\Vert_{H^{2}} \leq C \Vert \mathbf{v} \Vert_{\mathbf{H}^{1}}.
\end{gather}

We first prove that the vector potential $\mathbf{A}$ given by Definition~\ref{def:2.1} satisfies (\ref{subeq:6-004}), (\ref{subeq:6-005}), and (\ref{eq:6-04}) in Definition~\ref{def:2.2}. Obviously $\mathbf{A}$ satisfies (\ref{subeq:6-004}). Thanks to (\ref{subeq:6-002}) and (\ref{subeq:6-003}), we deduce that 
\begin{equation}\label{eq:6-09}
\big(\frac{\partial \mathbf{A}}{\partial t}, \, \nabla \varphi ) = 0, \; \forall \varphi \in H_{0}^{1}(\Omega)\cap H^{2}(\Omega).
\end{equation}
Then
\begin{equation}
\frac{\partial \mathbf{A}}{\partial t} \in C\big([0,T]; \,(\mathbf{H}_{t}^{1}(\Omega))^{\prime}\big)
\end{equation}
follows from (\ref{subeq:6-003}), (\ref{eq:6-08})-(\ref{eq:6-09}) and thus $\mathbf{A}$ satisfies (\ref{subeq:6-005}).
Since $\nabla \cdot \mathbf{A} = 0$, by using (\ref{eq:6-07}), we have
\begin{equation}\label{eq:6-010}
\begin{array}{lll}
{\displaystyle  \int_{0}^{T}\bigg[\big( \mathbf{A},\frac{\partial ^{2}  \widetilde{\mathbf{A}}}{\partial t^{2}})+\big(\nabla\times
 \mathbf{A},\nabla\times  \widetilde{\mathbf{A}}\big)+\big(\frac{\mathrm{i}}{2}(\Psi^{*}\nabla{\Psi}-\Psi\nabla{\Psi}^{*}), \,\widetilde{\mathbf{A}}\big) + \big( |\Psi|^{2}\mathbf{A},  \,\widetilde{\mathbf{A}}\big)  }\\[2mm]
 {\displaystyle \qquad \quad - \big(\frac{\partial \phi}{\partial t},\, \nabla \cdot \widetilde{\mathbf{A}}\big) \bigg]\mathrm{d}t= 0,   \; \forall \widetilde{\mathbf{A} }\in C_{0}^{2}\big((0,T);\,\nabla \big(H_{0}^{1}(\Omega)\cap H^{2}(\Omega)\big)\big). }
 \end{array}
\end{equation}
By applying (\ref{eq:6-02}), (\ref{eq:6-010}), and the Helmholtz docomposition (\ref{eq:6-08}),  we find that $\mathbf{A}$ satisfies (\ref{eq:6-04}). 

Next we assume that $\mathbf{A}$ is the vector potential given by Definition~\ref{def:2.2}. It is easy to see that $\mathbf{A}$ satisfies (\ref{subeq:6-003}) and (\ref{eq:6-03}). For any $\eta(t) \in C_{0}^{2}(0,T),\; \varphi \in H_{0}^{1}(\Omega)\cap H^{2}(\Omega)$, take $\widetilde{\mathbf{A}}=\eta\nabla\varphi$ in (\ref{eq:6-04}) and employ (\ref{eq:6-07}) to find
\begin{equation}
\int_{0}^{T}\big( \nabla\cdot \mathbf{A},\, \varphi)\frac{{\rm d}^{2} \eta}{{\rm d} t^{2}}\,\mathrm{d}t= 0,
\end{equation}
which implies that
\begin{equation}
\nabla \cdot \mathbf{A} = 0.
\end{equation}
Consequently $\mathbf{A}$ satisfies (\ref{subeq:6-002}) and we complete the proof of this theorem.
\end{proof}
\begin{remark}
Theorem~\ref{thm6-00} shows that weak solutions \uppercase\expandafter{\romannumeral2} satisfy $\nabla \cdot \mathbf{A} = 0$ implicity.
\end{remark}
 
Next we use the Galerkin method and compactness arguments to prove the existence of weak solutions to (\ref{eq:1.8})-(\ref{eq:1.9}). 

We first introduce two lemmas to construct finite dimensional subspaces of $\mathcal{H}_{0}^{1}(\Omega)$, $\mathbf{H}^{1}_{t,0}(\Omega)$, and $H^{1}_{0}(\Omega)$.
\begin{lemma}\label{lem6-0}
Suppose that $\Omega\subset \mathbb{R}^d (d= 2,3) $ is a bounded smooth domain. Then there exists a sequence $\{u_k\}_{k=1}^{\infty}$ being an orthogonal basis of $H_{0}^{1}(\Omega)$ as well as an  orthonormal basis of $L^{2}(\Omega)$. Here $u_k \in H_{0}^{1}(\Omega)$ is an eigenfunction corresponding to $\lambda_k$:
\begin{equation*}
\left\{
\begin{array}{@{}l@{}}
{\displaystyle -\Delta u_k = \lambda_k u_k \quad in \;\Omega, }\\[2mm]
{\displaystyle u_k = 0 \quad on \;\partial\Omega,}
\end{array}
\right.
\end{equation*}
for $k = 1,2\cdots.$
\end{lemma}

The proof of Lemma~\ref{lem6-0} is given in \cite{evans}. It is worth pointing out that the conclusion is also true for complex-valued functions, i.e. there exists a sequence $\{\varphi_k\}_{k=1}^{\infty}$ being an orthogonal basis of $\mathcal{H}_{0}^{1}(\Omega)$ as well as an  orthonormal basis of $\mathcal{L}^{2}(\Omega)$.

\begin{lemma}\label{lem6-1}
Suppose that $\Omega\subset \mathbb{R}^d (d= 2,3) $ is a bounded smooth domain. Then there exists an orthonormal basis $\{\mathbf{v}_k\}_{k=1}^{\infty}$ of $\mathbf{L}^{2}(\Omega)$, where $\mathbf{v}_k \in \mathbf{H}^{1}_{t,0}(\Omega)$ is an eigenfunction corresponding to $\mu_k$:
\begin{equation*}
\left\{
\begin{array}{@{}l@{}}
{\displaystyle \nabla\times\nabla\times\mathbf{v}_k = \mu_k \mathbf{v}_k \quad in \;\Omega, }\\[2mm]
{\displaystyle \mathbf{v}_k \times\mathbf{n}= 0 \quad on \;\partial\Omega,}
\end{array}
\right.
\end{equation*}
for $k = 1,2\cdots.$ Furthermore, $\{\mathbf{v}_k\}_{k=1}^{\infty}$ is an orthogonal basis of $\mathbf{H}^{1}_{t,0}(\Omega)$.
\end{lemma}
\begin{proof}
Let $L:\mathbf{H}^{1}_{t,0}(\Omega)\rightarrow \big(\mathbf{H}^{1}_{t,0}(\Omega)\big)^{\prime}$ be defined by
\begin{equation*}
(L\mathbf{u},\,\mathbf{v}) = (\nabla\times\mathbf{u},\,\nabla\times\mathbf{v}),\quad \forall \mathbf{u},\mathbf{v}\in \mathbf{H}^{1}_{t,0}(\Omega).
\end{equation*}
By the Lax-Milgram theorem, $L^{-1}: \big(\mathbf{H}^{1}_{t,0}(\Omega)\big)^{\prime}\rightarrow \mathbf{H}^{1}_{t,0}(\Omega)$ exists and is bounded. Since $\mathbf{H}^{1}_{t,0}(\Omega)$ is compactly embedded into $\mathbf{L}^{2}(\Omega)$, $S=L^{-1}$ is a bounded, linear, compact operator mapping $\mathbf{L}^{2}(\Omega)$ into itself. It is easy to show that $S$ is self-adjoint. Then by the Hilbert-Schmidt theorm, there exists a countable orthonormal basis $\{\mathbf{v}_k\}_{k=1}^{\infty}$ of $\mathbf{L}^{2}(\Omega)$ consisting  of eigenfunctions of $S$. The proof of $\{\mathbf{v}_k\}_{k=1}^{\infty}$ being an orthogonal basis of $\mathbf{H}^{1}_{t,0}(\Omega)$ is straightfoward.
\end{proof}

Let $\mathcal{Z}_n$, $\mathbf{Z}_{n}$, and $Z_{n}$ be n-dimensional subspaces of $\mathcal{H}_{0}^{1}(\Omega)$, $\mathbf{H}^{1}_{t,0}(\Omega)$, and $H^{1}_{0}(\Omega)$, respectively, 
\begin{equation*}
\mathcal{Z}_n = {\rm span}\{\varphi_1,\cdots,\varphi_n\}, \quad
 \mathbf{Z}_n = {\rm span}\{\mathbf{v}_1,\cdots,\mathbf{v}_n\},\quad
 Z_{n} = {\rm span}\{u_1,\cdots, u_n\},
\end{equation*}
where $\{\varphi_k\}_{k=1}^{\infty}$, $\{\mathbf{v}_k\}_{k=1}^{\infty}$, and $\{u_k\}_{k=1}^{\infty}$ are given in Lemma~\ref{lem6-0} and \ref{lem6-1}.

For each $n$, we can construct the Galerkin approximate solutions of weak solutions to the M-S-C system in the sense of Definition~\ref{def:2.1} as follows.

Find $\Psi_n(t) = \sum_{j=1}^{n}{h_{jn}(t)\varphi_j \in \mathcal{Z}_n}$, $ \mathbf{A}_{n}(t)= \sum_{j=1}^{n}{g_{jn}(t)\mathbf{v}_j\in\mathbf{Z}_n}$, and $\phi_n(t) = \sum_{j=1}^{n}{f_{jn}(t)u_j \in Z_n}$ such that
\begin{equation}\label{eq:6-2}
\left\{
\begin{array}{@{}l@{}}
{\displaystyle
 -\mathrm{i}(\frac{\partial \Psi_n}{\partial t},\varphi_j)+
\frac{1}{2}(\left(\mathrm{i}\nabla +\mathbf{A}_n\right)\Psi_n,\left(\mathrm{i}\nabla +\mathbf{A}_n\right)\varphi_j) + (V\Psi_n,\varphi_j) + (\phi_{n}\Psi_{n}, \varphi_j) = 0, }\\[2mm]
{\displaystyle ( \frac{\partial ^{2}\mathbf{A}_n}{\partial t^{2}},\mathbf{v}_j)+(\nabla\times
 \mathbf{A}_n,\nabla\times\mathbf{v}_j) +\big(\frac{\mathrm{i}}{2}(\Psi_n^{*}\nabla{\Psi_n}-\Psi_n\nabla{\Psi}_n^{*}),\mathbf{v}_j\big) }\\[2mm]
 {\displaystyle \qquad \qquad + ( |\Psi_n|^{2}\mathbf{A}_n,\mathbf{v}_j)=0, }\\[2mm]
 {\displaystyle (\nabla \phi_{n},\,\nabla u_j) = (|\Psi_{n}|^{2},\,u_j), }\\[2mm]
{\displaystyle \Psi_n(0) = \mathcal{P}_n\Psi_0, \,\,\mathbf{A}_n(0) = \mathbf{P}_n\mathbf{A}_0,\,\, \frac{\partial\mathbf{A}_n}{\partial t}(0) = \mathbf{P}_n\mathbf{A}_1}
\end{array}
\right.
\end{equation}
for any $t\in(0,T)$, $j = 1, 2,\cdots,n$. Here $ \mathcal{P}_n$ and $\mathbf{P}_n$ denote the orthogonal projection onto $\mathcal{Z}_n$ and $\mathbf{Z}_n$, respectively.

Using the local existence and uniqueness theory on ODEs,  we can show that the nonlinear differential system (\ref{eq:6-2}) has a unique local solution defined on some interval $[0, T_n]$. Next we derive some $a \,\,priori$ estimates to extend the local solution to a global solution defined on  $[0, T]$. In this paper, the following lemma will be used frequently.
\begin{lemma}\label{lem6-2}
Let $2<p<6$. Suppose that $\Omega\subset \mathbb{R}^d$, $ d\geq 2 $ is a bounded Lipschitz domain. Then for each $\epsilon > 0$, there exists some constant $C_{\epsilon}$ depending on $\epsilon$ (and on $p$ and $\Omega$) such that
\begin{equation*}
\Vert u \Vert_{\mathcal{L}^p}\leqslant \epsilon \Vert \nabla u \Vert_{\mathbf{L}^{2}} + C_{\epsilon} \Vert u \Vert_{\mathcal{L}^{2}},\qquad \forall u \in \mathcal{H}_{0}^{1}(\Omega).
\end{equation*}
\end{lemma}
Lemma~\ref{lem6-0} can be proved by applying Sobolev's embedding thorems, Poincar$\acute{\rm e}$'s inequality, and the following lemma in \cite{temma}.
\begin{lemma}\label{lem6-3}
Let $W_0$, $W$, and $W_1$ be three Banach spaces such that $W_0\subset W \subset W_1$, the injection of $W$ into $W_1$ being continuous, and the injection of  $W_0$ into $W$ is compact. Then for each $\epsilon > 0$, there exists some constant $C_{\epsilon}$ depending on $\epsilon$ (and on the spaces $W_0$, $W$, $W_1$) such that
\begin{equation*}
\Vert u \Vert_{W}\leqslant \epsilon \Vert u \Vert_{W_{0}} + C_{\epsilon} \Vert u \Vert_{W_{1}},\qquad \forall u \in W_{0}.
\end{equation*}
\end{lemma}

We define the energy  $\mathcal{E}_{n}(t)$ of the approximate system (\ref{eq:6-2}) as follows.
\begin{equation}\label{eq:6-3}
\begin{array}{@{}l@{}}
{\displaystyle \mathcal{E}_{n}(t) = \int_{\Omega}\big(\frac{1}{2}\big\vert \left[i\nabla + \mathbf{A}_{n}\right]\Psi_{n}\big\vert^{2} + V|\Psi_{n}|^{2} + \frac{1}{2}\big|\frac{\partial \mathbf{A}_{n}}{\partial t}\big|^{2} }\\[2mm]
 {\displaystyle \qquad \qquad \quad  + \frac{1}{2}|\nabla \times \mathbf{A}_{n}|^{2} + \frac{1}{2}|\nabla \phi_{n}|^{2} \big)\mathrm{d}\mathbf{x}.  }
 \end{array}
\end{equation}

\begin{lemma}\label{lem6-4}
For any $ t \in(0,T] $, if the solution $(\Psi_n(t), \mathbf{A}_{n}(t),\phi_{n}(t))$ of the approximate system (\ref{eq:6-2}) exists, we have the conservation of total charge and energy
\begin{equation}\label{eq:6-3-0}
\Vert\Psi_n(t)\Vert^{2}_{\mathcal{L}^{2}} = \Vert\Psi_n(0)\Vert^{2}_{\mathcal{L}^{2}},
\qquad \mathcal{E}_{n}(t) = \mathcal{E}_{n}(0).
\end{equation}
\end{lemma}
\begin{proof}
$(\ref{eq:6-3-0})_{1}$ can be proved by multiplying the first equation of (\ref{eq:6-2}) by $h_{jn}$, summing $j=1,2,\cdots,n$, and taking the imaginary part.

To prove $(\ref{eq:6-3-0})_{2}$, we first multiply the first equation of (\ref{eq:6-2}) by $\frac{d}{dt} h_{jn}$, sum $j=1,2,\cdots,n$, and take the real part. We discover
\begin{equation}\label{eq:6-5}
\begin{array}{@{}l@{}}
{\displaystyle \frac{d}{dt}\int_{\Omega}\big(\frac{1}{2}\big\vert \left[i\nabla + \mathbf{A}_{n}\right]\Psi_{n}\big\vert^{2} + V|\Psi_{n}|^{2}\big)\,\mathrm{d}\mathbf{x}
- \big(|\Psi_n |^{2}\mathbf{A}_n,\,\frac{\partial \mathbf{A}_n}{\partial t}\big)}\\[2mm]
{\displaystyle\quad\quad-\big(\frac{\mathrm{i}}{2}(\Psi_n^{*}\nabla{\Psi_n}-\Psi_n\nabla{\Psi}_n^{*}),\,\frac{\partial \mathbf{A}_n}{\partial t}\big) + \big(\phi_{n},\, \frac{\partial}{\partial t} |\Psi_{n}|^{2}\big)=0 }.
\end{array}
\end{equation}
Multiplying the second equation of (\ref{eq:6-2}) by $\frac{d}{dt} g_{jn}$ ,  summation gives
\begin{equation}\label{eq:6-5-0}
 \begin{array}{@{}l@{}}
{\displaystyle \frac{d}{dt}\int_{\Omega}\big(\frac{1}{2}\big|\frac{\partial \mathbf{A}_{n}}{\partial t}\big|^{2} + \frac{1}{2}|\nabla \times \mathbf{A}_{n}|^{2}\big)\,\mathrm{d}\mathbf{x}
+ \big(|\Psi_n |^{2}\mathbf{A}_n,\,\frac{\partial \mathbf{A}_n}{\partial t}\big)}\\[2mm]
{\displaystyle\quad\quad\quad\quad+\big(\frac{\mathrm{i}}{2}(\Psi_n^{*}\nabla{\Psi_n}-\Psi_n\nabla{\Psi}_n^{*}),\,\frac{\partial \mathbf{A}_n}{\partial t}\big) = 0 }.
\end{array}
\end{equation} 

Differentiating both sides of the third equation of (\ref{eq:6-2}) with respect to $t$, multiplying it by $f_{jn}$ and summing $j=1,2,\cdots,n$, we obtain 
\begin{equation}\label{eq:6-5-1}
\frac{1}{2}\frac{d}{dt}\int_{\Omega}|\nabla \phi_{n}|^{2} \,\mathrm{d}\mathbf{x} - \big(\frac{\partial}{\partial t} |\Psi_{n}|^{2},\,\phi_{n}\big) = 0.
\end{equation}
Adding (\ref{eq:6-5}), (\ref{eq:6-5-0}), and (\ref{eq:6-5-1}) together completes the proof of Lemma~\ref{lem6-1}.  \qquad\end{proof}

We now establish some estimates of the solution $(\Psi_n(t), \mathbf{A}_{n}(t),\phi_{n}(t))$.
\begin{theorem}\label{thm6-0}
For any $ t \in(0,T] $, if the solution $(\Psi_n(t), \mathbf{A}_{n}(t),\phi_{n}(t))$ exists, then it satisfies the estimates
\begin{equation}\label{eq:6-6}
\Vert \Psi_{n}\Vert_{\mathcal{H}^{1}} + \Vert \mathbf{A}_{n} \Vert_{\mathbf{H}^{1}} +\Vert \phi_{n} \Vert_{H^1} \leq C,
\end{equation}
\begin{equation}\label{eq:6-7}
\Vert \frac{\partial\Psi_{n}}{\partial t}\Vert_{\mathcal{H}^{-1}} + \Vert \frac{\partial\mathbf{A}_{n}}{\partial t} \Vert_{\mathbf{L}^{2}} +  \Vert \frac{\partial^{2}\mathbf{A}_{n}}{\partial t^{2}} \Vert_{(\mathbf{H}^{1}_{t,0})^{\prime}} +\Vert \frac{\partial\phi_{n} }{\partial t} \Vert_{L^2} \leq C,
\end{equation}
where $C$ is independent of $t$ and $n$.
\end{theorem}

\begin{proof}
By the definition of initial data $(\Psi_n(0), \mathbf{A}_{n}(0),\phi_{n}(0))$, it is easy to show that $\mathcal{E}_{n}(0) \leq C$. Thus by applying (\ref{eq:6-3-0}), we have 
\begin{equation}\label{eq:6-8}
\Vert \big(i\nabla + \mathbf{A}_{n}\big)\Psi_{n} \Vert_{\mathbf{L}^{2}} + \Vert \frac{\partial\mathbf{A}_{n}}{\partial t} \Vert_{\mathbf{L}^{2}} + \Vert \nabla \times \mathbf{A}_{n} \Vert_{\mathbf{L}^{2}} + \Vert 
\nabla \phi_{n} \Vert_{\mathbf{L}^{2}} \leq C.
\end{equation}
Since the semi-norm in $\mathbf{H}^{1}_{t,0}(\Omega)$ is equivalent to $\mathbf{H}^{1}$-norm, we get
\begin{equation}\label{eq:6-9}
\Vert \mathbf{A}_n\|_{\mathbf{H}^1} \leq C.
\end{equation}
Then Sobolev's imbedding theorem implies that
\begin{equation}\label{eq:6-10}
\Vert \mathbf{A}_n\|_{\mathbf{L}^p} \leq C ,
\end{equation}
with $ \, 1\leq p \leq 6$ for $d = 3$ and $ 1\leq p < \infty $ for $ d=2$.

Using  Lemma~\ref{lem6-0}, we further prove
\begin{equation}\label{eq:6-11}
{\displaystyle \|{\mathbf{A}}_{n}\Psi_{n}\|_{\mathbf{L}^2}\leq \|{\mathbf{A}}_{n}\|_{\mathbf{L}^6} \|\Psi_{n}\|_{\mathcal{L}^3} \leq C \|\Psi_{n}\|_{\mathcal{L}^3}
\leq C\|\Psi_{n}\|_{\mathcal{L}^2}+\frac{1}{2}\|\nabla\Psi_{n}\|_{\mathbf{L}^2}}\\[2mm]
{\displaystyle }.
\end{equation}
From (\ref{eq:6-8}), (\ref{eq:6-11}) and Lemma~\ref{lem6-1}, we deduce
\begin{equation*}
\begin{array}{@{}l@{}}
{\displaystyle \|\nabla\Psi_{n}\|_{\mathbf{L}^2}\leq \|\left(\mathrm{i}\nabla
+{\mathbf{A}}_{n}\right)\Psi_{n}\|_{\mathbf{L}^2}
+ \|{\mathbf{A}}_{n}\Psi_{n}\|_{\mathbf{L}^2}
\leq C+\frac{1}{2}\|\nabla\Psi_{n}\|_{\mathbf{L}^2}.}
\end{array}
\end{equation*}
Consequently, we obtain
\begin{equation}\label{eq:6-12}
\|\Psi_{n}\|_{\mathcal{H}^1} \leq C.
\end{equation}
By applying Poincar$\acute{\rm e}$'s inequality and (\ref{eq:6-8}), we see that
\begin{equation}\label{eq:6-13}
\|\phi_{n}\|_{{H}^1} \leq C.
\end{equation}
Therefore, (\ref{eq:6-6}) is proved by combining (\ref{eq:6-9}), (\ref{eq:6-12}), and (\ref{eq:6-13}).

To estimate $\Vert \frac{\partial\Psi_{n}}{\partial t}\Vert_{\mathcal{H}^{-1}}$, we first fix $\omega\in\mathcal{H}_{0}^{1}(\Omega)$ with $\Vert \omega \Vert_{\mathcal{H}_{0}^{1}}\leq 1 $. Note that $\{\varphi_n\}_{n=1}^{\infty}$ is an orthogonal basis of $\mathcal{H}_{0}^{1}(\Omega)$ as well as an  orthonormal basis of $\mathcal{L}^{2}(\Omega)$. Thus we can write $\omega = {\omega}^1 + {\omega}^2$, where ${\omega}^1\in \mathcal{Z}_{n} = {\rm span}\{\varphi_k\}_{k=1}^{n}$ and $({\omega}^2, \varphi_{k})  = 0$ for $k=1,\cdots,n.$ It is clear that $\Vert {\omega}^1 \Vert_{\mathcal{H}_{0}^{1}}\leq 1$. Then the first equation of (\ref{eq:6-2}) implies that 
\begin{equation*}
 \begin{array}{@{}l@{}}
{\displaystyle \big\langle \frac{\partial\Psi_{n}}{\partial t}, \, {\omega}\big\rangle_{\mathcal{H}^{-1},\,\mathcal{H}^{1}_{0}} = \big (\frac{\partial\Psi_{n}}{\partial t}, \, {\omega} \big) = \big (\frac{\partial\Psi_{n}}{\partial t}, \, {\omega}^{1} \big) 
 }\\[2mm]
{\displaystyle = -\frac{{\rm i}}{2}(\left(\mathrm{i}\nabla +\mathbf{A}_n\right)\Psi_n,\left(\mathrm{i}\nabla +\mathbf{A}_n\right) {\omega}^{1})  -{\rm i} (V\Psi_n, {\omega}^{1}) -{\rm i} (\phi_{n}\Psi_{n},  {\omega}^{1})    }.
\end{array}
\end{equation*}
Thus by applying (\ref{eq:6-6}), we obtain
\begin{equation}\label{eq:6-14}
\big |\big\langle \frac{\partial\Psi_{n}}{\partial t}, \, {\omega} \big\rangle_{\mathcal{H}^{-1},\,\mathcal{H}^{1}_{0}} \big | \leq C \Vert  {\omega}^1 \Vert_{\mathcal{H}_{0}^{1}} \leq C,
\end{equation}
which implies that
\begin{equation}\label{eq:6-15}
\Vert \frac{\partial\Psi_{n}}{\partial t}\Vert_{\mathcal{H}^{-1}}\leq C.
\end{equation}
Similarly, we can prove 
\begin{equation}\label{eq:6-16}
\Vert \frac{\partial^{2}\mathbf{A}_{n}}{\partial t^{2}} \Vert_{(\mathbf{H}^{1}_{t,0})^{\prime}} \leq C.
\end{equation}
It remains to show  $\Vert \frac{\partial\phi_{n} }{\partial t} \Vert_{L^2} \leq C.$ In order to estimate $\Vert \frac{\partial\phi_{n} }{\partial t} \Vert_{L^2}$, we fix $f\in L^{2}(\Omega)$ and find $\psi_{n} \in Z_{n} =  {\rm span}\{u_k\}_{k=1}^{n}$ such that
\begin{equation}\label{eq:6-17}
(\nabla \psi_{n}, \,\nabla u) = (f, \,u)  \quad \forall u \in Z_{n}.
\end{equation}
Then  by differentiating both sides of the third equation of (\ref{eq:6-2}) with respect to $t$ and using (\ref{eq:6-17}), we have
\begin{equation}
\big(\frac{\partial\phi_{n} }{\partial t},\, f\big) = \big(\nabla \frac{\partial\phi_{n} }{\partial t},\, \nabla \psi_{n}\big) = \big(\frac{\partial }{\partial t}|\Psi_{n}|^{2}, \, \psi_{n}\big).
\end{equation}
Thus from (\ref{eq:6-15}) and (\ref{eq:6-12}), we deduce
\begin{equation}\label{eq:6-18}
\big| \big(\frac{\partial\phi_{n} }{\partial t},\, f\big) \big | \leq C \Vert \frac{\partial\Psi_{n}}{\partial t}\Vert_{\mathcal{H}^{-1}} \Vert \Psi_{n}\Vert_{\mathcal{H}^{1}}\Vert \psi_{n}\Vert_{H^{2}} \leq C \Vert \psi_{n}\Vert_{H^{2}} .
\end{equation}
Next we will prove  $\Vert \psi_{n}\Vert_{H^{2}} \leq C \Vert f\Vert_{L^{2}}$. Note that each basis function $u_j\in Z_{n}$ satisfies $-\Delta u_j = \lambda_j u_j$. By (\ref{eq:6-17}), we have
\begin{equation}
(-\Delta \psi_{n}, \,-\Delta u_j) = (-\Delta \psi_{n}, \,\lambda_j u_j) =\lambda_j (\nabla \psi_{n}, \, \nabla u_j)
= \lambda_j (f, \,u_j) = (f,\, -\Delta u_j) .
\end{equation}
It follows that
\begin{equation}
\Vert \Delta \psi_{n} \Vert^{2}_{{L}^{2}} = (-\Delta \psi_{n}, \,-\Delta \psi_{n}) = (f,\, -\Delta \psi_{n})
\leq \Vert f\Vert_{L^{2}} \Vert \Delta \psi_{n} \Vert_{{L}^{2}},
\end{equation}
which implies that
\begin{equation}
\Vert \Delta \psi_{n} \Vert_{{L}^{2}} \leq  \Vert f\Vert_{L^{2}}.
\end{equation}
It is easy to show $\Vert \nabla \psi_{n} \Vert_{{L}^{2}} \leq  C \Vert f\Vert_{L^{2}}$. Hence we get
\begin{equation}\label{eq:6-19}
\Vert \psi_{n}\Vert_{H^{2}} \leq C \Vert f\Vert_{L^{2}}.
\end{equation}
Combining (\ref{eq:6-18}) and (\ref{eq:6-19}), we find $ \Vert \frac{\partial\phi_{n} }{\partial t} \Vert_{L^2} \leq C.$ Thus we complete the proof of Theorem~\ref{thm6-0}.   \quad \end{proof}

Using the above energy estimates, we have
\begin{corollary}
Given $T>0$, the nonlinear differential system (\ref{eq:6-2}) has a unique global solution $(\Psi_n,\mathbf{A}_n,\phi_n)$, which satisfies (\ref{eq:6-6}) and (\ref{eq:6-7}). 
\end{corollary}

We now quote a compactness lemma which can be found in \cite{Simon}.
\begin{lemma}\label{lem6-5}
Let $B_0$, $B$, $B_1$ be three Banach spaces such that $B_0\subset B \subset B_1$ with continuous embedding and the embedding $B_0\subset B$ is compact. Suppose $F$ is a bounded set in $L^{\infty}(a,b; B_0)$ such that $\frac{\partial F}{\partial t} = \{\frac{\partial f}{\partial t} \,|\, f\in F\}$ is bounded in 
$L^{p}(a,b; B_1)$ for some $p>1$. Then $F$ is relatively compact in $C([0,T]; B)$.
\end{lemma}

From Lemma~\ref{lem6-5} and Theorem~\ref{thm6-0}, we deduce that there exists $\Psi \in C([0,T];\\ \mathcal{L}^{p}(\Omega))\cap L^{\infty}((0, T); \mathcal{H}_{0}^{1}(\Omega))$, $\mathbf{A} \in C([0,T]; \mathbf{L}^{p}(\Omega))\cap L^{\infty}((0, T); \mathbf{H}^{1}_{t,0}(\Omega))$, $\phi \in C([0,T]; {L}^{p}(\Omega))\cap L^{\infty}((0, T); H_{0}^{1}(\Omega))$ and a subsequence $\{(\Psi_{n_k},\mathbf{A}_{n_k}, \phi_{n_k}) \}$ such that as $n_k\rightarrow \infty $
\begin{equation}\label{eq:6-20}
\begin{array}{lll}
{\displaystyle    \Psi_{n_k} \rightarrow \Psi \;in \;C([0,T]; \mathcal{L}^{p}(\Omega)), \;\;  \Psi_{n_k}\rightharpoonup \Psi \;in\; L^{\infty}(0, T; \mathcal{H}_{0}^{1}(\Omega)) \; weak\text{-}star,             }\\[2mm]
{\displaystyle \mathbf{A}_{n_k}\rightarrow \mathbf{A} \; in\; C([0,T]; \mathbf{L}^{p}(\Omega)),\;\; \mathbf{A}_{n_k}\rightharpoonup \mathbf{A} \; in\; L^{\infty}(0,T; \mathbf{H}^{1}_{t,0}(\Omega)) \; weak\text{-}star,   } \\[2mm]
{\displaystyle    \phi_{n_k} \rightarrow  \phi \;in \;C([0,T]; {L}^{p}(\Omega)), \;\;  \phi_{n_k}\rightharpoonup \phi \;in\; L^{\infty}(0, T; {H}_{0}^{1}(\Omega)) \; weak\text{-}star.          }
\end{array}
\end{equation}
Here $1<  p < \infty$ for $ d=2$ and $ 1<  p < 6$ for $d =3$.

Furthermore, we have the following convergence properties for time derivatives of $\{(\Psi_{n_k},\mathbf{A}_{n_k}, \phi_{n_k}) \}$.
\begin{equation}\label{eq:6-21}
\begin{array}{lll}
{\displaystyle \frac{\partial \Psi_{n_k} }{\partial t} \rightharpoonup \frac{\partial \Psi }{\partial t} \;in\; L^{\infty}(0, T; \mathcal{H}^{-1}(\Omega)) \;\;weak\text{-}star, }\\[2mm]
{\displaystyle \frac{\partial \phi_{n_k} }{\partial t} \rightharpoonup \frac{\partial \phi }{\partial t} \;in\; L^{\infty}(0, T; {L}^{2}(\Omega)) \; \; weak\text{-}star,    }\\[2mm]
{\displaystyle  \frac{\partial \mathbf{A}_{n_k} }{\partial t} \rightharpoonup \frac{\partial \mathbf{A} }{\partial t} \;in\; L^{\infty}(0, T; \mathbf{L}^{2}(\Omega)) \;\;  weak\text{-}star,}\\[2mm]
{\displaystyle  \frac{\partial^{2} \mathbf{A}_{n_k} }{\partial t^{2}} \rightharpoonup \frac{\partial ^{2} \mathbf{A} }{\partial t^{2}} \;in\; L^{\infty}(0, T; (\mathbf{H}^{1}_{t,0}(\Omega))^{\prime}) \; \; weak\text{-}star}.\\[2mm]
\end{array}
\end{equation}

Passing to the limits  $n_k \rightarrow \infty $ in our Galerkin appriximations, we obtain
\begin{theorem}\label{thm6-1}
Given $T > 0$, there exists a weak solution $(\Psi,\mathbf{A}, \phi)$ to the M-S-C system (\ref{eq:1.8})-(\ref{eq:1.9}) in the sense of Definition~\ref{def:2.1}, which satisfies the conservation of the total energy:
\begin{equation}
\mathcal{E}(t) = \mathcal{E}(0).  
\end{equation}
where $\mathcal{E}(t)$ is given in (\ref{eq:1.10}).
\end{theorem}

Here we omit the proof of the weak limit satisfies (\ref{eq:6-01})-(\ref{eq:6-03}) since the technique is standard.
\begin{remark}
By making a slight modification of the proof in this section, Theorem~\ref{thm6-1} holds for the M-S-C system with bounded coefficients:
\begin{equation}
\left\{
\begin{array}{@{}l@{}}
{\displaystyle  -\mathrm{i}\frac{\partial \Psi}{\partial t}+
\frac{1}{2}\big(\mathrm{i}\nabla +\mathbf{A}\big) \cdot \big( \frac{1}{m(\mathbf{x})}(\mathrm{i}\nabla +\mathbf{A})\Psi\big)
 + V\Psi + \phi\Psi= 0 ,\,\, (\mathbf{x},t)\in
\Omega\times(0,T),}\\[2mm]
{\displaystyle \frac{\partial ^{2}\mathbf{A}}{\partial t^{2}}+\nabla\times
(\mu^{-1}(\mathbf{x})\nabla\times \mathbf{A}) +\frac{\partial (\nabla \phi)}{\partial t}+\frac{\mathrm{i}}{2m(\mathbf{x})}\big(\Psi^{*}\nabla{\Psi}-\Psi\nabla{\Psi}^{*}\big) }\\[2mm]
{\displaystyle \qquad\qquad+\frac{1}{m(\mathbf{x})}\vert\Psi\vert^{2}\mathbf{A}=0,
\,\,\quad (\mathbf{x},t)\in \Omega\times(0,T),}\\[2mm]
{\displaystyle  \nabla \cdot \mathbf{A} =0, \quad -\Delta \phi = \vert\Psi\vert^{2},\,\, (\mathbf{x},t)\in \Omega\times(0,T) },
\end{array}
\right.
\end{equation}
where $0 < \alpha_{0} \leq m(\mathbf{x}) \leq \alpha_{1}$, $0 < \beta_{0} \leq \mu(\mathbf{x}) \leq \beta_{1}$.

In particular, Theorem~\ref{thm6-1} holds for the M-S-C system with rapidly oscillating discontinuous coefficients arising from the modeling of a heterogeneous structure with a periodic microstructure, i.e. $m(\mathbf{x}) = a(\frac{\mathbf{x}}{\varepsilon})$, $\mu(\mathbf{x}) = b(\frac{\mathbf{x}}{\varepsilon})$, where $a({\bm \xi})$, $b({\bm \xi})$ are 1-periodic in $ {\bm \xi}$ and $\varepsilon>0 $ is a small parameter. Furthermore, if $\mathcal{E}(0)$, the initial energy, is independent of $\varepsilon$, then $\mathcal{E}(t)$ is also independent of $\varepsilon$.
\end{remark}

\section{Fully discrete finite element scheme}\label{sec-3}

In this section, we consider the fully discretization of  the M-S-C system (\ref{eq:1.8})-(\ref{eq:1.9}) by the Galerkin finite element method in space together with the Crank-Nicolson scheme in time. In the following of the paper, we assume that $\Omega $ is  a bounded Lipschitz polyhedron convex domain in $\mathbb{R}^{3}$. 

Let us first triangulate the space domain $\Omega$ and assume that $\mathcal{T}_{h}=\{K\}$ is a regular partition of $\Omega$ into tetrahedrons  of maximal diameter $h$.  Without loss of generality, we assume that $0 < h < 1$. We denote by $P_{r}(K)$ the space of polynomials of degree $r$ defined on the element $K$. In the rest of this paper, we assume that $r \leq 2$ unless otherwise specified. For a given partition $\mathcal{T}_{h}$, we define the classical Lagrange finite element space
\begin{equation}
Y^{r}_{h} = \{u_{h} \in C(\Omega):\; u_{h}|_{K}  \in P_{r}(K), \; \forall \; K \in  \mathcal{T}_{h}\}.
\end{equation}

We have the following finite element subspaces of $H_{0}^{1}(\Omega)$,  $\mathcal{H}_{0}^{1}(\Omega)$, and $\mathbf{H}^{1}_{\rm t}(\Omega)$ 
\begin{equation}
X_{h}^{r} = Y^{r}_{h} \cap  H_{0}^{1}(\Omega), \quad \mathcal{X}_{h}^{r} = X_{h}^{r} \oplus {\rm i}X_{h}^{r},\quad \mathbf{X}_{h} = \big( Y^{2}_{h}\big)^{3}\cap \mathbf{H}^{1}_{\rm t}(\Omega).
\end{equation}

Let $\mathcal{I}_{h}$, $I_{h}$, and $\mathbf{I}_{h}$ be the commonly used Lagrange interpolation on $\mathcal{X}_{h}^{r}$, $X_{h}^{r}$, and $\mathbf{X}_{h}$, respectively. For $r=1,2$, $1\leq s \leq r+1$, we have the following interpolation error estimates \cite{Bre}:
\begin{subequations}
\begin{gather}
\Vert \psi - \mathcal{I}_{h}\psi\Vert_{\mathcal{L}^{2}} + h \Vert \psi - \mathcal{I}_{h}\psi\Vert_{\mathcal{H}^{1}} \leq C h^{s}\Vert\psi\Vert_{\mathcal{H}^{s}}, \; \forall \psi \in \mathcal{H}_{0}^{1}(\Omega)\cap\mathcal{H}^{r+1}(\Omega) , \label{subeq:2.00}  \\
  \Vert u - {I}_{h}u\Vert_{{L}^{2}}+h\Vert u - {I}_{h}u\Vert_{{H}^{1}} \leq C h^{s}\Vert u\Vert_{{H}^{s}}, \; \forall u \in {H}_{0}^{1}(\Omega)\cap {H}^{r+1}(\Omega) ,  \label{subeq:2.01} \\
 \Vert \mathbf{v} - \mathbf{I}_{h}\mathbf{v}\Vert_{\mathbf{L}^{2}} +   h \Vert \mathbf{v} - \mathbf{I}_{h}\mathbf{v}\Vert_{\mathbf{H}^{1}} \leq C h^{s}\Vert\mathbf{v}\Vert_{{\mathbf{H}}^{s}}, \; \forall \mathbf{v} \in \mathbf{H}_{\rm t}^{1}(\Omega)\cap\mathbf{H}^{r+1}(\Omega) .  \label{subeq:2.02}
\end{gather}
\end{subequations}

We approximate the scalar potential $\phi$ and the wave function $\Psi$ in $X_{h}^{r}$ and $ \mathcal{X}_{h}^{r} $ respectively, and find the approximate solution of the vector potential $\mathbf{A}$ in a subspace of $\mathbf{X}_{h}$:
\begin{equation}
\mathbf{X}_{0h} = \{ \mathbf{v}_{h} \in \mathbf{X}_{h} \;|\; (\nabla \cdot \mathbf{v}_{h}, \, q_{h} ) = 0, \; \forall q_{h} \in X_{h}^{1} \}.
\end{equation}
It is important to note that $\mathbf{X}_{0h} \nsubseteq \mathbf{H}^{1}_{t,0}(\Omega)$ since for each $\mathbf{v}_{h} \in \mathbf{X}_{0h}$, we only have $\rho_{h}(\nabla \cdot \mathbf{v}_{h}) = 0$, where $\rho_{h}$ is the orthogonal projection of $L^{2}(\Omega)$ onto $X_{h}^{1}.$

We now claim that there exists an interpolation operator ${\bm i}_{h}: \mathbf{H}^{1}_{t,0}(\Omega)\rightarrow \mathbf{X}_{0h}$, such that for every $\mathbf{v} \in \mathbf{H}^{1}_{t,0}(\Omega)\cap \mathbf{H}^{r+1}(\Omega)$,
\begin{equation}\label{eq:2-0}
\Vert \mathbf{v} - {\bm i}_{h}\mathbf{v}\Vert_{\mathbf{H}^{1}} \leq C h^{r}\Vert\mathbf{v}\Vert_{\mathbf{H}^{r+1}}.
\end{equation}
By the mixed finite element theory \cite{Brezzi,Gir}, we can ensure (\ref{eq:2-0}) by applying (\ref{subeq:2.02}) and the following discrete inf-sup condition: there exists a positive constant $\beta$, independent of $h$, such that 
\begin{equation}\label{eq:2-1}
\sup_{ \mathbf{v}_{h} \in \mathbf{X}_{h}} \frac{\big(\nabla \cdot \mathbf{v}_{h},\, q_h\big)}{ \Vert \mathbf{v}_{h} \Vert_{\mathbf{H}^{1}}} \geq \beta \,\Vert q_{h} \Vert_{L^{2}}, \quad  \forall \, q_{h} \in X_{h}^{1} .
\end{equation}
For $\widetilde{\mathbf{X}}_{h} = \big( Y^{2}_{h}\big)^{3}\cap \mathbf{H}^{1}_{0}(\Omega)$, $\widetilde{X}_{h} = Y^{1}_{h}$,  the following discrete inf-sup condition for Hood-Taylor element is proved in \cite{Brezzi} by Verf\"{u}rth's trick:
\begin{equation}\label{eq:2-1-0}
\sup_{ \mathbf{v}_{h} \in \widetilde{\mathbf{X}}_{h}} \frac{\big(\nabla \cdot \mathbf{v}_{h},\, q_h\big)}{ \Vert \mathbf{v}_{h} \Vert_{\mathbf{H}^{1}}} \geq \beta \,\Vert q_{h} \Vert_{L^{2}/{R}}, \quad  \forall \, q_{h} \in \widetilde{X}_{h} .
\end{equation}

The technique used in the proof of (\ref{eq:2-1-0}) can be applied directly to prove (\ref{eq:2-1}) by virtue of the fact that $\widetilde{\mathbf{X}}_{h} \subset \mathbf{X}_{h} $, $X_{h}^{1} \subset \widetilde{X}_{h} $  and the following continuous inf-sup condition:
\begin{equation}\label{eq:2-1-2}
\sup_{ \mathbf{v} \in \mathbf{H}_{\rm t}^{1}(\Omega) } \frac{\big(\nabla \cdot \mathbf{v},\, q\big)}{ \Vert \mathbf{v} \Vert_{\mathbf{H}^{1}}} \geq \beta \,\Vert q \Vert_{L^{2}}, \quad  \forall \, q \in L^{2}(\Omega).
\end{equation}
For more details, see \cite{Brezzi}. Thus (\ref{eq:2-0}) is verified.

Let ${\bm \pi}_{h}: \mathbf{H}_{\rm t}^{1}(\Omega) \rightarrow \mathbf{X}_{0h}$ be a Ritz projection as follows: $\forall\, \mathbf{A}\in\mathbf{H}_{\rm t}^{1}(\Omega) $, find ${\bm \pi}_{h}\mathbf{A} \in \mathbf{X}_{0h}$ such that 
\begin{equation}
\big(\nabla\cdot(\mathbf{A} - {\bm \pi}_{h}\mathbf{A} ), \, \nabla\cdot \mathbf{v}\big) + \big(\nabla\times(\mathbf{A} - {\bm \pi}_{h}\mathbf{A} ), \, \nabla\times \mathbf{v}\big) = 0,\quad \forall\,\mathbf{v}\in \mathbf{X}_{0h}.
\end{equation}
Owing to (\ref{eq:2-0}), we have the following error estimate of ${\bm \pi}_{h}$.
\begin{equation}\label{eq:2-2}
\Vert \mathbf{v} - {\bm \pi}_{h}\mathbf{v}\Vert_{\mathbf{H}^{1}} \leq C h^{r}\Vert\mathbf{v}\Vert_{\mathbf{H}^{r+1}}, \quad \forall \mathbf{v} \in \mathbf{H}^{1}_{t,0}(\Omega)\cap \mathbf{H}^{r+1}(\Omega).
\end{equation}

To define our fully discrete scheme, we divide the time interval $(0,T)$ into $M$ uniform subintervals using the nodal points
\begin{equation*}
0=t^{0} < t^{1} <\cdots<t^{M}=T,
\end{equation*}
with $t^{k} = k\tau$ and $\tau = T/M$. We denote $u^{k}=u(\cdot,t^{k})$ for any given functions $u\in C\big((0,T);\,W\big)$ with a Banach space $ W $. For a given sequence $\{u^{k}\}_{k=0}^{M}$, we introduce the following notation:
\begin{equation}\label{eq:2-3}
\begin{array}{@{}l@{}}
{\displaystyle \partial_{\tau} u^{k}= (u^{k}-u^{k-1})/\tau, \quad\partial_{\tau}^{2} u^{k}=(\partial_{\tau} u^{k}-\partial_{\tau} u^{k-1})/\tau,}\\[2mm]
{\displaystyle \overline{u}^{k} = (u^{k}+u^{k-1})/2, \quad \widetilde{u}^{k}=(u^{k}+u^{k-2})/2,}\\[2mm]
\end{array}
\end{equation}
 
For convenience, Let us assume that $ \mathbf{A}^{-1} $ is defined by
\begin{equation}\label{eq:2-4}
{\displaystyle \mathbf{A}^{-1}=\mathbf{A}(\cdot,0)-\tau \frac{\partial \mathbf{A}}{\partial t}(\cdot,0)=\mathbf{A}_{0}-\tau\mathbf{A}_{1}},
\end{equation}
which is an approximation of $\mathbf{A}(\cdot,-\tau)$ with second order accuracy. 

Using the above notation, we can formulate our first fully discrete finite element scheme for the M-S-C system as follows.

Scheme ({\rm \uppercase\expandafter{\romannumeral1}}).   $\quad$For $k=0 ,1,\cdots, M$, find $(\Psi_{h}^{k},\mathbf{A}_{h}^{k},\phi^{k}_{h})\in\mathcal{X}_{h}^{r}\times\mathbf{X}_{0h}\times X_{h}^{r}$ such that
\begin{equation}\label{eq:2-5}
\Psi_{h}^{0}=\mathcal{I}_{h}\Psi_{0},\quad\mathbf{A}_{h}^{0}={\bm \pi}_h \mathbf{A}_{0},\quad
 \mathbf{A}_{h}^{-1}=\mathbf{A}_{h}^{0}-\tau\, {\bm \pi}_h\mathbf{A}_{1},
\end{equation}
and for any $\varphi\in\mathcal{X}_{h}^{r}$, $\mathbf{v} \in \mathbf{X}_{0h}$, $u \in X_{h}^{r}$, the following equations hold:
\begin{equation}\label{eq:2-6}
\left\{
\begin{array}{@{}l@{}}
{\displaystyle
 -\mathrm{i}({\partial_{\tau}}\Psi_{h}^{k},\varphi)+
\frac{1}{2}\left((\mathrm{i}\nabla +\overline{\mathbf{A}}^{k}_{h})\overline{\Psi}_{h}^{k},(\mathrm{i}\nabla +\overline{\mathbf{A}}_{h}^{k})\varphi\right) +\left( (V+\overline{\phi}_{h}^{k})\overline{\Psi}_{h}^{k},\varphi\right) = 0, }\\[2mm]
{\displaystyle (\partial_{\tau}^{2}\mathbf{A}_{h}^{k},\mathbf{v})  + \big(\nabla\times
 \widetilde{\mathbf{A}}_{h}^{k},\nabla\times\mathbf{v}\big) +\big(\nabla \cdot
 \widetilde{\mathbf{A}}_{h}^{k},\nabla\cdot\mathbf{v}\big)+\big(|\Psi_{h}^{k-1}|^{2}\frac{\overline{\mathbf{A}}_{h}^{k} + \overline{\mathbf{A}}_{h}^{k-1}}{2},\, \mathbf{v}\big)}\\[2mm]
{\displaystyle\quad +\Big(\frac{\mathrm{i}}{2}\big((\Psi_{h}^{k-1})^{\ast}\nabla{\Psi_{h}^{k-1}}
-\Psi_h^{k-1}\nabla{(\Psi_{h}^{k-1})}^{\ast}\big),\mathbf{v}\Big) =0,}\\[2mm]
 {\displaystyle (\nabla {\phi}_{h}^{k}, \,\nabla u) = (\vert\Psi_{h}^{k}\vert^{2},\,u)}.
\end{array}
\right.
\end{equation}

\begin{remark}
In Section~\ref{sec-2} we prove that weak solutions in Definition~\ref{def:2.1} and Definition~\ref{def:2.2}
are equivalent. However, from the standpoint of finite element numerical computation, we choose to approximate weak solutions of the M-S-C system in the sense of Definition~\ref{def:2.2} instead of Definition~\ref{def:2.1} since it is very difficult to construct finite element subspaces of $\mathbf{H}^{1}_{t,0}(\Omega)$. We know that weak solutions of type \uppercase\expandafter{\romannumeral2} in Definition~\ref{def:2.2} imply that $\mathbf{A}$ is divergence-free although we only require $\mathbf{A}$ in $\mathbf{H}^{1}_{\rm t}(\Omega)$.  In the discrete level, if we approximate $\mathbf{A}$ in $\mathbf{X}_{h}$, it is difficult to degisn time integration schemes to ensure a discrete analogue of $\nabla\cdot \mathbf{A} = 0$, i.e. $P_{h}(\nabla\cdot\mathbf{A}_{h}) = 0$, where $P_{h}$ denotes the orthogonal projection of $L^{2}(\Omega)$ onto some finite element space. Thus we approximate $\mathbf{A}$ in $\mathbf{X}_{0h}$ to enforce the projection of $\nabla \cdot \mathbf{A}_{h}$ onto $X_{h}^{1}$ vanishes. Moreover, for the purpose of theoretical analysis, we add an extra term $\big(\nabla\cdot\widetilde{\mathbf{A}}_{h}^{k}, \,\nabla \cdot \mathbf{v}\big)$ to the discrete system (\ref{eq:2-6}).
It turns out that this term is indispensable to the proof of the error estimates and the existence of solutions to the discrete system (\ref{eq:2-6}).
\end{remark}

Apart from introducing the subspace $\mathbf{X}_{0h}$ of $\mathbf{X}_{h}$, we can also introduce a Lagrangian multiplier $p_{h}^{k}$ to relax the divergence-free constraint of $\mathbf{A}_{h}^{k}$ at each time step. We now give another fully discrete scheme based on the mixed finite element method as follows.
\begin{equation}
 {\rm Scheme } \;{\rm (\uppercase\expandafter{\romannumeral2}) }. \qquad {\rm Let} \quad  \Psi_{h}^{0}=\mathcal{I}_{h}\Psi_{0},\quad\mathbf{A}_{h}^{0}={\bm \pi}_h \mathbf{A}_{0},\quad
\mathbf{A}_{h}^{-1}=\mathbf{A}_{h}^{0}-\tau\, {\bm \pi}_h\mathbf{A}_{1},
\end{equation} 
and for $k=1 ,2 ,\cdots, M$, find $(\Psi_{h}^{k},\,\mathbf{A}_{h}^{k},\, p_{h}^{k}, \,\phi^{k}_{h})\in\mathcal{X}_{h}^{r}\times\mathbf{X}_{h}\times X_{h}^{1} \times X_{h}^{r}$ satisfying
\begin{equation}\label{eq:2-6-0}
\left\{
\begin{array}{@{}l@{}}
{\displaystyle
 -\mathrm{i}({\partial_{\tau}}\Psi_{h}^{k},\varphi)+
\frac{1}{2}\left((\mathrm{i}\nabla +\overline{\mathbf{A}}^{k}_{h})\overline{\Psi}_{h}^{k},(\mathrm{i}\nabla +\overline{\mathbf{A}}_{h}^{k})\varphi\right) +\left( (V+\overline{\phi}_{h}^{k})\overline{\Psi}_{h}^{k},\varphi\right) = 0, }\\[2mm]
{\displaystyle \qquad \qquad\qquad \qquad\qquad \qquad\qquad \qquad\qquad \quad  \qquad \quad\qquad    \quad \forall \varphi\in\mathcal{X}_{h}^{r},}\\[2mm]
{\displaystyle (\partial_{\tau}^{2}\mathbf{A}_{h}^{k},\mathbf{v})  + \big(\nabla\times
 \widetilde{\mathbf{A}}_{h}^{k},\nabla\times\mathbf{v}\big) +\big(\nabla \cdot
 \widetilde{\mathbf{A}}_{h}^{k},\nabla\cdot\mathbf{v}\big)+\big(|\Psi_{h}^{k-1}|^{2}\frac{\overline{\mathbf{A}}_{h}^{k} + \overline{\mathbf{A}}_{h}^{k-1}}{2},\, \mathbf{v}\big)}\\[2mm]
{\displaystyle\quad +\big(p_{h}^{k}, \, \nabla \cdot \mathbf{v}\big) +\Big(\frac{\mathrm{i}}{2}\big((\Psi_{h}^{k-1})^{\ast}\nabla{\Psi_{h}^{k-1}}
-\Psi_h^{k-1}\nabla{(\Psi_{h}^{k-1})}^{\ast}\big),\,\mathbf{v}\Big) =0, \quad \forall\mathbf{v}\in\mathbf{X}_{h},}\\[2mm]
{\displaystyle \big(\nabla \cdot \mathbf{A}_{h}^{k},\, q\big) = 0, \quad \forall\, q \in X_{h}^{1},    }\\[2mm]
 {\displaystyle (\nabla {\phi}_{h}^{k}, \,\nabla u) = (\vert\Psi_{h}^{k}\vert^{2},\,u),\quad\forall \;u \in X_{h}^{r}}.
\end{array}
\right.
\end{equation}

At each time step, the equation 
\begin{equation}\label{eq:2-6-1}
\begin{array}{@{}l@{}}
{\displaystyle (\partial_{\tau}^{2}\mathbf{A}_{h}^{k},\mathbf{v})  + \big(\nabla\times
 \widetilde{\mathbf{A}}_{h}^{k},\nabla\times\mathbf{v}\big) +\big(\nabla \cdot
 \widetilde{\mathbf{A}}_{h}^{k},\nabla\cdot\mathbf{v}\big)+\big(|\Psi_{h}^{k-1}|^{2}\frac{\overline{\mathbf{A}}_{h}^{k} + \overline{\mathbf{A}}_{h}^{k-1}}{2},\, \mathbf{v}\big)}\\[2mm]
{\displaystyle\quad +\Big(\frac{\mathrm{i}}{2}\big((\Psi_{h}^{k-1})^{\ast}\nabla{\Psi_{h}^{k-1}}
-\Psi_h^{k-1}\nabla{(\Psi_{h}^{k-1})}^{\ast}\big),\mathbf{v}\Big) =0, \quad \forall\mathbf{v}\in\mathbf{X}_{0h}}
\end{array}
\end{equation}
in scheme (\uppercase\expandafter{\romannumeral1}) and 
\begin{equation}\label{eq:2-6-2}
\left\{
\begin{array}{@{}l@{}}
{\displaystyle (\partial_{\tau}^{2}\mathbf{A}_{h}^{k},\mathbf{v})  + \big(\nabla\times
 \widetilde{\mathbf{A}}_{h}^{k},\nabla\times\mathbf{v}\big) +\big(\nabla \cdot
 \widetilde{\mathbf{A}}_{h}^{k},\nabla\cdot\mathbf{v}\big)+\big(|\Psi_{h}^{k-1}|^{2}\frac{\overline{\mathbf{A}}_{h}^{k} + \overline{\mathbf{A}}_{h}^{k-1}}{2},\, \mathbf{v}\big)}\\[2mm]
{\displaystyle\quad +\big(p_{h}^{k}, \, \nabla \cdot \mathbf{v}\big) +\Big(\frac{\mathrm{i}}{2}\big((\Psi_{h}^{k-1})^{\ast}\nabla{\Psi_{h}^{k-1}}
-\Psi_h^{k-1}\nabla{(\Psi_{h}^{k-1})}^{\ast}\big),\,\mathbf{v}\Big) =0, \quad \forall\mathbf{v}\in\mathbf{X}_{h},}\\[2mm]
{\displaystyle \big(\nabla \cdot \mathbf{A}_{h}^{k},\, q\big) = 0, \quad \forall\, q \in X_{h}^{1}  }
\end{array}
\right.
\end{equation}
in scheme (\uppercase\expandafter{\romannumeral2}) are decoupled from the other two equations, respectively. Due to  the discrete inf-sup condition (\ref{eq:2-1}) and the coercivity of the bilinear functional $a^{k}_{h}(\cdot,\, \cdot)$ in $\mathbf{X}_{h}$, where
\begin{equation}
\begin{array}{@{}l@{}}
{\displaystyle  a^{k}_{h}(\mathbf{u}, \,\mathbf{v}) = \big(\nabla\times
{\mathbf{u}},\nabla\times\mathbf{v}\big) +\big(\nabla \cdot
 {\mathbf{u}},\nabla\cdot\mathbf{v}\big)+ \frac{2}{\tau^{2}}(\mathbf{u},\mathbf{v})  } \\[2mm]
{\displaystyle   \qquad \qquad  + \frac{1}{2}\big(|\Psi_{h}^{k-1}|^{2}{\mathbf{u}},\, \mathbf{v}\big), \qquad \forall \, \mathbf{u}, \mathbf{v} \in \mathbf{X}_{h}, }
\end{array}
\end{equation} 
we know that there exists a unique solution to (\ref{eq:2-6-1}) and (\ref{eq:2-6-2}), respectively. It is easy to see that $\mathbf{A}_{h}^{k}$ in (\ref{eq:2-6-2}) satisfies (\ref{eq:2-6-1}) and thus the two above equations admit the same solution $\mathbf{A}_{h}^{k}$. Consequently, scheme (\uppercase\expandafter{\romannumeral1}) and scheme (\uppercase\expandafter{\romannumeral2}) are mathematically equivalent. However, scheme (\uppercase\expandafter{\romannumeral1}) is easier to perform theoretical analysis while scheme (\uppercase\expandafter{\romannumeral2}) is easier to carry out numerical computation. 

At each time step,  we first solve (\ref{eq:2-6-1}) or (\ref{eq:2-6-2}) and obtain $\mathbf{A}_{h}^{k}$. Then we substitute it into (\ref{eq:2-6-0}) and solve the nonlinear subsystem concerning $\Psi_{h}^{k}$ and $\phi_{h}^{k}$. The existence and uniqueness of solutions to this subsystem is proved in Section~\ref{sec-5}. In practical computations, we can apply the Picard simple iterative method or the Newton iterative method to solve the nonlinear subsystem.

For convenience, we define the following bilinear forms:
\begin{equation}\label{eq:2-7}
\begin{array}{@{}l@{}}
{\displaystyle B(\mathbf{A};\Psi,\varphi)=\left((\mathrm{i}\nabla+\mathbf{A})\Psi,(\mathrm{i}\nabla+\mathbf{A})\varphi\right),}\\[2mm]
{\displaystyle D(\mathbf{u},\,\mathbf{v})=(\nabla\cdot \mathbf{u},\,\nabla\cdot \mathbf{v})+
(\nabla\times\mathbf{u},\,\nabla\times\mathbf{v}),}\\[2mm]
{\displaystyle f(\Psi,\varphi)=\frac{\mathrm{i}}{2}(\varphi^{\ast}\nabla\Psi-\Psi\nabla\varphi^{\ast}).}
\end{array}
\end{equation}

Then (\ref{eq:2-6}) in scheme (\uppercase\expandafter{\romannumeral1})  can be rewritten as follows:
for $ k=1,2, \cdots,M$,
\begin{equation}\label{eq:2-8}
\left\{
\begin{array}{@{}l@{}}
{\displaystyle  -\mathrm{i}({\partial_{\tau}}\Psi_{h}^{k},\varphi)+
\frac{1}{2}B(\overline{\mathbf{A}}_{h}^{k};\,\overline{\Psi}_{h}^{k},\,\varphi)
 +( V\overline{\Psi}_{h}^{k},\varphi)+(\overline{\phi}_{h}^{k}\overline{\Psi}_{h}^{k},\varphi)  = 0 ,\quad \forall \varphi\in\mathcal{X}_{h}^{r} , } \\[2mm]
 {\displaystyle (\partial_{\tau} ^{2}\mathbf{A}_{h}^{k},\mathbf{v})+D(\widetilde{\mathbf{A}}_{h}^{k},\mathbf{v}) +\left( f(\Psi_{h}^{k-1},\Psi_h^{k-1}),\mathbf{v}\right) + \big(|\Psi_{h}^{k-1}|^{2}\frac{\overline{\mathbf{A}}_{h}^{k} + \overline{\mathbf{A}}_{h}^{k-1}}{2},\mathbf{v}\big)=0,}  \\[2mm]
 {\displaystyle \qquad \qquad\qquad \qquad\qquad \qquad\qquad \qquad\qquad \quad  \qquad \quad\qquad \quad   \forall\mathbf{v}\in\mathbf{X}_{0h} , } \\[2mm] 
{\displaystyle  (\nabla {\phi}_{h}^{k},\,\nabla u) = (\vert\Psi_{h}^{k}\vert^{2},\,u),\quad\forall u \in X_{h}^{r}.  } 
\end{array}
\right.
\end{equation}

In this paper we assume that the M-S-C system (\ref{eq:1.8})-(\ref{eq:1.9}) has one and only one weak solution $(\Psi,\mathbf{A},\phi)$ in the sense of Definition~\ref{def:2.2} and
the following regularity conditions are satisfied:
\begin{equation}\label{eq:2-9}
\begin{array}{@{}l@{}}
{\displaystyle
\Psi,\Psi_{t} \in {L}^{\infty}(0, T; \mathcal{H}^{r+1}(\Omega)),\quad \Psi_{tt} , \Psi_{ttt} \in {L}^{\infty}(0, T; \mathcal{H}^{1}(\Omega)),}\\[2mm]
{\displaystyle \qquad   \Psi_{tttt} \in L^{2}(0, T; \mathcal{L}^{2}(\Omega)), }\\[2mm]
{\displaystyle  \mathbf{A},\mathbf{A}_{t} \in {L}^{\infty}(0, T; \mathbf{H}^{r+1}(\Omega)) ,\quad \mathbf{A}_{tt} \in {L}^{\infty}(0, T; \mathbf{H}^{1}(\Omega))  }\\[2mm]
{\displaystyle\mathbf{A}_{ttt} \in {L}^{2}(0, T; \mathbf{H}^{1}(\Omega)),\;\; \mathbf{A}_{tttt} \in L^{2}(0, T; \mathbf{L}^{2}(\Omega)),}\\[2mm]
{\displaystyle \phi,\phi_{t}\in {L}^{\infty}(0, T; {H}^{r+1}(\Omega)), \;\;\phi_{tt} \in {L}^{\infty}(0, T; {H}^{1}(\Omega)),}\\[2mm]
{\displaystyle \phi_{ttt}\in {L}^{\infty}(0, T; {L}^{2}(\Omega)) , \quad \phi_{tttt}\in {L}^{2}(0, T; {L}^{2}(\Omega)).  }
\end{array}
\end{equation}
For the initial conditions $(\Psi_{0},\mathbf{A}_{0},\mathbf{A}_{1})$ , we assume that
\begin{equation}\label{eq:2-10}
\begin{array}{@{}l@{}}
{\displaystyle \Psi_{0}\in \mathcal{H}^{r+1}(\Omega) \cap \mathcal{H}_{0}^{1}(\Omega), \,\,\, \mathbf{A}_{0},\mathbf{A}_{1}\in\mathbf{H}^{r+1}(\Omega)\cap\mathbf{H}_{t,0}^{1}(\Omega).}
\end{array}
\end{equation}

We now give the main convergence result in this paper as follows:
\begin{theorem}\label{thm2-1}
Suppose that $ \Omega\subset \mathbb{R}^3$ is a bounded Lipschitz polyhedral convex domain.
Let $(\Psi,\mathbf{A},\phi)$ be the unique solution to the M-S-C system (\ref{eq:1.8})-(\ref{eq:1.9}) , and let
$ (\Psi_h^k,\mathbf{A}_h^k,\phi_{h}^{k})$ be the numerical solution to the discrete system (\ref{eq:2-5})-(\ref{eq:2-6}). Under the assumptions
(\ref{eq:2-9}) and (\ref{eq:2-10}), we have the following error estimates
\begin{equation}\label{eq:2-11}
\begin{array}{@{}l@{}}
 {\displaystyle \max_{1\leq k \leq M}\Big[\|\Psi_{h}^{k}-\Psi^{k}\|_{\mathcal{H}^1(\Omega)}^{2}
 + \|\mathbf{A}_{h}^{k}-\mathbf{A}^{k}\|_{\mathbf{H}^1(\Omega)}^{2} + \|\phi_{h}^{k}-\phi^{k}\|^{2}_{{H}^{1}(\Omega)} \Big]\leq C (h^{2r}+{\tau}^{4}),}
\end{array}
\end{equation}
where $ \Psi^k=\Psi(\cdot, t^k) $, $ \mathbf{A}^k=\mathbf{A}(\cdot, t^k) $, $ \phi^k=\phi(\cdot, t^k) $, $r\leq 2$, and $C$ is a constant independent of $h$ and $\tau$.
\end{theorem}

\section{Stability estimates}\label{sec-4}
In this section we first show the discrete system (\ref{eq:2-5})-(\ref{eq:2-6}) maintains the conservation of the total charge and energy. Then we deduce some stability estimates of the discrete solutions, which will be used to derive the error estimates in next section.

First we define the energy of the discrete system (\ref{eq:2-5})-(\ref{eq:2-6}) as follows:

\begin{equation}\label{eq:3.3}
\begin{array}{@{}l@{}}
{\displaystyle \mathcal{E}^{k}_{h}  = \frac{1}{2}B(\overline{\mathbf{A}}^{k}_{h}; \,  \Psi^{k}_{h},\, \Psi^{k}_{h}) + (V\Psi^{k}_{h}, \,\Psi^{k}_{h}) +\frac{1}{2}\Vert\nabla \phi^{k}_{h}\Vert_{\mathbf{L}^{2}}^{2} + \frac{1}{2}\Vert\partial_{\tau} \mathbf{A}_{h}^{k}\Vert_{\mathbf{L}^{2}}^{2} } \\[2mm]
{\displaystyle \qquad \qquad \qquad + \frac{1}{4} D(\mathbf{A}_{h}^{k},\,\mathbf{A}_{h}^{k}) +  \frac{1}{4} D(\mathbf{A}_{h}^{k-1},\,\mathbf{A}_{h}^{k-1}).  }
\end{array}
\end{equation}

Lemma~\ref{lem3-1} and Theorem~\ref{thm3-1} in the following are the discrete analogues of Lemma~\ref{lem6-4} and Theorem~\ref{thm6-0}, respectively. 

\begin{lemma}\label{lem3-1}
For $k=1,2\cdots,M$, the solution $(\Psi^{k}_{h}, \mathbf{A}_{h}^{k}, \phi_{h}^{k})$ of the discrete system (\ref{eq:2-5})-(\ref{eq:2-6}) satisfies
\begin{equation}\label{eq:3.4}
{\Vert \Psi_{h}^{k} \Vert}_{\mathcal{L}^2}^{2}={\Vert \Psi_{h}^{0} \Vert}_{\mathcal{L}^2}^{2},
\quad \mathcal{E}^{k}_{h} = \mathcal{E}^{0}_{h}.
\end{equation}
\end{lemma}
\begin{proof}
The proof the this lemma is very similar to its continuous counterpart. For $(\ref{eq:3.4})_{1}$ we can simply choose $\varphi = \overline{\Psi}_{h}^{k}$ in $(\ref{eq:2-8})_1$ and take its imaginary part.

To prove $(\ref{eq:3.4})_{2}$, we first notice that
\begin{equation}
\begin{array}{lll}
{\displaystyle \mathrm{Re}\left[B\left(\overline{\mathbf{A}}_{h}^{k};\overline{\Psi}_{h}^{k},\partial_{\tau} \Psi_{h}^{k}\right)\right]=\frac{1}{2}\partial_{\tau} B(\overline{\mathbf{A}}_{h}^{k};{\Psi}_{h}^{k},\Psi_{h}^{k}) }\\[2mm]
{\displaystyle \quad\quad+ \frac{1}{2\tau}\left[B(\overline{\mathbf{A}}_{h}^{k-1};{\Psi}_{h}^{k-1},\Psi_{h}^{k-1})-B(\overline{\mathbf{A}}_{h}^{k};{\Psi}_{h}^{k-1},\Psi_{h}^{k-1})\right]}\\[2mm]
{\displaystyle \quad\quad+ \frac{1}{2 \tau}\mathrm{Re}\left[B(\overline{\mathbf{A}}_{h}^{k};\Psi_{h}^{k-1},\Psi_{h}^{k})-B(\overline{\mathbf{A}}_{h}^{k};\Psi_{h}^{k},\Psi_{h}^{k-1})\right].}
\end{array}
\end{equation}
We also have the following identities by direct calculations
\begin{equation}\label{eq:3.6}
\begin{array}{lll}
{\displaystyle B(\mathbf{A};\psi,\varphi)=(\nabla\psi,\nabla\varphi)+(\mathbf{A}\psi,\mathbf{A}\varphi)+2(f(\psi,\varphi),\mathbf{A}),}\\[2mm]
{\displaystyle B(\mathbf{A};\psi,\varphi)-B(\tilde{\mathbf{A}};\psi,\varphi)=\left((\mathbf{A}+\tilde{\mathbf{A}})\psi \varphi^{*},\mathbf{A}-\tilde{\mathbf{A}}\right)+2(f(\psi,\varphi),\mathbf{A}-\tilde{\mathbf{A}}),}
\end{array}
\end{equation}
from which we deduce
\begin{equation}
\mathrm{Re}\left[B(\overline{\mathbf{A}}_{h}^{k};\Psi_{h}^{k-1},\Psi_{h}^{k})-B(\overline{\mathbf{A}}_{h}^{k};\Psi_{h}^{k},\Psi_{h}^{k-1})\right] = 0.
\end{equation}

Thus we get
\begin{equation}\label{eq:3.7}
\begin{array}{lll}
{\displaystyle \mathrm{Re}\left[B\left(\overline{\mathbf{A}}_{h}^{k};\overline{\Psi}_{h}^{k},\partial_{\tau} \Psi_{h}^{k}\right)\right]=\frac{1}{2}\partial_{\tau} B(\overline{\mathbf{A}}_{h}^{k};{\Psi}_{h}^{k},\Psi_{h}^{k})}\\[2mm]
{\displaystyle - \left(|\Psi_{h}^{k-1}|^{2}\frac{\overline{\mathbf{A}}_{h}^{k}+\overline{\mathbf{A}}_{h}^{k-1}}{2},\frac{\overline{\mathbf{A}}_{h}^{k}-\overline{\mathbf{A}}_{h}^{k-1}}{\tau}\right)-\left(f(\Psi_{h}^{k-1},\Psi_{h}^{k-1}),\frac{\overline{\mathbf{A}}_{h}^{k}-\overline{\mathbf{A}}_{h}^{k-1}}{\tau}\right).}
\end{array}
\end{equation}

Also we have
\begin{equation}\label{eq:3.8}
\mathrm{Re}\left[\big(V\overline{\Psi}_{h}^{k},\partial_{\tau} \Psi_{h}^{k}\big)\right]=\frac{1}{2}\partial_{\tau}\big(V\Psi_{h}^{k},\Psi_{h}^{k}\big),\quad \mathrm{Re}\left[\big(\overline{\phi}_{h}^{k}\overline{\Psi}_{h}^{k},\partial_{\tau} \Psi_{h}^{k}\big)\right]=\frac{1}{2}\big(\overline{\phi}_{h}^{k}, \,\partial_{\tau} |\Psi_{h}^{k}|^{2}\big).
\end{equation}

By choosing $\varphi=\partial_{\tau}\Psi_{h}^{k}$ in $(\ref{eq:2-8})_1$, taking the real part of the equation and combining with (\ref{eq:3.7}) and (\ref{eq:3.8}), we get
\begin{equation}\label{eq:3.9}
\begin{array}{lll}
{\displaystyle \frac{1}{2}\partial_{\tau}\Vert\left(\mathrm{i}\nabla+\overline{\mathbf{A}}_{h}^{k}\right)\Psi_{h}^{k}\Vert_{\mathbf{L}^2}^{2} + \partial_{\tau}\big(V\Psi_{h}^{k},\Psi_{h}^{k}\big) + \big(\overline{\phi}_{h}^{k}, \,\partial_{\tau} |\Psi_{h}^{k}|^{2}\big)  }\\[2mm]
{\displaystyle  - \left(|\Psi_{h}^{k-1}|^{2}\frac{\overline{\mathbf{A}}_{h}^{k}+\overline{\mathbf{A}}_{h}^{k-1}}{2},\frac{\overline{\mathbf{A}}_{h}^{k}-\overline{\mathbf{A}}_{h}^{k-1}}{\tau}\right)-\left(f(\Psi_{h}^{k-1},\Psi_{h}^{k-1}),\frac{\overline{\mathbf{A}}_{h}^{k}-\overline{\mathbf{A}}_{h}^{k-1}}{\tau}\right)=0.}
\end{array}
\end{equation}

Next by taking $\mathbf{v}=\frac{1}{2\tau}(\mathbf{A}_{h}^{k}-\mathbf{A}_{h}^{k-2})$ and adding it to (\ref{eq:3.9}), we obtain
\begin{equation}\label{eq:3.10}
\begin{array}{lll}
{\displaystyle \partial_{\tau} \left(\frac{1}{2}  B(\overline{\mathbf{A}}_{h}^{k};\Psi_{h}^{k},\Psi_{h}^{k})+\big(V\Psi_{h}^{k},\Psi_{h}^{k}\big) + \frac{1}{2}\Vert\partial_{\tau} \mathbf{A}_{h}^{k}\Vert_{\mathbf{L}^2}^{2}\right)}\\[2mm]
{\displaystyle \quad+\partial_{\tau}\left(\frac{1}{4} \Vert \nabla\times\mathbf{A}_{h}^{k}\Vert_{\mathbf{L}^{2}}^{2} +\frac{1}{4}\Vert \nabla\times\mathbf{A}_{h}^{k-1}\Vert_{\mathbf{L}^{2}}^{2}\right)+\big(\overline{\phi}_{h}^{k}, \,\partial_{\tau} |\Psi_{h}^{k}|^{2}\big) = 0.}
\end{array}
\end{equation}
Finally it is easy to deduce the following equation from the last equation of $(\ref{eq:2-8})$:
\begin{equation}\label{eq:3.11}
 (\nabla \,\partial_{\tau}{\phi}_{h}^{k},\,\nabla u) = (\partial_{\tau}\vert\Psi_{h}^{k}\vert^{2},\,u),\quad\forall u \in X_{h}^{r}.
\end{equation} 
 Take $u = \overline{\phi}_{h}^{k}$ in (\ref{eq:3.11}), insert it into (\ref{eq:3.10}) and we complete the proof of $(\ref{eq:3.4})_{2}$.
\end{proof}

\begin{theorem}\label{thm3-1}
The solution of the discrete system (\ref{eq:2-7}) fulfills the following estimate
\begin{equation}\label{eq:3.12}
 \Vert\Psi_{h}^{k}\Vert_{\mathcal{H}^1} + \Vert\partial_{\tau} \mathbf{A}_{h}^{k}\Vert_{\mathbf{L}^2} + \Vert\mathbf{A}_{h}^{k}\Vert_{\mathbf{H}^1}+\Vert \phi_{h}^{k} \Vert_{H^1}   \leq C,
\end{equation}
where $C$ is independent of $h$, $\tau$ and $k$.
\end{theorem}

The proof of this theorem is very similar to its continuous counterpart,  i.e. Theorem~\ref{thm6-0},  and thus we omit the proof.

\section{Solvability of the discrete system}\label{sec-5} In this section we consider the existence and uniqueness of the solutions to the discrete system (\ref{eq:2-5})-(\ref{eq:2-6}). To prove it, we first introduce a useful lemma in \cite{Browder} as follows.
\begin{lemma}\label{lem5-0}
Let $\big(H, \langle\cdot,\cdot\rangle   \big)$ be a finite-dimensional inner product space, $\Vert \cdot \Vert_{H}$ be the associated norm, and $g:H\longrightarrow H$ be continuous. Assume that
\begin{equation*}
\exists \,\alpha > 0,\quad \forall z \in H,\;\; \Vert z \Vert_{H}  = \alpha, \;\;{\rm Re}\langle g(z), z\rangle >0.
\end{equation*}
Then there exists a $z_0\in H$ such that $g(z_0) = 0$ and $\Vert z_0\Vert_{H}\leq \alpha$.
\end{lemma}

\begin{theorem}\label{thm5-0}
For any $\big(\Psi_{h}^{k-1}, \mathbf{A}_{h}^{k-2}, \mathbf{A}_{h}^{k-1}, \phi_{h}^{k-1}\big)$ satisfies (\ref{eq:3.12}), there exists a solution $\big(\Psi_{h}^{k}, \mathbf{A}_{h}^{k}, \phi_{h}^{k}\big)$ to the discrete system (\ref{eq:2-6}). Furthermore, if the time step $\tau$ is sufficiently small, the solution is unique.
\end{theorem}
\begin{proof}
As noted in Section~\ref{sec-3}, to solve the discrete system (\ref{eq:2-5})-(\ref{eq:2-6}), we need to solve (\ref{eq:2-6-1}) and the following subsystem alternately. 
\begin{equation}\label{eq:5.1}
\left\{
\begin{array}{@{}l@{}}
{\displaystyle
 -\mathrm{i}({\partial_{\tau}}\Psi_{h}^{k},\varphi)+
\frac{1}{2}\left((\mathrm{i}\nabla +\overline{\mathbf{A}}^{k}_{h})\overline{\Psi}_{h}^{k},(\mathrm{i}\nabla +\overline{\mathbf{A}}_{h}^{k})\varphi\right)
 +\left( (V+\overline{\phi}_{h}^{k})\overline{\Psi}_{h}^{k},\varphi\right) = 0 ,}\\[2mm]
 {\displaystyle \qquad \qquad\qquad \qquad\qquad \qquad\qquad \qquad\qquad \quad  \qquad \quad\qquad    \quad \forall \varphi\in\mathcal{X}_{h}^{r},}\\[2mm]
 {\displaystyle (\nabla {\phi}_{h}^{k}, \,\nabla u) = (\vert\Psi_{h}^{k}\vert^{2},\,u),\quad\forall u \in X_{h}^{r}},
\end{array}
\right.
\end{equation}
Since we have proved the solvability of (\ref{eq:2-6-1}) in Section~\ref{sec-3}, we only need to consider the solvability of (\ref{eq:5.1}), which can be rewritten as follows:
\begin{equation}\label{eq:5.2}
\left\{
\begin{array}{@{}l@{}}
{\displaystyle
(\overline{\Psi}_{h}^{k},\,\varphi) = ({\Psi}_{h}^{k-1},\varphi)
-{\rm i}\frac{\tau}{4}B(\overline{\mathbf{A}}_{h}^{k};\,\overline{\Psi}_{h}^{k},\,\varphi)
 -{\rm i}\frac{\tau}{2}\left( (V+\overline{\phi}_{h}^{k})\overline{\Psi}_{h}^{k},\,\varphi\right), \forall \varphi\in\mathcal{X}_{h}^{r}}\\[2mm]
 {\displaystyle (\nabla \overline{\phi}_{h}^{k}, \,\nabla u) =\big(\frac{1}{2}(|\Psi_{h}^{k-1}|^{2} + |2\overline{\Psi}_{h}^{k} - \Psi_{h}^{k-1}|^{2}),\, u\big),\quad\forall u \in X_{h}^{r}}.
\end{array}
\right.
\end{equation}

For a given $h$, assume that $\{u_1, u_2,\cdots,u_N\}$ is a basis of $X_{h}^{r}$. Note that $\mathcal{X}_{h}^{r} = X_{h}^{r} \oplus {\rm i}X_{h}^{r}$. Then $\overline{\phi}_{h}^{k}$, $\overline{\Psi}_{h}^{k}$, and $\Psi_{h}^{k-1}$ can be written as follows 
\begin{equation}  
\overline{\phi}_{h}^{k} = \sum_{j=1}^{N}{a_{j}u_j},\quad \overline{\Psi}_{h}^{k} = \sum_{j=1}^{N}{b_{j}u_j},\quad {\Psi}_{h}^{k-1} = \sum_{j=1}^{N}{c_{j}u_j},
\end{equation}
where 
\begin{equation}
\vec{a} = (a_1,\cdots, a_N)^{T} \in \mathbb{R}^{N}, \, \vec{b} = (b_1,\cdots, b_N)^{T} \in \mathbb{C}^{N},\, \vec{c} = (c_1,\cdots, c_N)^{T} \in \mathbb{C}^{N}.
\end{equation}
Denote by
\begin{equation}\label{eq:5.3}
\begin{array}{@{}l@{}}
{\displaystyle W = \big( w_{ij} \big) \in \mathbb{R}^{N\times N}, \quad\vec{r} = (r_1,\cdots, r_N)^{T}\in \mathbb{R}^{N}, }\\[2mm]
{\displaystyle S = \big(s_{ij}\big)\in \mathbb{R}^{N\times N}, \quad  Q = \big(q_{ij}\big)\in \mathbb{R}^{N\times N},  }
\end{array}
\end{equation}
where
\begin{equation}\label{eq:5.4}
\begin{array}{@{}l@{}}
{\displaystyle w_{ij} = (u_i, u_j), \;r_i=\frac{1}{2}\big(|\Psi_{h}^{k-1}|^{2} + |2\overline{\Psi}_{h}^{k} - \Psi_{h}^{k-1}|^{2},\, u_i\big), }\\[2mm]
{\displaystyle  s_{ij} = B(\overline{\mathbf{A}}_{h}^{k};\, u_i,\,u_j), \quad q_{ij} = \big( (V + \overline{\phi}_{h}^{k})u_i, \, u_j\big) \quad i,j = 1,\cdots, N. }
\end{array}
\end{equation}
Using the above notation, we can write (\ref{eq:5.2}) in the form of matrix: 
\begin{equation}\label{eq:5.5}
W\vec{b} = W\vec{c}-{\rm i}\frac{\tau}{4}S\vec{b} - {\rm i}\frac{\tau}{2}Q\vec{b}, \quad W\vec{a} = \vec{r},\quad Q = Q(\vec{a}),  \quad \vec{r} = \vec{r}(\vec{b}),
\end{equation}
or a more compact form
\begin{equation}\label{eq:5.6}
\vec{b} = \vec{c}-{\rm i}\frac{\tau}{4}W^{-1}S\vec{b} - {\rm i}\frac{\tau}{2}W^{-1}Q(\vec{b})\vec{b}.
\end{equation}

Now we define a finite dimensional space $\big(H, \langle\cdot,\cdot\rangle   \big)$ and a mapping $g: H \longrightarrow H$ as following:
\begin{equation}
\begin{array}{@{}l@{}}
{\displaystyle H = \mathbb{C}^{N}, \quad \langle \vec{e}_{1} , \vec{e}_{2}\rangle = (\vec{e}_{2})^{\ast}W\vec{e}_{1}, \quad \forall \vec{e}_{1}, \; \vec{e}_{2} \in H,   }\\[2mm]
{\displaystyle g(\vec{e}) = \vec{e}-\vec{c}+{\rm i}\frac{\tau}{4}W^{-1}S\vec{e} + {\rm i}\frac{\tau}{2}W^{-1}Q(\vec{e})\vec{e}, \quad \forall \vec{e}\in H, }
\end{array}
\end{equation}
where $(\vec{e}_{2})^{\ast} $ denotes the conjugate transpose of $\vec{e}_{2}$ and $W$, $S$ and $Q$ are defined in (\ref{eq:5.3})-(\ref{eq:5.4}). Obviously $g$ is continuous. Moreover,
\begin{equation}
{\rm Re}\langle g(\vec{e}),\,\vec{e}\rangle = \Vert \vec{e} \Vert_{H} ^{2}- {\rm Re}\langle \vec{c},\, \vec{e}\rangle \geq \Vert \vec{e} \Vert_{H}\big(\Vert \vec{e} \Vert_{H} - \Vert \vec{c} \Vert_{H} \big), \quad \forall \,\vec{e} \in H,
\end{equation}
which implies that 
\begin{equation}
{\rm Re}\langle g(\vec{e}),\,\vec{e}\rangle > 0 , \quad {\rm if} \;\Vert \vec{e} \Vert_{H} = \Vert \vec{c} \Vert_{H} +1.
\end{equation}
Thus the existence of $\vec{b}$ and $\vec{a}=W^{-1}\vec{r}(\vec{b})$ for (\ref{eq:5.6}) follows from Lemma~\ref{lem5-0}. Combining with the existence of $\mathbf{A}_{h}^{k}$, we obtain the existence of $\big(\Psi_{h}^{k}, \mathbf{A}_{h}^{k}, \phi_{h}^{k}\big)$ to the discrete system (\ref{eq:2-6}).

Now we study the uniqueness of the solutions to (\ref{eq:5.1}). Let $(\Psi_{h}^{k}, \phi_{h}^{k})$ and $(\hat{\Psi}_{h}^{k}, \hat{\phi}_{h}^{k})$ be two solutions of (\ref{eq:5.1}). Set $\eta = \Psi_{h}^{k} - \hat{\Psi}_{h}^{k},\; \psi = \phi_{h}^{k} - \hat{\phi}_{h}^{k}.$ They satisfy
\begin{equation}\label{eq:5.7}
\left\{
\begin{array}{@{}l@{}}
{\displaystyle
 -\mathrm{i}(\eta,\varphi)+
\frac{\tau}{4}B(\overline{\mathbf{A}}^{k}_{h};\,\eta,\,\varphi) +\frac{\tau}{2}\big((V+\overline{\phi}_{h}^{k})\eta,\,\varphi\big) + \frac{\tau}{4} \big((\hat{\Psi}_{h}^{k} + {\Psi}_{h}^{k-1})\psi,\,\varphi\big)
  = 0, }\\[2mm]
  {\displaystyle \qquad \qquad\qquad \qquad\qquad \qquad\qquad \qquad\qquad \quad  \qquad \quad\qquad    \quad \forall \varphi\in\mathcal{X}_{h}^{r},}\\[2mm]
 {\displaystyle (\nabla \psi, \,\nabla u) = ( (\Psi_{h}^{k})^{\ast}\eta + {\eta}^{\ast}\hat{\Psi}_{h}^{k} ,\,u),\quad\forall u \in X_{h}^{r}},
\end{array}
\right.
\end{equation}
By choosing $\varphi = \eta$ in the first equation of (\ref{eq:5.7}) and taking its imaginary part, we obtain
\begin{equation}\label{eq:5.8}
\Vert \eta \Vert_{\mathcal{L}^{2}} \leq C \tau \Vert\psi\Vert_{L^6}.
\end{equation}
Take $u=\psi$ in the second equation of (\ref{eq:5.7}) and we get
\begin{equation}\label{eq:5.9}
\Vert\psi\Vert_{L^6} \leq C \Vert\nabla \psi\Vert_{\mathbf{L}^2} \leq C\Vert \eta \Vert_{\mathcal{L}^{2}}.
\end{equation}
By substituting (\ref{eq:5.9}) into (\ref{eq:5.8}) and taking $\tau$ sufficiently small, we find $\eta = 0$ and consequently obtain the uniqueness of the solutions. \quad \end{proof}
\section{The error estimates}\label{sec-6}
We now turn to the proof of Theorem~\ref{thm2-1}. Set $e_{\Psi}=\mathcal{I}_{h}\Psi-\Psi$,  $e_{\mathbf{A}}={\bm \pi}_{h}\mathbf{A}-\mathbf{A}$ and $e_{\phi} = I_{h}\phi-\phi$, where $\mathcal{I}_{h}$, ${\bm \pi}_{h}$ and $I_{h}$ are defined in Section~\ref{sec-3}. From the regularity assumption and the interpolation error estimates (\ref{subeq:2.00}), (\ref{subeq:2.01}) and (\ref{eq:2-2}),
we deduce
\begin{equation}\label{eq:4-0}
\Vert e_{\Psi} \Vert_{\mathcal{H}^{1}} + \Vert e_{\mathbf{A}} \Vert_{\mathbf{H}^{1}} + \Vert e_{\phi} \Vert_{{H}^{1}} \leq Ch^{r}.
\end{equation}

By using standard finite element theory \cite{Bre} and the regularity assumption (\ref{eq:2-9}), we also have 
\begin{equation}\label{eq:4-1}
\Vert {\mathcal{I}}_{h}\Psi \Vert_{\mathcal{L}^{\infty}} + \Vert \nabla {\mathcal{I}}_{h}\Psi \Vert_{\mathbf{L}^{3}} + \Vert {\bm \pi}_{h}\mathbf{A} \Vert_{\mathbf{H}^{1}}\leq C .
\end{equation}

In the rest of this paper,  we need the following discrete integration by parts formulas. 
\begin{equation}\label{eq:4-2}
\begin{array}{lll}
{\displaystyle \sum_{k=1}^{M}{(a_{k}-a_{k-1})b_{k}}=a_{M}b_{M}-a_{0}b_{1}-\sum_{k=1}^{M-1}{a_{k}(b_{k+1}-b_{k})}.}\\[2mm]
{\displaystyle \sum_{k=1}^{M}{(a_{k}-a_{k-1})b_{k}}=a_{M}b_{M}-a_{0}b_{0}-\sum_{k=1}^{M}{a_{k-1}(b_{k}-b_{k-1})}.}\\[2mm]
\end{array}
\end{equation}
To simplify the notation , we denote by $\theta_{\Psi}^{k}=\Psi_{h}^{k}-\mathcal{I}_{h}\Psi^{k}$, $\theta_{\mathbf{A}}^{k}=\mathbf{A}_{h}^{k}-{\bm \pi}_{h}\mathbf{A}^{k}$ and $\theta_{\phi}^{k}=\phi_{h}^{k}-I_{h}\phi^{k}$. In view of the interpolation error estimates (\ref{eq:4-0}), we only need to prove that for $k =1,2,\cdots, M$, there holds
\begin{equation}\label{eq:4-2-0}
\Vert \theta_{\Psi}^{k} \Vert_{\mathcal{H}^{1}} + \Vert \theta_{\mathbf{A}}^{k} \Vert_{\mathbf{H}^{1}} + \Vert \theta_{\phi}^{k} \Vert_{H^{1}} \leq C(h^{r} + \tau^{2}).
\end{equation}

By assuming that 
\begin{equation}
\frac{\partial \Psi}{\partial t}  \in L^{2}(0,T;\mathcal{L}^{2}(\Omega)),\quad \frac{\partial ^{2} \mathbf{A}}{\partial t^{2}} \in L^{2}(0,T;\mathbf{L}^{2}(\Omega)),
\end{equation} 
the weak solution $(\Psi, \mathbf{A}, \phi)$ of the M-S-C system in sense of Definition~\ref{def:2.2} satisfies 
\begin{equation}\label{eq:4-3}
\left\{
\begin{array}{@{}l@{}}
{\displaystyle  -{\rm i}\big(\frac{\partial \Psi}{\partial t},\, \varphi\big) + \frac{1}{2}B\big(\mathbf{A}; \Psi, \,\varphi\big) + \big(V\Psi, \, \varphi\big) + \big(\phi\Psi,\, \varphi\big) = 0, \;\; \forall \varphi \in \mathcal{H}_{0}^{1}(\Omega), } \\[2mm]
{\displaystyle \big(\frac{\partial ^{2} \mathbf{A}}{\partial t^{2}},  \, \mathbf{v}\big) + \big(\nabla \times\mathbf{A},\,\nabla \times\mathbf{v}\big) - \big(\frac{\partial \phi}{\partial t}, \,\nabla \cdot \mathbf{v}\big) + \big(f(\Psi, \Psi), \, \mathbf{v}\big)  + \big(|\Psi|^{2}\mathbf{A}, \mathbf{v}\big) = 0 ,}\\[2mm]
{\displaystyle \qquad\qquad \qquad \qquad\qquad \qquad \qquad\qquad \qquad \qquad \qquad \qquad \qquad  \forall \mathbf{v}\in \mathbf{H}_{\rm t}^{1}(\Omega), } \\[2mm]
{\displaystyle \big(\nabla \phi,\, \nabla u\big) = \big(|\Psi|^{2},\, u\big),  \quad \forall u \in H_{0}^{1}(\Omega).}
\end{array}
\right.
\end{equation}
Subtracting the discrete sysytem (\ref{eq:2-8}) from (\ref{eq:4-3}) and noting that $\nabla \cdot \mathbf{A}^{k-1} = 0$, we have
\begin{equation}\label{eq:4-4}
\begin{array}{@{}l@{}}
{\displaystyle -2\mathrm{i}\big(\partial_{\tau} \theta^{k}_{\Psi},\,\varphi\big)+B\big(\overline{\mathbf{A}}^{k}_{h};\,\overline{\theta}_{\Psi}^{k}, \, \varphi\big)
= 2\mathrm{i}\Big(\partial_{\tau} \mathcal{I}_{h}\Psi^{k}-(\Psi_t)^{k-\frac{1}{2}},\, \varphi\Big)}\\[2mm]
{\displaystyle \quad+2\big(V(\Psi^{k-\frac{1}{2}}-\overline{\Psi}_{h}^{k}),\,\varphi\big)+ B\big(\mathbf{A}^{k-\frac{1}{2}}; \, (\Psi^{k-\frac{1}{2}}-\mathcal{I}_{h}\overline{\Psi}^{k}), \, \varphi\big)}\\[2mm]
{\displaystyle \quad +2\big(\phi^{k-\frac12}\Psi^{k-\frac12}-\overline{\phi}^{k}_{h}\overline{\Psi}^{k}_{h}, \, \varphi\big)+\left(B\big(\mathbf{A}^{k-\frac{1}{2}};\,\mathcal{I}_{h}\overline{\Psi}^{k},\varphi\big)-B\big(\overline{\mathbf{A}}^{k}_{h}; \, \mathcal{I}_{h}\overline{\Psi}^{k},\varphi\big)\right)} \\
{\displaystyle \; := \sum_{i=1}^{5}V_{i}^{k}(\varphi), \quad \forall \varphi\in\mathcal{X}_{h}^{r},   }
\end{array}
\end{equation}
\begin{equation}\label{eq:4-5}
\begin{array}{@{}l@{}}
{\displaystyle \big(\partial^{2}_{\tau}\theta^{k}_{\mathbf{A}}, \, \mathbf{v}\big)+D\big(\widetilde{\theta}^{k}_{\mathbf{A}},\,\mathbf{v}\big)
= \Big((\mathbf{A}_{tt})^{k-1}-\partial^{2}_{\tau} {\bm \pi}_{h}\mathbf{A}^{k},\, \mathbf{v}\Big) +D\big(\mathbf{A}^{k-1}-\widetilde{{\bm \pi}_{h}\mathbf{A}}^{k},\, \mathbf{v}\big)}\\[2mm]
{\displaystyle \qquad \quad -\Big((\phi_t)^{k-1},\, \nabla \cdot \mathbf{v}\Big) + \Big(|\Psi^{k-1}|^{2}\mathbf{A}^{k-1}-|\Psi^{k-1}_{h}|^{2}\frac{\overline{\mathbf{A}}_{h}^{k} + \overline{\mathbf{A}}_{h}^{k-1}}{2},
\;\mathbf{v}\Big)}\\[2mm]
{\displaystyle \qquad \quad+ \left(f(\Psi^{k-1},\Psi^{k-1})-f(\Psi^{k-1}_{h},\Psi^{k-1}_{h}),\;\mathbf{v}\right)}\\[2mm]
{\displaystyle \quad :=\sum_{i=1}^{5}U_{i}^{k}(\mathbf{v}), \quad \forall\mathbf{v}\in\mathbf{X}_{0h},}
\end{array}
\end{equation}
\begin{align}\label{eq:4-6}
& \big(\nabla {\theta}^{k}_{\phi}, \,\nabla u\big) = \big(\nabla(\phi^{k} - I_{h}\phi^{k}), \nabla u\big) +\big(\vert\Psi^{k}_{h}\vert^{2}-\vert\Psi^{k}\vert^{2},\, u \big), \quad \forall u \in X_{h}^{r}.
\end{align}

In the following of the section, we analyze the above three error equations term by term, respectively.

\subsection{\textbf {Estimates for (\ref{eq:4-4})}}
First, by taking $\varphi = \overline{\theta}_{\Psi}^{k}$ in (\ref{eq:4-4}), the imaginary part of the equation implies 
 \begin{equation}\label{eq:4-5-0}
  \begin{array}{@{}l@{}}
 {\displaystyle \frac{1}{\tau} \left(\| {\theta}_{\Psi}^{k} \|^{2}_{\mathcal{L}^{2}} - \| {\theta}_{\Psi}^{k-1} \|^{2}_{\mathcal{L}^{2}}\right) =  -\mathrm{Im} \sum_{i=1}^{5}V_i^{k}( \overline{\theta}_{\Psi}^{k}) \leq \;\sum_{i=1}^{5}\vert V_i^{k}( \overline{\theta}_{\Psi}^{k}) \vert. }\\[2mm]
 \end{array}
 \end{equation}
 Now we are going to estimate the terms $V_i^{k}( \overline{\theta}_{\Psi}^{k}), \,i=1,\cdots, 5 $ one by one. By the error estimates for the interpolation operator $\mathcal{I}_{h}$ and
the regularity of $\Psi$ in (\ref{eq:2-9}), we see that
 \begin{equation}\label{eq:4-6-0}
 \begin{array}{@{}l@{}}
{\displaystyle  |V_1^{k}( \overline{\theta}_{\Psi}^{k})|+ |V_2^{k}( \overline{\theta}_{\Psi}^{k})|\leq C\big(\tau^{4}+h^{2r}\big)+C\|\overline{\theta}_{\Psi}^{k}\|^{2}_{\mathcal{L}^2}.}
\end{array}
\end{equation}

Thanks to 
\begin{equation}\label{eq:4-7}
\begin{array}{@{}l@{}}
 {\displaystyle B(\mathbf{A};\psi,\varphi)=(\nabla\psi,\nabla\varphi)+\left(|\mathbf{A}|^{2}\psi,\varphi\right)
 +i\left(\varphi^{\ast}\nabla\psi-\psi\nabla\varphi^{\ast},\mathbf{A}\right)}\\[2mm]
 {\displaystyle \qquad \leq C\|\nabla\psi\|_{\mathbf{L}^2}
 \|\nabla\varphi\|_{\mathbf{L}^2},\quad \forall \mathbf{A}\in \mathbf{L}^6(\Omega),
 \,\,\,\psi,\varphi\in\mathcal{H}_{0}^{1}(\Omega),}
 \end{array}
 \end{equation}
 and
 \begin{equation}\label{eq:4-8}
 \begin{array}{@{}l@{}}
 {\displaystyle V_3^{k}( \overline{\theta}_{\Psi}^{k})= B(\mathbf{A}^{k-\frac{1}{2}};(\overline{\Psi}^{k} - \mathcal{I}_{h}\overline{\Psi}^{k}),\overline{\theta}_{\Psi}^{k})
 +B(\mathbf{A}^{k-\frac{1}{2}};( \Psi^{k-\frac{1}{2}} - \overline{\Psi}^{k}),\overline{\theta}_{\Psi}^{k}),}
 \end{array}
 \end{equation}
we get
 \begin{equation}\label{eq:4-9}
 {\displaystyle |V_3^{k}(\overline{\theta}_{\Psi}^{k})|\leq C\big(h^{2r}+\tau^{4}\big)+C\|\nabla\overline{\theta}_{\Psi}^{k}\|^{2}_{\mathbf{L}^2} .}
 \end{equation}
 In order to estimate $V_4^{k}( \overline{\theta}_{\Psi}^{k})$, we first decompose it as follows.
\begin{equation}\label{eq:4-9-0}
\begin{array}{@{}l@{}}
{\displaystyle V_4^{k}( \overline{\theta}_{\Psi}^{k}) = \left((\Psi^{k-\frac12}-\mathcal{I}_{h}\Psi^{k-\frac12})\phi^{k-\frac12},\,\overline{\theta}_{\Psi}^{k}\right) + \left(\mathcal{I}_{h}(\Psi^{k-\frac12}-\overline{\Psi}^{k})\phi^{k-\frac12}, \,\overline{\theta}_{\Psi}^{k}\right) }\\[2mm]
{\displaystyle \qquad \quad + \left( (\mathcal{I}_{h}\overline{\Psi}^{k}-\overline{\Psi}^{k}_{h})\phi^{k-\frac12}, \,\overline{\theta}_{\Psi}^{k}\right) + \left( \overline{\Psi}^{k}_{h}(\phi^{k-\frac12}-I_{h}\phi^{k-\frac12}), \,\overline{\theta}_{\Psi}^{k}\right) }\\[2mm]
{\displaystyle \qquad \quad + \left( \overline{\Psi}^{k}_{h}I_{h}(\phi^{k-\frac12}-\overline{\phi}^{k}), \,\overline{\theta}_{\Psi}^{k}\right) +  \left( \overline{\Psi}^{k}_{h} (I_{h}\overline{\phi}^{k}-\overline{\phi}^{k}_{h}), \,\overline{\theta}_{\Psi}^{k}\right) }
\end{array}
\end{equation}
By using Theorem~\ref{thm3-1}, the regularity assumption, and the properties of the interpolation operators, we obtain from (\ref{eq:4-9-0}) that
\begin{equation}\label{eq:4-9-1}
\begin{array}{@{}l@{}}
{\displaystyle \vert V_4^{k}( \overline{\theta}_{\Psi}^{k}) \vert \leq C\big(h^{2r}+\tau^{4}\big)+C\left(\|\nabla\overline{\theta}_{\Psi}^{k}\|^{2}_{\mathbf{L}^2}
 +\|\nabla\overline{\theta}_{\phi}^{k}\|^{2}_{\mathbf{L}^2}\right).}
\end{array}
\end{equation}
Notice that
 \begin{equation}\label{eq:4-10}
 \begin{array}{@{}l@{}}
 {\displaystyle V_5^{k}( \overline{\theta}_{\Psi}^{k}) = \Big[B(\overline{\mathbf{A}}^{k}_{h}; \mathcal{I}_{h}\overline{\Psi}^{k},\overline{\theta}_{\Psi}^{k})
 -B({\bm\pi}_{h}\overline{\mathbf{A}}^{k}; \mathcal{I}_{h}\overline{\Psi}^{k},\overline{\theta}_{\Psi}^{k})\Big]+\Big[B({\bm \pi}_{h}\overline{\mathbf{A}}^{k}; \mathcal{I}_{h}\overline{\Psi}^{k},\overline{\theta}_{\Psi}^{k})}\\[2mm]
 {\displaystyle \quad
 -B(\overline{\mathbf{A}}^{k}; \mathcal{I}_{h}\overline{\Psi}^{k},\overline{\theta}_{\Psi}^{k})\Big]+\Big[B(\overline{\mathbf{A}}^{k}; \mathcal{I}_{h}\overline{\Psi}^{k},\overline{\theta}_{\Psi}^{k})
 -B(\mathbf{A}^{k-\frac{1}{2}}; \mathcal{I}_{h}\overline{\Psi}^{k},\overline{\theta}_{\Psi}^{k})\Big].}\\[2mm]
 \end{array}
 \end{equation}
By applying (\ref{eq:3.6}) and Theorem~\ref{thm3-1}, it is easy to see that
\begin{equation}\label{eq:4-13}
 \begin{array}{@{}l@{}}
{\displaystyle |V_5^{k}( \overline{\theta}_{\Psi}^{k})|\leq C\big(h^{2r}+\tau^{4}\big) + C\left( D(\overline{\theta}^{k}_{\mathbf{A}},\overline{\theta}^{k}_{\mathbf{A}}) + \|\nabla\overline{\theta}_{\Psi}^{k}\|_{\mathbf{L}^2}^{2} \right). }
\end{array}
\end{equation}
Now multiplying (\ref{eq:4-5-0}) by $\tau$, summing over $k = 1,2,\cdots,m $, and applying the above estimates , we have
\begin{equation}\label{eq:4-15}
\begin{array}{@{}l@{}}
 {\displaystyle \|\theta^{m}_{\Psi}\|^{2}_{\mathcal{L}^2}\leq C\big(h^{2r}+\tau^{4}\big)+C \tau
 \sum_{k=1}^{m}\left(D(\overline{\theta}^{k}_{\mathbf{A}},\overline{\theta}^{k}_{\mathbf{A}}) + \|\nabla\overline{\theta}_{\Psi}^{k}\|_{\mathbf{L}^2}^{2} + \|\nabla\overline{\theta}_{\phi}^{k}\|^{2}_{\mathbf{L}^2}\right)   }\\[2mm]
 {\displaystyle \quad \leq C\big(h^{2r}+\tau^{4}\big)
 +C\tau \sum_{k=0}^{m}\left( D({\theta}^{k}_{\mathbf{A}},
 {\theta}^{k}_{\mathbf{A}}) + \|\nabla\theta_{\Psi}^{k}\|^{2}_{\mathbf{L}^2} + \|\nabla {\theta}_{\phi}^{k}\|^{2}_{\mathbf{L}^2}\right).}
 \end{array}
 \end{equation}

Next, we take $\varphi=\partial_{\tau}{\theta_{\Psi}^{k}}$ in (\ref{eq:4-4}), which gives
\begin{equation}\label{eq:4-16}
 {\displaystyle -2\mathrm{i}(\partial_{\tau} \theta^{k}_{\Psi},\partial_{\tau}{\theta_{\Psi}^{k}})+ B\Big(\overline{\mathbf{A}}^{k}_{h};\,\overline{\theta}_{\Psi}^{k},\partial_{\tau}{\theta_{\Psi}^{k}}\Big)
 =\sum_{j=1}^{5}V_{j}^{k}(\partial_{\tau}{\theta_{\Psi}^{k}}) .}
 \end{equation}
From the real part of (\ref{eq:4-16}) and (\ref{eq:3.6}), we obtain 
\begin{equation}\label{eq:4-16-0}
\begin{array}{@{}l@{}}
{\displaystyle \frac{1}{2\tau}\left(B(\overline{\mathbf{A}}_{h}^{k};\theta_{\Psi}^{k},\theta_{\Psi}^{k})- B(\overline{\mathbf{A}}_{h}^{k-1};\theta_{\Psi}^{k-1},\theta_{\Psi}^{k-1}) \right) = \sum_{j=1}^{5}\mathrm{Re}\big[V_{j}^{k}(\partial_{\tau}{\theta_{\Psi}^{k}})\big]  }\\[2mm]
{\displaystyle \qquad + \Big(\frac{1}{2}(\overline{\mathbf{A}}_{h}^{k}+\overline{\mathbf{A}}_{h}^{k-1})|\theta_{\Psi}^{k-1}|^{2},\,\frac{1}{2}(\partial_{\tau} \mathbf{A}_{h}^{k}+\partial_{\tau} \mathbf{A}_{h}^{k-1})\Big)        }\\[2mm]
{\displaystyle\qquad +\Big(f(\theta_{\Psi}^{k-1},\theta_{\Psi}^{k-1}),\,\frac{1}{2}(\partial_{\tau} \mathbf{A}_{h}^{k}+\partial_{\tau} \mathbf{A}_{h}^{k-1})\Big) ,}
\end{array}
\end{equation}
which yields
\begin{equation}\label{eq:4-16-1}
\begin{array}{@{}l@{}}
{\displaystyle \frac{1}{2}B(\overline{\mathbf{A}}_{h}^{m};\theta_{\Psi}^{m},\theta_{\Psi}^{m}) = \frac{1}{2}B(\overline{\mathbf{A}}_{h}^{0};\theta_{\Psi}^{0},\theta_{\Psi}^{0}) + \tau \sum_{j=1}^{5} \sum_{k=1}^{m}\mathrm{Re}\big[V_{j}^{k}(\partial_{\tau}{\theta_{\Psi}^{k}})\big] }\\[2mm]
{\displaystyle \quad + \tau\sum_{k=1}^{m}\Big(\frac{1}{2}(\overline{\mathbf{A}}_{h}^{k}+\overline{\mathbf{A}}_{h}^{k-1})|\theta_{\Psi}^{k-1}|^{2},\,\partial_{\tau} \overline{\mathbf{A}}_{h}^{k}\Big) +\tau\sum_{k=1}^{m} \Big(f(\theta_{\Psi}^{k-1},\theta_{\Psi}^{k-1}),\, \partial_{\tau} \overline{\mathbf{A}}_{h}^{k}\Big).}
\end{array}
\end{equation}
From Theorem~\ref{thm3-1}, we deduce
\begin{equation}\label{eq:4-16-2}
\begin{array}{@{}l@{}}
{\displaystyle   \sum_{k=1}^{m}\Big(\frac{1}{2}(\overline{\mathbf{A}}_{h}^{k}+\overline{\mathbf{A}}_{h}^{k-1})|\theta_{\Psi}^{k-1}|^{2},\,\partial_{\tau} \overline{\mathbf{A}}_{h}^{k}\Big) \leq \sum_{k=1}^{m} \| \overline{\mathbf{A}}_{h}^{k}+\overline{\mathbf{A}}_{h}^{k-1}\|_{\mathbf{L}^{6}} \|\theta_{\Psi}^{k-1}\|^{2}_{\mathcal{L}^{6}} \|\partial_{\tau} \overline{\mathbf{A}}_{h}^{k}\|_{\mathbf{L}^{2}}  }\\[2mm]
{\displaystyle \qquad  \leq C\sum_{k=1}^{m}\|\theta_{\Psi}^{k-1}\|^{2}_{\mathcal{L}^{6}}   \leq C\sum_{k=0}^{m}\|\nabla\theta_{\Psi}^{k}\|^{2}_{\mathbf{L}^{2}}, }\\[2mm]
{\displaystyle  \sum_{k=1}^{m}\Big(f(\theta_{\Psi}^{k-1},\theta_{\Psi}^{k-1}),\,\partial_{\tau} \overline{{\bm \pi}_{h}\mathbf{A}}^{k}\Big) \leq \sum_{k=1}^{m} \Vert  \nabla \theta_{\Psi}^{k-1} \Vert_{\mathbf{L}^{2}}\Vert \theta_{\Psi}^{k-1} \Vert_{\mathcal{L}^{6}} \Vert\partial_{\tau} \overline{{\bm \pi}_{h}\mathbf{A}}^{k}\Vert_{\mathbf{L}^{3}}  }\\[2mm]
{\displaystyle \qquad \qquad \leq C\sum_{k=0}^{m}\|\nabla\theta_{\Psi}^{k}\|^{2}_{\mathbf{L}^{2}}. }
\end{array}
\end{equation}
Denoting by 
\begin{equation}\label{eq:4-16-3}
J_1^{k} = \Big(f(\theta_{\Psi}^{k-1},\theta_{\Psi}^{k-1}),\;\overline{\partial_{\tau} \theta}_{\mathbf{A}}^{k}\Big) ,
\end{equation}
we have
\begin{equation}\label{eq:4-16-4}
\begin{array}{@{}l@{}}
{\displaystyle \frac{1}{2}B(\overline{\mathbf{A}}_{h}^{m};\theta_{\Psi}^{m},\theta_{\Psi}^{m}) \leq 
 C h^{2r} + C \tau\sum_{k=0}^{m}\|\nabla \theta_{\Psi}^{k}\|^{2}_{\mathbf{L}^{2}} }\\[2mm]
{\displaystyle \qquad \qquad  +  \tau\sum_{k=1}^{m}J_1^{k} +\tau \sum_{j=1}^{5} \sum_{k=1}^{m}\mathrm{Re}\big[V_{j}^{k}(\partial_{\tau}{\theta_{\Psi}^{k}})\big]. }\\[2mm]
 \end{array}
\end{equation}

Now let us estimate $\sum_{k=1}^{m}\mathrm{Re}\big[V_{j}^{k}(\partial_{\tau}{\theta_{\Psi}^{k}})\big] $, $j=1,\cdots,5$ term by term. In light of (\ref{eq:4-2}), we get
\begin{equation}\label{eq:4-17}
\begin{array}{@{}l@{}}
{\displaystyle \tau\sum_{k=1}^{m} V_{1}^{k}(\partial_{\tau}{\theta_{\Psi}^{k}})=2\mathrm{i}\sum_{k=1}^{m}\Big(\partial_{\tau} \mathcal{I}_{h}\Psi^{k}-(\Psi_t)^{k-\frac{1}{2}},\,\theta_{\Psi}^{k}-{\theta_{\Psi}^{k-1}}\Big)}\\[2mm]
{\displaystyle \quad= 2\mathrm{i}\Big(\partial_{\tau} \mathcal{I}_{h}\Psi^{m}-(\Psi_t)^{m-\frac{1}{2}},\;\theta_{\Psi}^{m}\Big)-2\mathrm{i}\Big(\partial_{\tau} \mathcal{I}_{h}\Psi^{1}
-(\Psi_t)^{\frac{1}{2}},\;\theta_{\Psi}^{0}\Big)}\\[2mm]
{\displaystyle \quad-2\mathrm{i}\sum_{k=1}^{m-1}\Big(\partial_{\tau} \mathcal{I}_{h}\Psi^{k+1}
-\partial_{\tau} \mathcal{I}_{h}\Psi^{k}-(\Psi_t)^{k+\frac{1}{2}}
+(\Psi_t)^{k-\frac{1}{2}},\;\theta_{\Psi}^{k}\Big).}
\end{array}
\end{equation}
By employing the regularity assumption and the error estimates of interpolation operators, we see that
\begin{equation}\label{eq:4-19}
\begin{array}{@{}l@{}}
{\displaystyle |\tau\sum_{k=1}^{m} V_{1}^{k}(\partial_{\tau}{\theta_{\Psi}^{k}})|\leq C\big(h^{2r}+\tau^{4}\big)
+ C\|\theta_{\Psi}^{m}\|_{\mathcal{L}^2}^{2}+C\tau \sum_{k=1}^{m-1}{\|\theta_{\Psi}^{k}\|_{\mathcal{L}^2}^{2}}.}
\end{array}
\end{equation}
The second term can be rewritten as
\begin{equation*}
\begin{array}{@{}l@{}}
{\displaystyle \tau V_{2}^{k}(\partial_{\tau}{\theta_{\Psi}^{k}}) = 2\Big(V(\Psi^{k-\frac{1}{2}}-\overline{\Psi}_{h}^{k}),\,\theta_{\Psi}^{k}-\theta_{\Psi}^{k-1}\Big) }\\[2mm]
{\displaystyle   = 2\Big(V(\Psi^{k-\frac{1}{2}}
-\mathcal{I}_{h}\overline{\Psi}^{k}),\,\theta_{\Psi}^{k}-\theta_{\Psi}^{k-1}\Big)-2\Big(V(\frac{1}{2}(\theta_{\Psi}^{k}+\theta_{\Psi}^{k-1}),\,\theta_{\Psi}^{k}-\theta_{\Psi}^{k-1}\Big). }\\[2mm]
\end{array}
\end{equation*}
Arguing as before, we obtain
\begin{equation}\label{eq:4-22}
\begin{array}{@{}l@{}}
{\displaystyle  \big|\tau\sum_{k=1}^{m} \mathrm{Re}\big[ V_{2}^{k}(\partial_{\tau}{\theta_{\Psi}^{k}})\big]\big|\leq C\big(h^{2r}+\tau^{4}\big)
+C \|\theta_{\Psi}^{m}\|_{\mathcal{L}^2}^{2}+ C\tau \sum_{k=1}^{m-1}{\|\theta_{\Psi}^{k}\|_{\mathcal{L}^2}^{2}}.}
\end{array}
\end{equation}

By the definition of the bilinear functional $B(\mathbf{A};\psi,\varphi)$ in (\ref{eq:2-7}), we can  rewrite $\tau  V_{3}^{k}(\partial_{\tau}{\theta_{\Psi}^{k}})$ as follows:
\begin{equation}\label{eq:4-23}
\begin{array}{@{}l@{}}
{\displaystyle \tau  V_{3}^{k}(\partial_{\tau}{\theta_{\Psi}^{k}})= \left(\nabla( \Psi^{k-\frac{1}{2}}-\mathcal{I}_{h}\overline{\Psi}^{k}),\;\nabla
(\theta_{\Psi}^{k}-\theta_{\Psi}^{k-1})\right)}\\[2mm]
{\displaystyle\quad \qquad +\left(|\mathbf{A}^{k-\frac{1}{2}}|^2(\Psi^{k-\frac{1}{2}}-\mathcal{I}_{h}\overline{\Psi}^{k}),\;\theta_{\Psi}^{k}-\theta_{\Psi}^{k-1}\right)}\\[2mm]
{\displaystyle \quad \qquad+ \mathrm{i}\left(\nabla(\Psi^{k-\frac{1}{2}}-\mathcal{I}_{h}\overline{\Psi}^{k})\mathbf{A}^{k-\frac{1}{2}},\;\theta_{\Psi}^{k}-\theta_{\Psi}^{k-1}\right)}\\[2mm]
{\displaystyle \quad \qquad-\mathrm{i}\left((\Psi^{k-\frac{1}{2}}-\mathcal{I}_{h}\overline{\Psi}^{k})\mathbf{A}^{k-\frac{1}{2}},\;\nabla \theta_{\Psi}^{k}
-\nabla \theta_{\Psi}^{k-1}\right).}\\[2mm]
\end{array}
\end{equation}

By employing (\ref{eq:4-2}), (\ref{subeq:2.00}), the regularity assumption (\ref{eq:2-9}), and the Young's inequality,  we can prove the following estimate
\begin{equation}\label{eq:4-23-0}
{\displaystyle  |\tau\sum_{k=1}^{m}  V_{3}^{k}(\partial_{\tau}{\theta_{\Psi}^{k}})| \leq C\big(h^{2r}+\tau^{4}\big) +C \Vert\theta_{\Psi}^{m}\Vert_{\mathcal{L}^2}^{2}
+ \frac{1}{32} \Vert\nabla\theta_{\Psi}^{m}\Vert_{\mathbf{L}^2}^{2}+C\tau \sum_{k=0}^{m}{\Vert\nabla\theta_{\Psi}^{k}\Vert_{\mathbf{L}^2}^{2}}}
\end{equation}
by some standard but tedious arguments which are analogous to the estimate of $\sum_{k=1}^{m}  V_{1}^{k}(\partial_{\tau}{\theta_{\Psi}^{k}})$. Due to space limitations, we omit the proof here.

To estimate the term $\tau V_{4}^{k}(\partial_{\tau}{\theta_{\Psi}^{k}})$, we rewrite it by
\begin{equation}\label{eq:4-24}
 \begin{array}{@{}l@{}}
 {\displaystyle \tau V_{4}^{k}(\partial_{\tau}{\theta_{\Psi}^{k}}) = 2\left(\phi^{k-\frac12}\Psi^{k-\frac12}-I_{h}\overline{\phi}^{k}\mathcal{I}_{h}\overline{\Psi}^{k}, \;\theta_{\Psi}^{k}-\theta_{\Psi}^{k-1} \right) -  2\left(I_{h}\overline{\phi}^{k}\overline{\theta}_{\Psi}^{k}, \;\theta_{\Psi}^{k}-\theta_{\Psi}^{k-1} \right)    }\\[2mm]
 {\displaystyle \qquad \qquad - 2\left( \mathcal{I}_{h}\overline{\Psi}^{k}\overline{\theta}_{\phi}^{k}, \;\theta_{\Psi}^{k}-\theta_{\Psi}^{k-1} \right) -2 \left(\overline{\theta}_{\phi}^{k}\overline{\theta}_{\Psi}^{k}, \;\theta_{\Psi}^{k}-\theta_{\Psi}^{k-1} \right) }.
\end{array}
\end{equation}

Arguing as before, we can obtain
\begin{equation}\label{eq:4-25}
 \begin{array}{@{}l@{}}
 {\displaystyle |\sum_{k=1}^{m}\big(\phi^{k-\frac12}\Psi^{k-\frac12}-I_{h}\overline{\phi}^{k}\mathcal{I}_{h}\overline{\Psi}^{k}, \;\theta_{\Psi}^{k}-\theta_{\Psi}^{k-1} \big)| \leq C (h^{2r}+\tau^{4})  }\\[2mm]
 {\displaystyle \qquad\quad +C \Vert\theta_{\Psi}^{m}\Vert_{\mathcal{L}^2}^{2} + C\tau \sum_{k=0}^{m}{\Vert\theta_{\Psi}^{k}\Vert_{\mathcal{L}^2}^{2}}, }\\[2mm]
 {\displaystyle  \big|{\rm Re}\sum_{k=1}^{m} \big(I_{h}\overline{\phi}^{k}\overline{\theta}_{\Psi}^{k}, \;\theta_{\Psi}^{k}-\theta_{\Psi}^{k-1} \big) \big|  \leq C(h^{2r} + \tau^{4}) + C \Vert\theta_{\Psi}^{m}\Vert_{\mathcal{L}^2}^{2} +\frac{1}{32}\Vert\nabla{\theta}_{\Psi}^{m}\Vert_{\mathbf{L}^2}^{2}  }\\[2mm]
 {\displaystyle \qquad\quad+ C\tau \sum_{k=0}^{m} \Vert\nabla\theta_{\Psi}^{k}\Vert_{\mathbf{L}^2}^{2}. }
\end{array}
\end{equation}

The real part of the last two terms on the right hand side of (\ref{eq:4-24}) can be decomposed as follows:
\begin{equation}\label{eq:4-26}
 \begin{array}{@{}l@{}}
{\displaystyle {\rm Re}\big[  \big( \mathcal{I}_{h}\overline{\Psi}^{k}\overline{\theta}_{\phi}^{k}, \;\theta_{\Psi}^{k}-\theta_{\Psi}^{k-1} \big) + \big(\overline{\theta}_{\phi}^{k}\overline{\theta}_{\Psi}^{k}, \;\theta_{\Psi}^{k}-\theta_{\Psi}^{k-1} \big) \big] }\\[2mm]
{\displaystyle = {\rm Re}\big[  \big( \mathcal{I}_{h}\overline{\Psi}^{k} (\theta_{\Psi}^{k}-\theta_{\Psi}^{k-1})^{\ast}, \;\overline{\theta}_{\phi}^{k} \big) \big] +  \frac{1}{2}\big(|\theta_{\Psi}^{k}|^{2}-|\theta_{\Psi}^{k-1}|^{2}, \;\overline{\theta}_{\phi}^{k} \big)} \\[2mm]
{\displaystyle =  {\rm Re}\big[  \big( \mathcal{I}_{h}{\Psi}^{k} (\theta_{\Psi}^{k})^{\ast}-\mathcal{I}_{h}{\Psi}^{k-1}(\theta_{\Psi}^{k-1})^{\ast}, \;\overline{\theta}_{\phi}^{k} \big) \big] + \frac{1}{2} \big(|\theta_{\Psi}^{k}|^{2}-|\theta_{\Psi}^{k-1}|^{2}, \;\overline{\theta}_{\phi}^{k} \big)   }\\[2mm]
{\displaystyle \qquad  - {\rm Re}\big[  \big( \,\overline{\theta}_{\phi}^{k} (\mathcal{I}_{h}{\Psi}^{k} -\mathcal{I}_{h}{\Psi}^{k-1} ), \; \overline{\theta}_{\Psi}^{k}\big) \big] }\\[2mm]
{\displaystyle = \frac{1}{2}\big( |\Psi_{h}^{k}|^{2} - |\mathcal{I}_{h}\Psi^{k}|^{2},\; \overline{\theta}_{\phi}^{k} \big) -  \frac{1}{2}\big( |\Psi_{h}^{k-1}|^{2} - |\mathcal{I}_{h}\Psi^{k-1}|^{2},\; \overline{\theta}_{\phi}^{k} \big)  } \\[2mm]
{\displaystyle  \qquad  - {\rm Re}\big[  \big( \,\overline{\theta}_{\phi}^{k} (\mathcal{I}_{h}{\Psi}^{k} -\mathcal{I}_{h}{\Psi}^{k-1} ), \; \overline{\theta}_{\Psi}^{k}\big) \big]. }\\[2mm]
 \end{array}
\end{equation}

Combining (\ref{eq:4-24})-(\ref{eq:4-26}) and setting
\begin{equation}\label{eq:4-26-0}
J_2^{k} = \big( |\Psi_{h}^{k}|^{2} - |\mathcal{I}_{h}\Psi^{k}|^{2},\; \overline{\theta}_{\phi}^{k} \big) -  \big( |\Psi_{h}^{k-1}|^{2} - |\mathcal{I}_{h}\Psi^{k-1}|^{2},\; \overline{\theta}_{\phi}^{k} \big), 
\end{equation}
 we obtain
\begin{equation}\label{eq:4-27}
  \begin{array}{@{}l@{}}
 {\displaystyle  \tau \sum_{k=1}^{m} \mathrm{Re}\big[V_{4}^{k}(\partial_{\tau}{\theta_{\Psi}^{k}}) \big] \leq  C\big(h^{2r}+\tau^{4}\big) + C \Vert\theta_{\Psi}^{m}\Vert_{\mathcal{L}^2}^{2} + \frac{1}{32} \|\nabla\theta_{\Psi}^{m}\|^{2}_{\mathbf{L}^{2}}  }\\[2mm]
  {\displaystyle \qquad\quad+ C\tau \sum_{k=0}^{m}\big(\Vert\nabla \theta_{\Psi}^{k}\Vert_{\mathbf{L}^2}^{2} + \Vert\nabla \theta_{\phi}^{k}\Vert_{\mathbf{L}^2}^{2}\big) - \sum_{k=1}^{m}J_2^{k}. }
\end{array}
\end{equation}
 
 We now focus on the analysis of $\tau V_{5}^{k}(\partial_{\tau}{\theta_{\Psi}^{k}})$, which can be rewritten as 
\begin{equation}\label{eq:4-30}
\begin{array}{@{}l@{}}
{\displaystyle   \tau V_{5}^{k}(\partial_{\tau}{\theta_{\Psi}^{k}})= \Big[B(\mathbf{A}^{k-\frac{1}{2}}; \mathcal{I}_{h}\overline{\Psi}^{k},\theta_{\Psi}^{k}-\theta_{\Psi}^{k-1})
-B(\overline{\mathbf{A}}^{k}; \mathcal{I}_{h}\overline{\Psi}^{k},\theta_{\Psi}^{k}-\theta_{\Psi}^{k-1})\Big]}\\[2mm]
{\displaystyle  \qquad\quad \qquad\quad +\Big[B(\overline{\mathbf{A}}^{k}; \mathcal{I}_{h}\overline{\Psi}^{k},\theta_{\Psi}^{k}-\theta_{\Psi}^{k-1})-B({\bm \pi}_{h}\overline{\mathbf{A}}^{k}; \mathcal{I}_{h}\overline{\Psi}^{k},\theta_{\Psi}^{k}-\theta_{\Psi}^{k-1})\Big]}\\[2mm]
{\displaystyle  \qquad \quad \qquad\quad +\Big[B({\bm \pi}_{h}\overline{\mathbf{A}}^{k}; \mathcal{I}_{h}\overline{\Psi}^{k},\theta_{\Psi}^{k}-\theta_{\Psi}^{k-1})-B(\overline{\mathbf{A}}^{k}_{h}; \mathcal{I}_{h}\overline{\Psi}^{k},\theta_{\Psi}^{k}-\theta_{\Psi}^{k-1})\Big]}\\[2mm]
{\displaystyle\qquad \qquad\quad:=V_5^{k,1}+V_5^{k,2}+V_5^{k,3}.}
\end{array}
\end{equation}

By applying (\ref{eq:3.6}) and (\ref{eq:4-2}) and arguing as before, we deduce
\begin{equation}\label{eq:4-34}
\begin{array}{@{}l@{}}
{\displaystyle  |\sum_{k=1}^{m} V_5^{k,1}| +  |\sum_{k=1}^{m} V_5^{k,2}| \leq C\big(h^{2r}+\tau^{4}\big)
 +C \|\theta_{\Psi}^{m}\|_{\mathcal{L}^2}^{2}}\\[2mm]
{\displaystyle \quad+ \frac{1}{32} \|\nabla\theta_{\Psi}^{m}\|_{\mathbf{L}^2}^{2}+C\tau \sum_{k=1}^{m}\|\nabla\theta_{\Psi}^{k}\|_{\mathbf{L}^2}^{2}.}
\end{array}
\end{equation}

In order to estimate $|\sum\limits_{k=1}^{m} V_5^{k,3}|$, we rewrite it as follows.
\begin{equation}\label{eq:4-35}
\begin{array}{@{}l@{}}
{\displaystyle  \sum_{k=1}^{m}V_5^{k,3}=\sum_{k=1}^{m}{ \left(\mathcal{I}_{h}\overline{\Psi}^{k}({\bm \pi}_{h}\overline{\mathbf{A}}^{k}+\overline{\mathbf{A}}^{k}_{h})({\bm\pi}_{h}\overline{\mathbf{A}}^{k}
-\overline{\mathbf{A}}^{k}_{h}),\;\theta_{\Psi}^{k}-\theta_{\Psi}^{k-1}\right)}}\\[2mm]
{\displaystyle\quad\quad\quad\quad-\sum_{k=1}^{m}{\mathrm{i}\left( \mathcal{I}_{h}\overline{\Psi}^{k}({\bm\pi}_{h}\overline{\mathbf{A}}^{k}-\overline{\mathbf{A}}^{k}_{h}),\;
\nabla\theta_{\Psi}^{k}-\nabla\theta_{\Psi}^{k-1}\right)}}\\[2mm]
{\displaystyle \quad\quad\quad\quad+\sum_{k=1}^{m}{\mathrm{i}\left(\nabla \mathcal{I}_{h}\overline{\Psi}^{k}({\bm\pi}_{h}\overline{\mathbf{A}}^{k}
-\overline{\mathbf{A}}^{k}_{h}),\;\theta_{\Psi}^{k}-\theta_{\Psi}^{k-1}\right)}}\\[2mm]
{\displaystyle\quad\quad\quad :=T_1+T_2+T_3.}
\end{array}
\end{equation}
We decompose the term $T_1$ as follows.
\begin{equation}\label{eq:4-36}
\begin{array}{@{}l@{}}
{\displaystyle T_1= \sum_{k=1}^{m}\left(\mathcal{I}_{h}\overline{\Psi}^{k}({\bm\pi}_{h}\overline{\mathbf{A}}^{k}+\overline{\mathbf{A}}^{k}_{h})
({\bm\pi}_{h}\overline{\mathbf{A}}^{k}-\overline{\mathbf{A}}^{k}_{h}),\;\theta_{\Psi}^{k}-\theta_{\Psi}^{k-1}\right)}\\[2mm]
{\displaystyle \quad = -\left(\mathcal{I}_{h}\overline{\Psi}^{m}({\bm\pi}_{h}\overline{\mathbf{A}}^{m}
+\overline{\mathbf{A}}^{m}_{h})
\overline{\theta}_{\mathbf{A}}^{m},\;\theta_{\Psi}^{m}\right) + \left(\mathcal{I}_{h}\overline{\Psi}^{0}({\bm\pi}_{h}\overline{\mathbf{A}}^{0}+\overline{\mathbf{A}}^{0}_{h})
\overline{\theta}_{\mathbf{A}}^{0},\;\theta_{\Psi}^{0}\right)}\\[2mm]
{\displaystyle \quad + \sum_{k=1}^{m}\left(\mathcal{I}_{h}\overline{\Psi}^{k}({\bm\pi}_{h}\overline{\mathbf{A}}^{k}
+\overline{\mathbf{A}}^{k}_{h})
\overline{\theta}_{\mathbf{A}}^{k}-\mathcal{I}_{h}\overline{\Psi}^{k-1}({\bm\pi}_{h}\overline{\mathbf{A}}^{k-1}
+\overline{\mathbf{A}}^{k-1}_{h})
\overline{\theta}_{\mathbf{A}}^{k-1},\;\theta_{\Psi}^{k-1}\right).}
\end{array}
\end{equation}

By applying the Young's inequality and Theorem~\ref{thm3-1}, we can estimate the first two terms on the right side of (\ref{eq:4-36}) by
\begin{equation}\label{eq:4-37}
\begin{array}{@{}l@{}}
{\displaystyle |\left(\mathcal{I}_{h}\overline{\Psi}^{m}({\bm\pi}_{h}\overline{\mathbf{A}}^{m}
+\overline{\mathbf{A}}^{m}_{h})\overline{\theta}_{\mathbf{A}}^{m},\;\theta_{\Psi}^{m}\right)|
+|\left(\mathcal{I}_{h}\overline{\Psi}^{0}({\bm\pi}_{h}\overline{\mathbf{A}}^{0}
+\overline{\mathbf{A}}^{0}_{h})\overline{\theta}_{\mathbf{A}}^{0},\;\theta_{\Psi}^{0}\right)|}\\[2mm]
{\displaystyle \quad\leq \frac{1}{16}D(\overline{\theta}_{\mathbf{A}}^{m},\overline{\theta}_{\mathbf{A}}^{m})+ C\|\theta_{\Psi}^{m}\|_{\mathcal{L}^2}^{2}+ Ch^{2r}.}
\end{array}
\end{equation}

Since
\begin{equation}\label{eq:4-38}
\begin{array}{@{}l@{}}
{\displaystyle \left(\mathcal{I}_{h}\overline{\Psi}^{k}({\bm\pi}_{h}\overline{\mathbf{A}}^{k}+\overline{\mathbf{A}}^{k}_{h})
\overline{\theta}_{\mathbf{A}}^{k}-\mathcal{I}_{h}\overline{\Psi}^{k-1}({\bm\pi}_{h}\overline{\mathbf{A}}^{k-1}
+\overline{\mathbf{A}}^{k-1}_{h})\overline{\theta}_{\mathbf{A}}^{k-1},\;\theta_{\Psi}^{k-1}\right)}\\[2mm]
{\displaystyle \quad= \tau\left(\mathcal{I}_{h}\overline{\Psi}^{k}({\bm\pi}_{h}\overline{\mathbf{A}}^{k}
+\overline{\mathbf{A}}^{k}_{h})\frac{\overline{\theta}_{\mathbf{A}}^{k}-\overline{\theta}_{\mathbf{A}}^{k-1}}{\tau},\;\theta_{\Psi}^{k-1}\right)}\\[2mm]
{\displaystyle \quad\quad+ \tau\left(\frac{\mathcal{I}_{h}\overline{\Psi}^{k}-\mathcal{I}_{h}\overline{\Psi}^{k-1}}{\tau}({\bm\pi}_{h}\overline{\mathbf{A}}^{k}+\overline{\mathbf{A}}^{k}_{h})\overline{\theta}_{\mathbf{A}}^{k-1},\;\theta_{\Psi}^{k-1}\right)}\\[2mm]
{\displaystyle \quad \quad+ \tau \left(\mathcal{I}_{h}\overline{\Psi}^{k-1}\overline{\theta}_{\mathbf{A}}^{k-1}\Big(\frac{{\bm\pi}_{h}\overline{\mathbf{A}}^{k}
-{\bm\pi}_{h}\overline{\mathbf{A}}^{k-1}}{\tau}+\frac{\overline{\mathbf{A}}_{h}^{k}-\overline{\mathbf{A}}_{h}^{k-1}}{\tau}\Big),\;\theta_{\Psi}^{k-1}\right),}
\end{array}
\end{equation}
we deduce
\begin{equation}\label{eq:4-39}
\begin{array}{@{}l@{}}
{\displaystyle |\left(\mathcal{I}_{h}\overline{\Psi}^{k}({\bm\pi}_{h}\overline{\mathbf{A}}^{k}
+\overline{\mathbf{A}}^{k}_{h})\overline{\theta}_{\mathbf{A}}^{k}-\mathcal{I}_{h}\overline{\Psi}^{k-1}
({\bm\pi}_{h}\overline{\mathbf{A}}^{k-1}+\overline{\mathbf{A}}^{k-1}_{h})
\overline{\theta}_{\mathbf{A}}^{k-1},\;\theta_{\Psi}^{k-1}\right)|}\\[2mm]
{\displaystyle \quad\leq \tau \|\mathcal{I}_{h}\overline{\Psi}^{k}\|_{\mathcal{L}^6}\| {\bm\pi}_{h}\overline{\mathbf{A}}^{k}+\overline{\mathbf{A}}^{k}_{h}\|_{\mathbf{L}^6}
\|\partial_{\tau}\overline{\theta}_{\mathbf{A}}^{k}\|_{\mathbf{L}^2}\|\theta_{\Psi}^{k-1}\|_{\mathcal{L}^6}}\\[2mm]
{\displaystyle \quad \quad+ \tau \|\partial_{\tau} \mathcal{I}_{h}\overline{\Psi}^{k}\|_{\mathcal{L}^2}\| {\bm\pi}_{h}\overline{\mathbf{A}}^{k}+\overline{\mathbf{A}}^{k}_{h}\|_{\mathbf{L}^6}
\|\overline{\theta}_{\mathbf{A}}^{k-1}\|_{\mathbf{L}^6}\|\theta_{\Psi}^{k-1}\|_{\mathcal{L}^6}}\\[2mm]
{\displaystyle \quad\quad +\tau \|\mathcal{I}_{h}\overline{\Psi}^{k-1}\|_{\mathcal{L}^6}\| \overline{\theta}_{\mathbf{A}}^{k-1}\|_{\mathbf{L}^6}\| \partial_{\tau} {\bm\pi}_{h}\overline{\mathbf{A}}^{k} +\partial_{\tau}\overline{\mathbf{A}}_{h}^{k}\|_{\mathbf{L}^2}\|\theta_{\Psi}^{k-1}\|_{\mathcal{L}^6}}\\[2mm]
{\displaystyle \quad\leq C\tau\left(\|\partial_{\tau} \overline{\theta}_{\mathbf{A}}^{k}\|_{\mathbf{L}^2}
 + \|\overline{\theta}_{\mathbf{A}}^{k-1}\|_{\mathbf{H}^1}\right)\|\theta_{\Psi}^{k-1}\|_{\mathcal{H}^1}}\\[2mm]
{\displaystyle \quad\leq C\tau\left(\|\partial_{\tau} \overline{\theta}_{\mathbf{A}}^{k}\|^{2}_{\mathbf{L}^2} + D(\overline{\theta}_{\mathbf{A}}^{k-1},\overline{\theta}_{\mathbf{A}}^{k-1}) + \|\nabla\theta_{\Psi}^{k-1}\|_{\mathbf{L}^2}^{2} \right) }
\end{array}
\end{equation}
by applying Theorem~\ref{thm3-1}.

Hence we get the estimate of $T_1$ as follows.
\begin{equation}\label{eq:4-40}
\begin{array}{@{}l@{}}
{\displaystyle |T_1| \leq \frac{1}{16}D(\overline{\theta}_{\mathbf{A}}^{m},\overline{\theta}_{\mathbf{A}}^{m})
+ C\|\theta_{\Psi}^{m}\|_{\mathcal{L}^2}^{2}+ C h^{2r} }\\[2mm]
{\displaystyle \qquad  + C\tau \sum_{k=0}^{m}\left(\|\partial_{\tau} {\theta}_{\mathbf{A}}^{k}\|_{\mathbf{L}^2}^{2}
+ D({\theta}_{\mathbf{A}}^{k},{\theta}_{\mathbf{A}}^{k})+ \|\nabla\theta_{\Psi}^{k}\|_{\mathbf{L}^2}^{2}\right).}\\[2mm]
\end{array}
\end{equation}
By virtue of (\ref{eq:4-2}) and integrating by parts, we discover
\begin{equation}\label{eq:4-41}
\begin{array}{@{}l@{}}
{\displaystyle {\rm i}T_2=\sum_{k=1}^{m}\left( \mathcal{I}_{h}\overline{\Psi}^{k}({\bm\pi}_{h}\overline{\mathbf{A}}^{k}
-\overline{\mathbf{A}}^{k}_{h}),\;\nabla\theta_{\Psi}^{k}-\nabla\theta_{\Psi}^{k-1}\right)}\\[2mm]
{\displaystyle \quad=\left(\nabla \mathcal{I}_{h}\overline{\Psi}^{m}\overline{\theta}_{\mathbf{A}}^{m},\;\theta_{\Psi}^{m}\right)
 + \left(\mathcal{I}_{h}\overline{\Psi}^{m}\nabla\cdot\overline{\theta}_{\mathbf{A}}^{m},\;\theta_{\Psi}^{m}\right)
 +\left(\mathcal{I}_{h}\overline{\Psi}^{0}\overline{\theta}_{\mathbf{A}}^{0},\;\nabla\theta_{\Psi}^{0}\right)}\\[2mm]
{\displaystyle \quad \quad\quad +\sum_{k=1}^{m}\left(\mathcal{I}_{h}\overline{\Psi}^{k}\overline{\theta}_{\mathbf{A}}^{k}- \mathcal{I}_{h}\overline{\Psi}^{k-1}\overline{\theta}_{\mathbf{A}}^{k-1},\;\nabla\theta_{\Psi}^{k-1}\right),}
\end{array}
\end{equation}
By using Young's inequality and (\ref{eq:4-1}), we can estimate the first three terms on the right side of (\ref{eq:4-41}) as follows:
 \begin{equation}\label{eq:4-42}
 \begin{array}{@{}l@{}}
 {\displaystyle |\big(\nabla \mathcal{I}_{h}\overline{\Psi}^{m}\overline{\theta}_{\mathbf{A}}^{m},\;\theta_{\Psi}^{m}\big)|
 +|\big(\mathcal{I}_{h}\overline{\Psi}^{m}\nabla\cdot\overline{\theta}_{\mathbf{A}}^{m},\;\theta_{\Psi}^{m}\big)|
 +|\big(\mathcal{I}_{h}\overline{\Psi}^{0}\overline{\theta}_{\mathbf{A}}^{0},\;\nabla\theta_{\Psi}^{0}\big)|}\\[2mm]
 {\displaystyle \quad\leq \|\nabla \mathcal{I}_{h}\overline{\Psi}^{m}\|_{\mathbf{L}^3}\|\overline{\theta}_{\mathbf{A}}^{m}\|_{\mathbf{L}^6}\|\theta_{\Psi}^{m}\|_{\mathcal{L}^2}
 +\|\mathcal{I}_{h}\overline{\Psi}^{m}\|_{\mathcal{L}^{\infty}}\|\nabla\cdot\overline{\theta}_{\mathbf{A}}^{m}\|_{{L}^2}
 \|\theta_{\Psi}^{m}\|_{\mathcal{L}^2}+C h^{2r} }\\[2mm]
{\displaystyle \quad\leq C\|\overline{\theta}_{\mathbf{A}}^{m}\|_{\mathbf{H}^1}\|\theta_{\Psi}^{m}\|_{\mathcal{L}^2}
 +C h^{2r}\leq \frac{1}{16} D(\overline{\theta}_{\mathbf{A}}^{m},\overline{\theta}_{\mathbf{A}}^{m}) + C\|\theta_{\Psi}^{m}\|^{2}_{\mathcal{L}^2}+C h^{2r}.  }\\[2mm]
 \end{array}
 \end{equation}
The last term at the right side of (\ref{eq:4-41}) satisfies the following estimate.
 \begin{equation}\label{eq:4-43}
 \begin{array}{@{}l@{}}
 {\displaystyle |\sum_{k=1}^{m}\big(\mathcal{I}_{h}\overline{\Psi}^{k}\overline{\theta}_{\mathbf{A}}^{k}- \mathcal{I}_{h}\overline{\Psi}^{k-1}\overline{\theta}_{\mathbf{A}}^{k-1},\;\nabla\theta_{\Psi}^{k-1}\big)|}\\[2mm]
{\displaystyle \leq  \tau \sum_{k=1}^{m}{ \Big( \|\partial_{\tau} \mathcal{I}_{h}\overline{\Psi}^{k}\|_{\mathcal{L}^3}\|\overline{\theta}_{\mathbf{A}}^{k}
\|_{\mathbf{L}^6}\|\nabla\theta_{\Psi}^{k-1}\|_{\mathbf{L}^2}} + \|\mathcal{I}_{h}\overline{\Psi}^{k-1}\|_{\mathcal{L}^{\infty}}
\|\partial_{\tau} \overline{\theta}_{\mathbf{A}}^{k}\|_{\mathbf{L}^2}\|\nabla\theta_{\Psi}^{k-1}\|_{\mathbf{L}^2}  \Big) }\\[2mm]
{\displaystyle \leq  C\tau \sum_{k=0}^{m}\left(D({\theta}_{\mathbf{A}}^{k},{\theta}_{\mathbf{A}}^{k})
+\|\partial_{\tau} {\theta}_{\mathbf{A}}^{k}\|_{\mathbf{L}^2}^{2}+\|\nabla\theta_{\Psi}^{k}\|_{\mathbf{L}^2}^{2}\right).}
 \end{array}
 \end{equation}

Hence we get
 \begin{equation}\label{eq:4-44}
 \begin{array}{@{}l@{}}
 {\displaystyle |T_2| \leq \frac{1}{16} D(\overline{\theta}_{\mathbf{A}}^{m},\overline{\theta}_{\mathbf{A}}^{m}) + C\|\theta_{\Psi}^{m}\|^{2}_{\mathcal{L}^2}+C h^{2r} }\\[2mm]
 {\displaystyle \qquad + C\tau \sum_{k=0}^{m}\left(D({\theta}_{\mathbf{A}}^{k},{\theta}_{\mathbf{A}}^{k})
+\|\partial_{\tau} {\theta}_{\mathbf{A}}^{k}\|_{\mathbf{L}^2}^{2}+\|\nabla\theta_{\Psi}^{k}\|_{\mathbf{L}^2}^{2}\right).}
 \end{array}
 \end{equation}

Reasoning as before, we can estimate $ T_3 $ as follows.
 \begin{equation}\label{eq:4-45}
 \begin{array}{@{}l@{}}
 {\displaystyle |T_3|=\big|\mathrm{i}\sum_{k=1}^{m}\big(\nabla \mathcal{I}_{h}\overline{\Psi}^{k}({\bm\pi}_{h}\overline{\mathbf{A}}^{k}
 -\overline{\mathbf{A}}^{k}_{h}),\;\theta_{\Psi}^{k}-\theta_{\Psi}^{k-1}\big)\big|}\\[2mm]
  {\displaystyle \leq \frac{1}{16} D(\overline{\theta}_{\mathbf{A}}^{m},\overline{\theta}_{\mathbf{A}}^{m}) + C\|\theta_{\Psi}^{m}\|^{2}_{\mathcal{L}^2}+C h^{2r} }\\[2mm]
 {\displaystyle \qquad + C\tau \sum_{k=0}^{m}\left(D({\theta}_{\mathbf{A}}^{k},{\theta}_{\mathbf{A}}^{k})
+\|\partial_{\tau} {\theta}_{\mathbf{A}}^{k}\|_{\mathbf{L}^2}^{2}+\|\nabla\theta_{\Psi}^{k}\|_{\mathbf{L}^2}^{2}\right).}
 \end{array}
 \end{equation}

Combining (\ref{eq:4-40}), (\ref{eq:4-44}), and (\ref{eq:4-45}) implies
\begin{equation}\label{eq:4-46}
\begin{array}{@{}l@{}}
{\displaystyle |\sum_{k=1}^{m} V_5^{k,3}| \leq  \frac{3}{16} D(\overline{\theta}_{\mathbf{A}}^{m},\overline{\theta}_{\mathbf{A}}^{m}) + C\|\theta_{\Psi}^{m}\|^{2}_{\mathcal{L}^2}+C h^{2r}}\\[2mm]
 {\displaystyle \quad\quad\quad + C\tau \sum_{k=0}^{m}\left(D({\theta}_{\mathbf{A}}^{k},{\theta}_{\mathbf{A}}^{k})
 +\|\partial_{\tau} {\theta}_{\mathbf{A}}^{k}\|_{\mathbf{L}^2}^{2}
 +\|\nabla\theta_{\Psi}^{k}\|_{\mathbf{L}^2}^{2}\right),}
 \end{array}
 \end{equation}
and thus
 \begin{equation}\label{eq:4-47}
 \begin{array}{@{}l@{}}
 {\displaystyle \displaystyle |\tau\sum_{k=1}^{m}V_{5}^{k}(\partial_{\tau}{\theta_{\Psi}^{k}})|\leq  |\sum_{k=1}^{m} V_5^{k,1}| + |\sum_{k=1}^{m} V_5^{k,2}| + |\sum_{k=1}^{m} V_5^{k,3}|  
   }\\[2mm]
  {\displaystyle  \quad \leq  C\big(h^{2r}+\tau^{4}\big)+\frac{3}{16} D(\overline{\theta}_{\mathbf{A}}^{m},\overline{\theta}_{\mathbf{A}}^{m}) + \frac{1}{32} \|\nabla\theta_{\Psi}^{m}\|_{\mathbf{L}^2}^{2} + C\|\theta_{\Psi}^{m}\|_{\mathcal{L}^2}^{2}             }\\[2mm] 
 {\displaystyle   \qquad  + C\tau \sum_{k=0}^{m}\left(D({\theta}_{\mathbf{A}}^{k},{\theta}_{\mathbf{A}}^{k}) +\|\partial_{\tau} {\theta}_{\mathbf{A}}^{k}\|_{\mathbf{L}^2}^{2}+\|\nabla\theta_{\Psi}^{k}\|_{\mathbf{L}^2}^{2} \right).}\\[2mm]
 \end{array}
 \end{equation}

 Now substituting (\ref{eq:4-19}), (\ref{eq:4-22}), (\ref{eq:4-23-0}), (\ref{eq:4-27}), and (\ref{eq:4-47}) into (\ref{eq:4-16-4}), we have
 \begin{equation}\label{eq:4-48}
\begin{array}{@{}l@{}}
{\displaystyle \frac{1}{2}B(\overline{\mathbf{A}}_{h}^{m};\theta_{\Psi}^{m},\theta_{\Psi}^{m}) + \sum_{k=1}^{m}J_2^{k} \, \leq 
 C \big(h^{2r} +\tau^{4}\big)+ \frac{3}{32} \|\nabla\theta_{\Psi}^{m}\|_{\mathbf{L}^2}^{2} +\frac{3}{16} D(\overline{\theta}_{\mathbf{A}}^{m},\overline{\theta}_{\mathbf{A}}^{m})}\\[2mm]
 {\displaystyle  +C\|\theta_{\Psi}^{m}\|_{\mathcal{L}^2}^{2} +\tau\sum_{k=1}^{m}J_1^{k} +C \tau\sum_{k=0}^{m}\left( \|\nabla \theta_{\Psi}^{k}\|^{2}_{\mathbf{L}^{2}} +\|\nabla \theta_{\phi}^{k}\|^{2}_{\mathbf{L}^{2}}+ D(\theta_{\mathbf{A}}^{k},\theta_{\mathbf{A}}^{k}) + \|\partial_{\tau} \theta_{\mathbf{A}}^{k} \|^{2}_{\mathbf{L}^{2}}\right)  .}
 \end{array}
\end{equation}
Arguing as in the proof of Theorem~\ref{thm6-0}, we have
\begin{equation}\label{eq:4-49}
\frac{9}{32}\Vert\nabla \theta_{\Psi}^{m} \Vert_{\mathbf{L}^{2}}^{2} \leq B(\overline{\mathbf{A}}_{h}^{m};\theta_{\Psi}^{m},\theta_{\Psi}^{m}) + C \Vert \theta_{\Psi}^{m} \Vert_{\mathcal{L}^{2}}^{2}.
\end{equation}
 Consequently, by inserting (\ref{eq:4-49}) into (\ref{eq:4-48}), we find
 \begin{equation}\label{eq:4-50}
 \begin{array}{@{}l@{}}
 {\displaystyle \frac{3}{64}\|\nabla\theta_{\Psi}^{m}\|_{\mathbf{L}^2}^{2}
 + \sum_{k=1}^{m}J_2^{k} \,  \leq 
C \big(h^{2r} +\tau^{4}\big)+\frac{3}{16} D(\overline{\theta}_{\mathbf{A}}^{m},\overline{\theta}_{\mathbf{A}}^{m}) + C\|\theta_{\Psi}^{m}\|_{\mathcal{L}^2}^{2} }\\[2mm]
 {\displaystyle   +\tau\sum_{k=1}^{m}J_1^{k} +C \tau\sum_{k=0}^{m}\left( \|\nabla \theta_{\Psi}^{k}\|^{2}_{\mathbf{L}^{2}} +\|\nabla \theta_{\phi}^{k}\|^{2}_{\mathbf{L}^{2}}+ D(\theta_{\mathbf{A}}^{k},\theta_{\mathbf{A}}^{k}) + \|\partial_{\tau} \theta_{\mathbf{A}}^{k} \|^{2}_{\mathbf{L}^{2}}\right)  .}
 \end{array}
 \end{equation}
By substituting (\ref{eq:4-15}) into (\ref{eq:4-50}), we end up with
 \begin{equation}\label{eq:4-51}
 \begin{array}{@{}l@{}}
 {\displaystyle \frac{3}{64}\|\nabla\theta_{\Psi}^{m}\|_{\mathbf{L}^2}^{2}
 + \sum_{k=1}^{m}J_2^{k} \,  \leq 
C \big(h^{2r} +\tau^{4}\big)+\frac{3}{16} D(\overline{\theta}_{\mathbf{A}}^{m},\overline{\theta}_{\mathbf{A}}^{m}) + \tau\sum_{k=1}^{m}J_1^{k} }\\[2mm]
 {\displaystyle    \qquad + C \tau\sum_{k=0}^{m}\left( \|\nabla \theta_{\Psi}^{k}\|^{2}_{\mathbf{L}^{2}} +\|\nabla \theta_{\phi}^{k}\|^{2}_{\mathbf{L}^{2}}+ D(\theta_{\mathbf{A}}^{k},\theta_{\mathbf{A}}^{k}) + \|\partial_{\tau} \theta_{\mathbf{A}}^{k} \|^{2}_{\mathbf{L}^{2}}\right)  .}
 \end{array}
 \end{equation}
 
 \subsection{\textbf{Estimates for (\ref{eq:4-5})}}
Taking $\mathbf{v}= \frac{\displaystyle 1}{\displaystyle 2\tau}(\theta_{\mathbf{A}}^{k}
-\theta_{\mathbf{A}}^{k-2})$ in (\ref{eq:4-5}), we see that
\begin{equation}\label{eq:4-62-0}
\begin{array}{@{}l@{}}
{\displaystyle \Big(\partial_{\tau}^{2}\theta^{k}_{\mathbf{A}},\;\frac{1}{2}(\partial_{\tau} \theta_{\mathbf{A}}^{k}+\partial_{\tau} \theta_{\mathbf{A}}^{k-1})\Big)
+D\big(\widetilde{\theta^{k}_{\mathbf{A}}},\;\frac{1}{2}(\partial_{\tau} \theta_{\mathbf{A}}^{k}+\partial_{\tau} \theta_{\mathbf{A}}^{k-1})\big)}\\[2mm]
{\displaystyle \quad = \frac{1}{2\tau}\left(\|\partial_{\tau} \theta_{\mathbf{A}}^{k}\|_{\mathbf{L}^2}^{2}
-\|\partial_{\tau} \theta_{\mathbf{A}}^{k-1}\|_{\mathbf{L}^2}^{2}\right)+\frac{1}{4\tau}\left(D({\theta}_{\mathbf{A}}^{k}, {\theta}_{\mathbf{A}}^{k})
-(D({\theta}_{\mathbf{A}}^{k-2}, {\theta}_{\mathbf{A}}^{k-2})\right)}\\[2mm]
{\displaystyle  \quad =  \sum_{i =1}^{5}U_i^{k}(\overline{\partial_{\tau} \theta}_{\mathbf{A}}^{k}),}
\end{array}
\end{equation}
which leads to 
\begin{equation}\label{eq:4-62-1}
\begin{array}{@{}l@{}}
{\displaystyle  \frac{1}{2}\|\partial_{\tau} \theta_{\mathbf{A}}^{m}\|_{\mathbf{L}^2}^{2}
+\frac{1}{4}D({\theta}_{\mathbf{A}}^{m}, {\theta}_{\mathbf{A}}^{m})
+\frac{1}{4}D({\theta}_{\mathbf{A}}^{m-1}, {\theta}_{\mathbf{A}}^{m-1})}\\[2mm]
{\displaystyle \quad= \frac{1}{2}\|\partial_{\tau} \theta_{\mathbf{A}}^{0}\|_{\mathbf{L}^2}^{2}
+ \frac{1}{4}D({\theta}_{\mathbf{A}}^{0}, {\theta}_{\mathbf{A}}^{0})
+\frac{1}{4}D({\theta}_{\mathbf{A}}^{-1}, {\theta}_{\mathbf{A}}^{-1}) + {\tau}\sum_{k=1}^{m} \sum_{i =1}^{5}U_i^{k}(\overline{\partial_{\tau} \theta}_{\mathbf{A}}^{k}) }\\[2mm]
{\displaystyle \quad \leq Ch^{2r} + {\tau}\sum_{k=1}^{m} \sum_{i =1}^{5}U_i^{k}(\overline{\partial_{\tau} \theta}_{\mathbf{A}}^{k}). }\\[2mm]
\end{array}
\end{equation}
Now we estimate  $\sum\limits_{k=1}^{m} U_i^{k}(\overline{\partial_{\tau} \theta}_{\mathbf{A}}^{k})$, $i=1,2,3,4,5$. Under the regularity assumption of $\mathbf{A}$ in (\ref{eq:2-9}), we have
\begin{equation}\label{eq:4-63}
\begin{array}{@{}l@{}}
{\displaystyle \tau \sum_{k=1}^{m} |U_1^{k}(\overline{\partial_{\tau} \theta}_{\mathbf{A}}^{k})| \leq C\left(h^{2r}+\tau^{4}\right)
+ C\tau\sum_{k=1}^{m} \|\overline{\partial_{\tau} \theta}_{\mathbf{A}}^{k}\|_{\mathbf{L}^2}^{2}.}
\end{array}
\end{equation}
Applying (\ref{eq:4-2}), the regularity assumption, and the Young's inequality, we can bound $\tau \sum_{k=1}^{m} U_2^{k}(\overline{\partial_{\tau} \theta}_{\mathbf{A}}^{k})$ as follows.

\begin{equation}\label{eq:4-63-4}
\begin{array}{@{}l@{}}
{\displaystyle  \tau \sum_{k=1}^{m} U_2^{k}(\overline{\partial_{\tau} \theta}_{\mathbf{A}}^{k})\leq C\big(h^{2r} + \tau^{4}\big)
+ C\tau \sum_{k=0}^{m}D\left(\theta_{\mathbf{A}}^{k},\theta_{\mathbf{A}}^{k}\right)}\\[2mm]
{\displaystyle \quad\quad+ \frac{1}{32} D\left(\theta_{\mathbf{A}}^{m},\theta_{\mathbf{A}}^{m}\right)
+\frac{1}{32} D\left(\theta_{\mathbf{A}}^{m-1},\theta_{\mathbf{A}}^{m-1}\right).}
\end{array}
\end{equation}

Since $\overline{\partial_{\tau} \theta}_{\mathbf{A}}^{k} \in \mathbf{X}_{0h}$, we have
\begin{equation}
\big(\frac{1}{2\tau}(I_{h}\phi^{k} - I_{h}\phi^{k-2}),\, \nabla \cdot \overline{\partial_{\tau} \theta}_{\mathbf{A}}^{k}\big) = 0,
\end{equation}
from which we deduce
\begin{equation}
\begin{array}{@{}l@{}}
{\displaystyle U_3^{k}(\overline{\partial_{\tau} \theta}_{\mathbf{A}}^{k}) = -\Big( (\phi_t)^{k-1},\, \nabla \cdot \overline{\partial_{\tau} \theta}_{\mathbf{A}}^{k} \Big)  = -\Big((\phi_t)^{k-1} - \frac{1}{2\tau}(\phi^{k} - \phi^{k-2}) ,\, \nabla \cdot \overline{\partial_{\tau} \theta}_{\mathbf{A}}^{k} \Big) }\\[2mm]
{\displaystyle \qquad\qquad  - \Big(\frac{1}{2\tau}(\phi^{k} - \phi^{k-2}) - \frac{1}{2\tau}(I_{h}\phi^{k} - I_{h}\phi^{k-2})  ,\; \nabla \cdot \overline{\partial_{\tau} \theta}_{\mathbf{A}}^{k} \Big).   }
\end{array}
\end{equation}
By applying (\ref{eq:4-2}) and reasoning as before, we have
\begin{equation}\label{eq:4-63-5}
\begin{array}{@{}l@{}}
{\displaystyle  \tau \sum_{k=1}^{m} U_3^{k}(\overline{\partial_{\tau} \theta}_{\mathbf{A}}^{k})\leq C\big(h^{2r} + \tau^{4}\big)
+ C\tau \sum_{k=0}^{m}D\left(\theta_{\mathbf{A}}^{k},\theta_{\mathbf{A}}^{k}\right)}\\[2mm]
{\displaystyle \quad\quad+ \frac{1}{32} D\left(\theta_{\mathbf{A}}^{m},\theta_{\mathbf{A}}^{m}\right)
+\frac{1}{32} D\left(\theta_{\mathbf{A}}^{m-1},\theta_{\mathbf{A}}^{m-1}\right).}
\end{array}
\end{equation} 
By applying Theorem~\ref{thm3-1} and the regularity assumption, the terms $\tau\sum_{k=1}^{m}\\ U_4^{k}(\overline{\partial_{\tau} \theta}_{\mathbf{A}}^{k})$ can be estimated by a standard argument.
\begin{equation}\label{eq:4-65}
\tau\sum_{k=1}^{m} U_4^{k}(\overline{\partial_{\tau} \theta}_{\mathbf{A}}^{k}) \leq C\big(h^{2r} + \tau^{4}\big) + C\tau \sum_{k=0}^{m}\big(\Vert\nabla \theta_{\Psi}^{k} \Vert_{\mathbf{L}^{2}}^{2} + \Vert \partial_{\tau} \theta_{\mathbf{A}}^{k} \Vert _{\mathbf{L}^{2}}^{2}\big).
\end{equation}

To estimate  $\sum\limits_{k=1}^{m} U_5^{k}(\overline{\partial_{\tau} \theta}_{\mathbf{A}}^{k})$, we first rewrite $U_5^{k}(\overline{\partial_{\tau} \theta}_{\mathbf{A}}^{k})$ as follows.
\begin{equation}\label{eq:4-69}
\begin{array}{@{}l@{}}
{\displaystyle U_5^{k}(\overline{\partial_{\tau} \theta}_{\mathbf{A}}^{k})=\left(f(\Psi^{k-1},\Psi^{k-1})-f(\mathcal{I}_{h}\Psi^{k-1}, \mathcal{I}_{h}\Psi^{k-1}),\;\overline{\partial_{\tau} \theta}_{\mathbf{A}}^{k}\right) }\\[2mm]
{\displaystyle \quad +\left(f(\mathcal{I}_{h}\Psi^{k-1}, \mathcal{I}_{h}\Psi^{k-1})-f(\Psi^{k-1}_{h},\Psi^{k-1}_{h}),\;\overline{\partial_{\tau} \theta}_{\mathbf{A}}^{k}\right) }\\[2mm]
{\displaystyle \quad := U_5^{k,1}(\overline{\partial_{\tau} \theta}_{\mathbf{A}}^{k})+U_5^{k,2}(\overline{\partial_{\tau} \theta}_{\mathbf{A}}^{k}).}
\end{array}
\end{equation}

A simple calculation shows that
\begin{equation}\label{eq:4-70}
\begin{array}{@{}l@{}}
{\displaystyle f(\psi,\psi)-f(\varphi,\varphi)=\frac{\mathrm{i}}{2}(\psi^{\ast}\nabla \psi- \psi \nabla\psi^{\ast})
  -\frac{\mathrm{i}}{2} (\varphi^{\ast} \nabla\varphi-\varphi \nabla\varphi^{\ast}) }\\[2mm]
{\displaystyle = -\frac{\mathrm{i}}{2}\left(\varphi^{\ast}\nabla(\varphi-\psi)-\varphi\nabla(\varphi-\psi)^{\ast}\right)
 +\frac{\mathrm{i}}{2}\left((\varphi-\psi)\nabla\psi^{\ast}-(\varphi-\psi)^{\ast}\nabla\psi\right),}
\end{array}
\end{equation}
and  consequently, we have
\begin{equation}\label{eq:4-71}
\begin{array}{@{}l@{}}
{\displaystyle U_5^{k,1}(\overline{\partial_{\tau} \theta}_{\mathbf{A}}^{k})=\left(f(\Psi^{k-1},\Psi^{k-1})-f(\mathcal{I}_{h}\Psi^{k-1}, \mathcal{I}_{h}\Psi^{k-1}),\;\overline{\partial_{\tau} \theta}_{\mathbf{A}}^{k}\right)}\\[2mm]
{\displaystyle\quad \leq  C\left(h^{2r}+\|\overline{\partial_{\tau} \theta}_{\mathbf{A}}^{k}\|_{\mathbf{L}^2}^{2}\right)}
\end{array}
\end{equation}
by applying (\ref{eq:4-0}) and (\ref{eq:4-1}).

Similarly, by employing (\ref{eq:4-1}), we deduce
\begin{equation}\label{eq:4-72}
\begin{array}{@{}l@{}}
{\displaystyle U_5^{k,2}(\overline{\partial_{\tau} \theta}_{\mathbf{A}}^{k})=\left(f(\mathcal{I}_{h}\Psi^{k-1}, \mathcal{I}_{h}\Psi^{k-1})-f(\Psi^{k-1}_{h},\Psi^{k-1}_{h}),\;\overline{\partial_{\tau} \theta}_{\mathbf{A}}^{k}\right)}\\[2mm]
{\displaystyle \quad =  -\frac{\mathrm{i}}{2}\left((\theta_{\Psi}^{k-1})^{\ast}\nabla\theta_{\Psi}^{k-1}
-\theta_{\Psi}^{k-1}\nabla(\theta_{\Psi}^{k-1})^{\ast},\;\overline{\partial_{\tau} \theta}_{\mathbf{A}}^{k}\right) }\\[2mm]
{\displaystyle \quad\quad-\frac{\mathrm{i}}{2}\left((\mathcal{I}_{h}\Psi^{k-1})^{\ast}\nabla\theta_{\Psi}^{k-1}
-\mathcal{I}_{h}\Psi^{k-1}\nabla(\theta_{\Psi}^{k-1})^{\ast},\;\overline{\partial_{\tau} \theta}_{\mathbf{A}}^{k}\right) }\\[2mm]
{\displaystyle\quad \quad +\frac{\mathrm{i}}{2}\left(\theta_{\Psi}^{k-1}\nabla (\mathcal{I}_{h}\Psi^{k-1})^{\ast}-(\theta_{\Psi}^{k-1})^{\ast}\nabla \mathcal{I}_{h}\Psi^{k-1},\;\overline{\partial_{\tau} \theta}_{\mathbf{A}}^{k}\right)}\\[2mm]
{\displaystyle \quad \leq -\left(f(\theta_{\Psi}^{k-1},\theta_{\Psi}^{k-1}),\;\overline{\partial_{\tau} \theta}_{\mathbf{A}}^{k}\right)
+ C \| \mathcal{I}_{h}\Psi^{k-1}\|_{\mathcal{L}^{\infty}}\|\nabla\theta_{\Psi}^{k-1}\|_{\mathbf{L}^2}\|\overline{\partial_{\tau} \theta}_{\mathbf{A}}^{k}\|_{\mathbf{L}^2}}\\[2mm]
{\displaystyle \quad\quad+C\|\nabla \mathcal{I}_{h}\Psi^{k-1}\|_{\mathbf{L}^{3}}\|\theta_{\Psi}^{k-1}\|_{\mathcal{L}^6}
\|\overline{\partial_{\tau} \theta}_{\mathbf{A}}^{k}\|_{\mathbf{L}^2}}\\[2mm]
{\displaystyle \quad \leq -\left(f(\theta_{\Psi}^{k-1},\theta_{\Psi}^{k-1}),\;\overline{\partial_{\tau} \theta}_{\mathbf{A}}^{k}\right)
+C\left(\|\nabla\theta_{\Psi}^{k-1}\|_{\mathbf{L}^2}^{2}
+\|\overline{\partial_{\tau} \theta}_{\mathbf{A}}^{k}\|_{\mathbf{L}^2}^{2}\right).}
\end{array}
\end{equation}

Therefore,  we have 
\begin{equation}\label{eq:4-73}
\begin{array}{@{}l@{}}
{\displaystyle \tau\sum_{k=1}^{m} U_5^{k}(\overline{\partial_{\tau} \theta}_{\mathbf{A}}^{k})\leq C h^{2r} -\tau\sum_{k=1}^{m}\left(f(\theta_{\Psi}^{k-1},\theta_{\Psi}^{k-1}),\;\overline{\partial_{\tau} \theta}_{\mathbf{A}}^{k}\right) }\\[2mm]
{\displaystyle \qquad +C\tau\sum_{k=0}^{m}\left(\|\nabla\theta_{\Psi}^{k}\|_{\mathbf{L}^2}^{2}
+\|\partial_{\tau} \theta_{\mathbf{A}}^{k}\|_{\mathbf{L}^2}^{2}\right).}
\end{array}
\end{equation}

Substituting (\ref{eq:4-63}), (\ref{eq:4-63-4}), (\ref{eq:4-63-5}), (\ref{eq:4-65}), and (\ref{eq:4-73}) into (\ref{eq:4-62-1}) and recalling the definition of $J_1^{k}$ in (\ref{eq:4-16-3}),  we obtain
\begin{equation}\label{eq:4-74}
\begin{array}{@{}l@{}}
{\displaystyle  \frac{1}{2}\|\partial_{\tau} \theta_{\mathbf{A}}^{m}\|_{\mathbf{L}^2}^{2}
+\frac{3}{16}D({\theta}_{\mathbf{A}}^{m}, {\theta}_{\mathbf{A}}^{m})
+\frac{3}{16}D({\theta}_{\mathbf{A}}^{m-1}, {\theta}_{\mathbf{A}}^{m-1})}\\[2mm]
{\displaystyle \quad \leq C\big(h^{2r}+\tau^{4} \big)
+ C\tau\sum_{k=0}^{m}\left(D({\theta}_{\mathbf{A}}^{k}, {\theta}_{\mathbf{A}}^{k})
+ \|\partial_{\tau} \theta_{\mathbf{A}}^{k}\|_{\mathbf{L}^2}^{2} + \|\nabla\theta_{\Psi}^{k}\|_{\mathbf{L}^2}^{2}\right) -\tau\sum_{k=1}^{m}J_1^{k}.}\\[2mm]
\end{array}
\end{equation}

 \subsection{\textbf{Estimates for (\ref{eq:4-6})}}
We can deduce the estimate of  $\nabla {\theta}^{k}_{\phi}$ by a standard argument.
\begin{equation}\label{eq:4-74-0}
\Vert \nabla {\theta}^{k}_{\phi} \Vert_{\mathbf{L}^{2}}^{2}\leq Ch^{2r} + C\Vert \nabla {\theta}^{k}_{\Psi} \Vert_{\mathbf{L}^{2}}^{2},\quad k=0,1,\cdots,m.
\end{equation}

From (\ref{eq:4-6}) we know that
\begin{equation}\label{eq:4-75}
\big( \partial_{\tau}\nabla {\theta}^{k}_{\phi}, \,\nabla u\big) = \big( \partial_{\tau}\nabla e_{\phi}^{k}, \,\nabla u\big) + \frac{1}{\tau}\big(\vert\Psi^{k}_{h}\vert^{2}-\vert\Psi^{k}\vert^{2} - \vert\Psi^{k-1}_{h}\vert^{2}+\vert\Psi^{k-1}\vert^{2},\, u \big), \; \forall u \in X_{h}^{r}.
\end{equation}
By taking $u = \overline{\theta}_{\phi}^{k}$ in (\ref{eq:4-75}) and recalling the definition of $J_2^{k}$ in (\ref{eq:4-26-0}), we find
\begin{equation}
\begin{array}{@{}l@{}}
{\displaystyle \frac{1}{2\tau}\big(\Vert \nabla {\theta}_{\phi}^{k} \Vert_{\mathbf{L}^{2}}^{2} -  \Vert \nabla {\theta}_{\phi}^{k-1} \Vert_{\mathbf{L}^{2}}^{2}\big) = \big( \partial_{\tau}\nabla e_{\phi}^{k}, \,\nabla \overline{\theta}_{\phi}^{k}\big) +\frac{1}{\tau} J_2^{k}   } \\[2mm]
{\displaystyle\qquad \qquad  + \frac{1}{\tau}\big(\vert\mathcal{I}_{h}\Psi^{k}\vert^{2}-\vert\Psi^{k}\vert^{2} - \vert\mathcal{I}_{h}\Psi^{k-1}\vert^{2}+\vert\Psi^{k-1}\vert^{2},\, \overline{\theta}_{\phi}^{k} \big),   }
\end{array}
\end{equation}
which implies that
\begin{equation}
\begin{array}{@{}l@{}}
{\displaystyle \frac{1}{2}\Vert \nabla {\theta}_{\phi}^{m} \Vert_{\mathbf{L}^{2}}^{2} = \frac{1}{2}\Vert \nabla {\theta}_{\phi}^{0} \Vert_{\mathbf{L}^{2}}^{2} + \tau \sum_{k=1}^{m} \big( \partial_{\tau}\nabla e_{\phi}^{k}, \,\nabla \overline{\theta}_{\phi}^{k}\big) + \sum_{k=1}^{m}J_2^{k}  }\\[2mm]
{\displaystyle \qquad \quad + \sum_{k=1}^{m} \big(\vert\mathcal{I}_{h}\Psi^{k}\vert^{2}-\vert\Psi^{k}\vert^{2} - \vert\mathcal{I}_{h}\Psi^{k-1}\vert^{2}+\vert\Psi^{k-1}\vert^{2},\, \overline{\theta}_{\phi}^{k} \big).  }
\end{array}
\end{equation}
Employing the error estimates of interpolation operators and the regularity assumption, we obtain
\begin{equation}
\begin{array}{@{}l@{}}
{\displaystyle \tau \sum_{k=1}^{m} \big( \partial_{\tau}\nabla e_{\phi}^{k}, \,\nabla \overline{\theta}_{\phi}^{k}\big)  \leq Ch^{2r} + C\tau\sum_{k=0}^{m} \Vert \nabla {\theta}_{\phi}^{k} \Vert_{\mathbf{L}^{2}}^{2} ,    }\\[2mm]
{\displaystyle  \sum_{k=1}^{m} \big(\vert\mathcal{I}_{h}\Psi^{k}\vert^{2}-\vert\Psi^{k}\vert^{2} - \vert\mathcal{I}_{h}\Psi^{k-1}\vert^{2}+\vert\Psi^{k-1}\vert^{2},\, \overline{\theta}_{\phi}^{k} \big)  \leq Ch^{2r} + C\tau\sum_{k=0}^{m} \Vert \nabla {\theta}_{\phi}^{k} \Vert_{\mathbf{L}^{2}}^{2}. }
\end{array}
\end{equation}
It follows that 
\begin{equation}\label{eq:4-76}
\frac{1}{2}\Vert \nabla {\theta}_{\phi}^{m} \Vert_{\mathbf{L}^{2}}^{2}  \leq Ch^{2r} + C\tau\sum_{k=0}^{m} \Vert \nabla {\theta}_{\phi}^{k} \Vert_{\mathbf{L}^{2}}^{2} + \sum_{k=1}^{m}J_2^{k}.
\end{equation}
By combining (\ref{eq:4-51}), (\ref{eq:4-74}), and (\ref{eq:4-76}), we finally obtain
\begin{equation}\label{eq:4-77}
\begin{array}{@{}l@{}}
{\displaystyle \frac{3}{64}\|\nabla\theta_{\Psi}^{m}\|_{\mathbf{L}^2}^{2} + \frac{1}{2}\Vert \nabla {\theta}_{\phi}^{m} \Vert_{\mathbf{L}^{2}}^{2} + \frac{1}{2}\|\partial_{\tau} \theta_{\mathbf{A}}^{m}\|_{\mathbf{L}^2}^{2}
+\frac{3}{32}D({\theta}_{\mathbf{A}}^{m}, {\theta}_{\mathbf{A}}^{m})
+\frac{3}{32}D({\theta}_{\mathbf{A}}^{m-1}, {\theta}_{\mathbf{A}}^{m-1}) } \\[2mm]
{\displaystyle \quad \leq C\big(h^{2r} +\tau^{4}\big) + C \tau\sum_{k=0}^{m}\left( \|\nabla \theta_{\Psi}^{k}\|^{2}_{\mathbf{L}^{2}} +\|\nabla \theta_{\phi}^{k}\|^{2}_{\mathbf{L}^{2}}+ D(\theta_{\mathbf{A}}^{k},\theta_{\mathbf{A}}^{k}) + \|\partial_{\tau} \theta_{\mathbf{A}}^{k} \|^{2}_{\mathbf{L}^{2}}\right),  }
\end{array}
\end{equation}
which yields the desired estimate (\ref{eq:4-2-0}) by using the discrete Gronwall's inequality.

\section{numerical experiments}\label{sec-7}
In this section, we present two numerical examples to confirm our theoretical analysis.
\begin{example}\label{exam7-1}
To verify the conservation of total charge and energy of our scheme, we consider the M-S-C system (\ref{eq:1.8})-(\ref{eq:1.9}) with the initial datas:
\begin{equation*}
\begin{array}{@{}l@{}}
{\displaystyle \Psi(\mathbf{x},0) = 2\sin(\pi x_1)\sin(\pi x_2)\sin(\pi x_3) + 2\sin(2\pi x_1)\sin(2\pi x_2)\sin(2\pi x_3), }\\[2mm]
{\mathbf{A}(\mathbf{x},0)=\mathbf{A}_0(\mathbf{x})=\big(5\sin(2\pi x_3)(1 - \cos(2\pi x_1))\sin(\pi x_2), \;\;0,  }\\[2mm]
{\displaystyle\qquad \qquad \quad  5\sin(2\pi x_1)(1 - \cos(2\pi x_3))\sin(\pi x_2)\big), }\\[2mm]
{\displaystyle \mathbf{A}_{t}(\mathbf{x},0)=\mathbf{A}_1(\mathbf{x})=0.}
\end{array}
\end{equation*}
\end{example}
In this example we take $\Omega=(0,1)^{3} $, $ V(\mathbf{x}) = 5.0 $, $ T = 4.0$, and the time step $\tau = 0.01$.

The domain is partitioned into uniform tetrahedrals with $51$ nodes in each direction and $6\times 50^{3}$ elements in total. We solve the M-S-C system by the scheme (\ref{eq:2-6-0}) using mixed finite element method with $r = 2$.

The evolutions of the total charge and energy of the discrete system are displayed in Fig.1, which clearly show that our algorithm almost exactly keeps the conservation of the total charge and energy of the system.

\begin{figure}
\begin{center}
\tiny{a}\includegraphics[width=5cm,height=5cm] {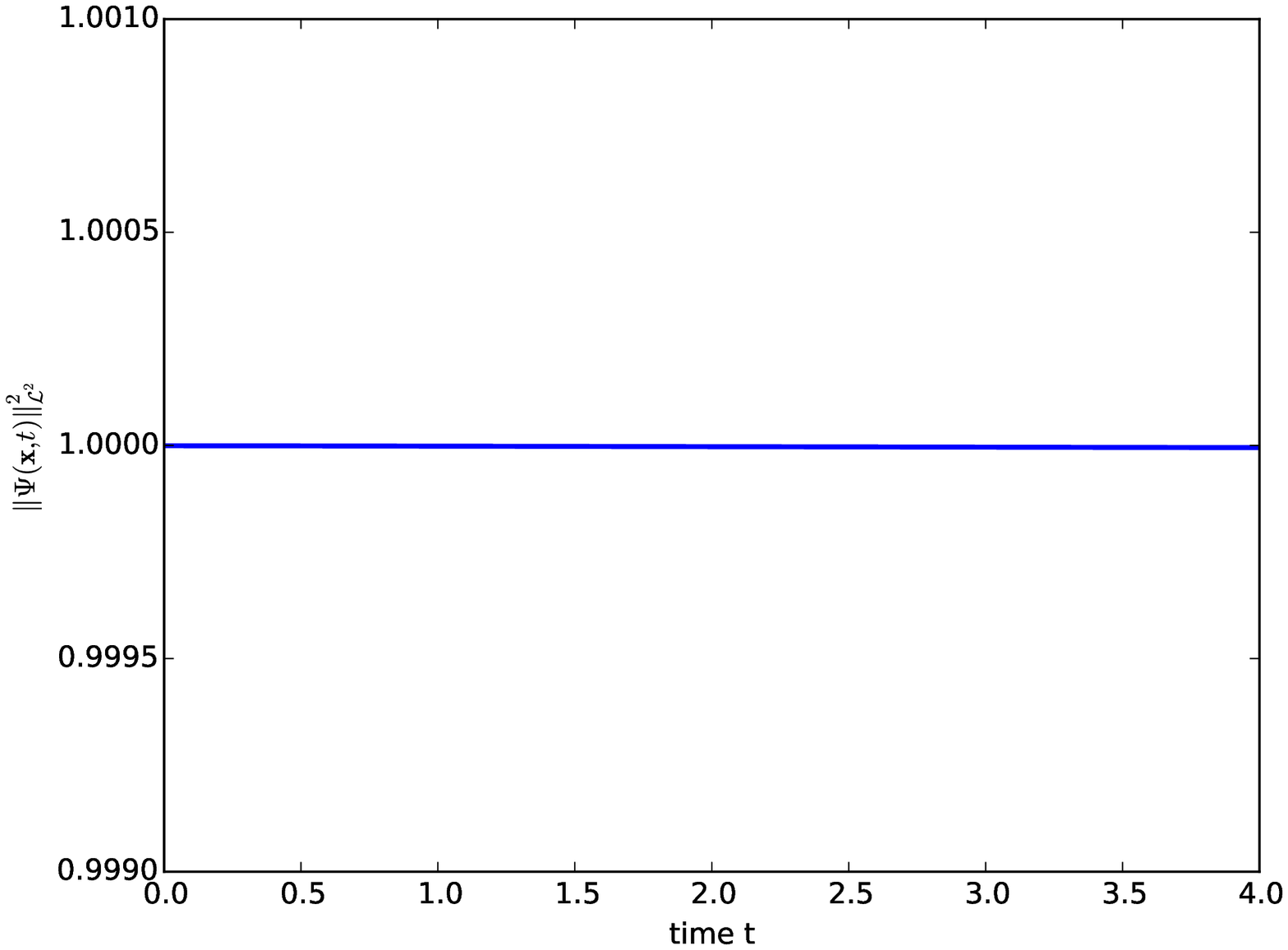}
\tiny{b}\includegraphics[width=5cm,height=5cm] {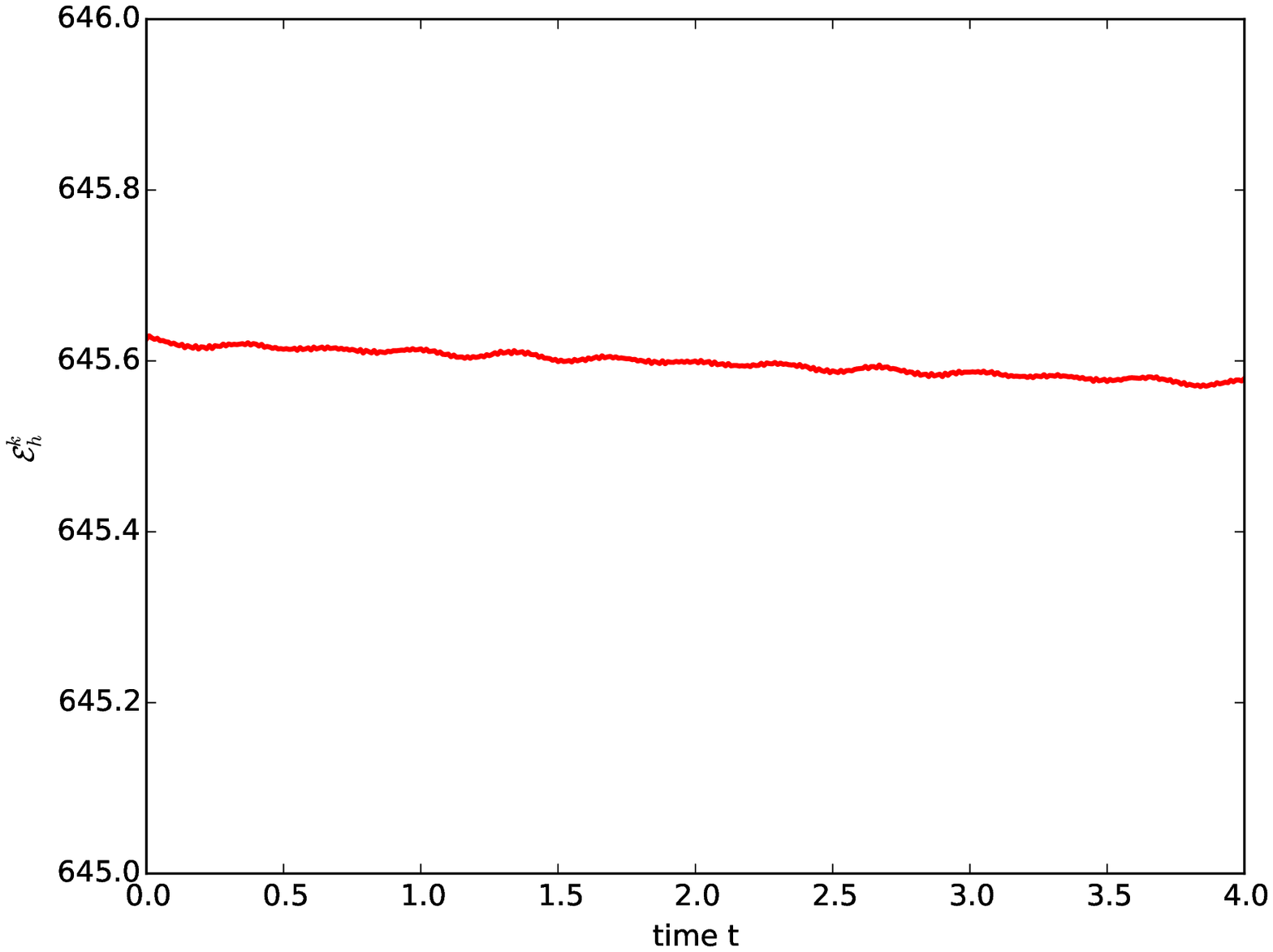}
\caption{Example~\ref{exam7-1} : (a) The evolution of the total charge $\Vert \Psi_{h}^{k}\Vert_{\mathcal{L}^{2}}^{2}$ ; (b) The evolution of the total energy $\mathcal{E}^{k}_{h}$ of the discrete system. 
}
\end{center}
\end{figure}\label{fig:7-1}

\begin{example}\label{exam7-2}
We consider the following M-S-C system:
\begin{equation}\label{eq:7-1}
\left\{
\begin{array}{@{}l@{}}
{\displaystyle  -\mathrm{i}\frac{\partial \Psi}{\partial t}+
\frac{1}{2}\left(\mathrm{i}\nabla +\mathbf{A}\right)^{2}\Psi
 + V\Psi + \phi\Psi= g(\mathbf{x},t) ,\,\, (\mathbf{x},t)\in
\Omega\times(0,T),}\\[2mm]
{\displaystyle \frac{\partial ^{2}\mathbf{A}}{\partial t^{2}}+\nabla\times
(\nabla\times \mathbf{A}) +\frac{\partial (\nabla \phi)}{\partial t}+\frac{\mathrm{i}}{2}\big(\Psi^{*}\nabla{\Psi}-\Psi\nabla{\Psi}^{*}\big) }\\[2mm]
{\displaystyle \qquad\qquad+\vert\Psi\vert^{2}\mathbf{A}=\mathbf{f}(\mathbf{x},t),
\,\,\quad (\mathbf{x},t)\in \Omega\times(0,T),}\\[2mm]
{\displaystyle  \nabla \cdot \mathbf{A} =0, \quad -\Delta \phi - \vert\Psi\vert^{2} = h(\mathbf{x},t),\,\, (\mathbf{x},t)\in \Omega\times(0,T) } \\[2mm]
{\displaystyle \Psi(\mathbf{x},t)=0,\quad \mathbf{A}(\mathbf{x},t)\times\mathbf{n}=0, \quad \phi(\mathbf{x},t) = 0, \,\, (\mathbf{x},t)\in \partial \Omega\times(0,T).}
\end{array}
\right.
\end{equation}
with the exact solution $(\Psi,\mathbf{A},\phi)$ 
\begin{equation*}
\begin{array}{@{}l@{}}
{\displaystyle  \Psi(\mathbf{x}, t) =  5.0e^{\mathrm{i}\pi t}  \sin(2\pi x_1)\sin(2\pi x_2)\sin(2\pi x_3), }\\[2mm]
{\displaystyle \mathbf{A}(\mathbf{x},t)=\sin(\pi t)\Big(\cos(\pi x_1)\sin(\pi x_2)\sin(\pi x_3),\,
\sin(\pi x_1)\cos(\pi x_2)\sin(\pi x_3),}\\[2mm]
{\displaystyle \quad -2\sin(\pi x_1)\sin(\pi x_2)\cos(\pi x_3)\Big)
+ \cos(\pi t)\Big(\sin(2\pi x_3)(1 - \cos(2\pi x_1))\sin(\pi x_2),}\\[2mm]
{\displaystyle \qquad 0, \,\,\sin(2\pi x_1)(1 - \cos(2\pi x_3))\sin(\pi x_2)\Big),}\\[2mm]
{\displaystyle \phi(\mathbf{x}, t) = 4\sin(\pi t) x_1 x_2 x_3(1-x_1)(1-x_2)(1-x_3) }\\[2mm]
{\displaystyle \qquad \qquad  + \;\;cos(\pi t)\sin(\pi x_1)\sin(\pi x_2)\sin(\pi x_3).   }
\end{array}
\end{equation*}
where $V(\mathbf{x})= \frac{1}{2}(x_1^{2} + x_2^{2} + x_3^{2})$ and the right-hand side functions $g(\mathbf{x},t)$, $\mathbf{f}(\mathbf{x} ,t)$ and $h(\mathbf{x},t)$ are determined by the exact solution.
\end{example}

In this example, we set $\Omega = (0,1)^3$ and $T = 2$.
We take a uniform tetrahedral partition with $N+1$ nodes in each direction and $6N^{3}$ elements in total as in the example~\ref{exam7-1}.
The M-S-C system (\ref{eq:7-1}) are solved by the proposed scheme (\ref{eq:2-6-0}) with $r=1,2$, which are denoted by linear element method and quadratic element method, respectively. To test the convergence order of our scheme, we pick $\tau = h^{\frac{1}{2}}$ for the linear element method and $\tau = h$ for the quadratic element method respectively. We present numerical results for the linear element method and the quadratic element method at the final time $T = 2.0$ in Tables~\ref{table7-1} and~\ref{table7-2}, respectively.  We  can clearly see that  the convergence rate of the quadratic element method agrees with our theoretical analysis while the linear element method has better convergence order than our theoretical analysis, which is partially because we use a quadratic element approximation of the vector potential $\mathbf{A}$.

\begin{table}[htb]
\caption{$H^1$ error of linear FEM with $h = \frac{1}{N}$ and $\tau = h^{\frac12}$.}\label{table7-1}
\begin{center}
\begin{tabular}{c|c|c|c}
   \hline\hline
     &$\Vert \Psi_{h}^{M} - \Psi(\cdot, 2)\Vert_{\mathcal{H}^{1}}$ & $\Vert \mathbf{A}^{M}_{h} - \mathbf{A}(\cdot, 2) \Vert _{\mathbf{H}^{1}} $ & $\Vert \phi_{h}^{M} - \phi(\cdot, 2)\Vert_{{H}^{1}}$\\
   \hline
    N=25 &3.3481e-01  & 2.8513e-01 & 4.0193e-02  \\
    
    N=50 & 1.5649e-01 & 1.1960e-01 & 1.6210e-02  \\
     
     N=100 & 6.9587e-02 & 4.5826e-02 & 7.0549e-03  \\
    \hline
     order & 1.13 & 1.33 & 1.25 \\
      \hline\hline
  \end{tabular}
 \end{center}
\end{table}

\begin{table}[htb]
\caption{$H^1$ error of quadratic FEM with $h = \frac{1}{N}$ and $\tau = h$.}\label{table7-2}
\begin{center}
\begin{tabular}{c|c|c|c}
   \hline\hline
     &$\Vert \Psi_{h}^{M} - \Psi(\cdot, 2)\Vert_{\mathcal{H}^{1}}$ & $\Vert \mathbf{A}^{M}_{h} - \mathbf{A}(\cdot, 2) \Vert _{\mathbf{H}^{1}} $ & $\Vert \phi_{h}^{M} - \phi(\cdot, 2)\Vert_{{H}^{1}}$\\
   \hline
    N=25 & 2.3605e-02 & 1.5911e-02 & 6.8377e-03  \\
    
    N=50 & 6.1967e-03 & 4.1329e-03 & 1.6405e-03  \\
     
     N=100 & 1.5204e-03 & 9.4441e-04 &3.9781e-04 \\
    \hline
     order &1.97 & 2.04 & 2.05 \\
      \hline\hline
  \end{tabular}
 \end{center}
\end{table}

%\begin{acknowledgements}
%If you'd like to thank anyone, place your comments here
%and remove the percent signs.
%\end{acknowledgements}

% Non-BibTeX users please use

\end{document}